\documentclass[reqno,12pt]{preprint}
\usepackage[marginparwidth=1in]{geometry}
\usepackage[full]{textcomp}
\usepackage[osf]{newtxtext}
\usepackage{cabin}
\usepackage{enumerate}
\usepackage{enumitem}
\usepackage{bbm}

\usepackage[varqu,varl]{inconsolata}
\usepackage[cal=boondoxo]{mathalfa}

\usepackage{comment}
\usepackage{hyperref}
\usepackage{breakurl}
\usepackage{mathrsfs}
\usepackage{booktabs}
\usepackage{caption}
\usepackage{stmaryrd}

\usepackage{tikz}
\usepackage{amsmath} 
\usepackage{mhequ} 
\usepackage{mhsymb} 
\usepackage{mhenvs} 
\usepackage{microtype}
\usepackage{amssymb} 
\usepackage{eufrak}

\usepackage{wasysym}
\usepackage{centernot}

\newcommand{\Exp}{\mathbb{E}}

\newcommand{\T}{\mathbb{T}}
\newcommand{\D}{\mathbb{D}}
\newcommand{\Zint}{\mathbb{Z}}

\newcommand{\Prob}{\mathbb{P}}
\newcommand{\1}{\mathbbm{1}}
\newcommand{\Ch}{\mathbb{C}}

\newcommand{\cA}{\mathcal{A}}

\newcommand{\cD}{\mathcal{D}}

\newcommand{\cF}{\mathcal{F}}

\newcommand{\cK}{\mathcal{K}}
\newcommand{\cM}{\mathcal{M}}
\newcommand{\cN}{\mathcal{N}}
\newcommand{\cP}{\mathcal{P}}
\newcommand{\cW}{\mathcal{W}}

\newcommand{\sK}{\mathfrak{K}}
\newcommand{\sC}{\mathfrak{C}}
\newcommand{\sE}{\mathfrak{E}}
\newcommand{\sH}{\mathfrak{H}}
\newcommand{\sT}{\mathfrak{T}}

\newcommand{\fb}{\mathfrak{b}}
\newcommand{\sh}{\mathfrak{h}}
\newcommand{\s}{\mathfrak{s}}
\newcommand{\ft}{\mathfrak{t}}
\newcommand{\sv}{\mathfrak{v}}
\newcommand{\sw}{\mathfrak{w}}
\newcommand{\sy}{\mathfrak{y}}
\newcommand{\sz}{\mathfrak{z}}
\newcommand{\sx}{\mathfrak{x}}

\newcommand{\rl}{\mathrm{l}}
\newcommand{\rp}{\mathrm{p}}

\newcommand{\rr}{\mathrm{r}}

\newcommand{\bc}{\mathrm{bc}}
\newcommand{\bd}{{0\text{-}\mathrm{bd}}}

\newcommand{\bw}{\mathrm{bw}}
\newcommand{\com}{\mathrm{c}}

\newcommand{\dd}{\mathrm{d}}      
  
\newcommand{\dis}{\mathop{\mathrm{dis}}}
\newcommand{\Disc}{\mathop{\mathrm{Disc}}}
\newcommand{\eval}{\mathrm{eval}}
\newcommand{\Gau}{\mathrm{Gau}}
\newcommand{\Id}{\mathrm{Id}}

\newcommand{\M}{\mathrm{M1}}

\newcommand{\per}{\mathrm{per}}

\newcommand{\Poi}{\mathrm{Poi}}

\newcommand{\Sk}{\mathrm{Sk}}
\newcommand{\TV}{\mathrm{TV}}
\newcommand{\st}{\mathrm{s}}
\newcommand{\Sp}{\mathrm{sp}}
\newcommand{\bsp}{\mathrm{bsp}}

\newcommand{\uDelta}{\boldsymbol{\Delta}}

\newcommand{\ST}{\mathscr{T}}

\def\restr{\mathord{\upharpoonright}}
\def\c{\mathfrak{c}}

\renewcommand{\R}{\mathbb{R}}
\renewcommand{\C}{\mathbb{C}}
\renewcommand{\N}{\mathbb{N}}
\renewcommand{\Q}{\mathbb{Q}}
\renewcommand{\T}{\mathbb{T}}
\renewcommand{\Z}{\mathbb{Z}}
\renewcommand{\E}{\mathbb{E}}
\renewcommand{\P}{\mathbb{P}}

\def\Cauchy{\mathop{\mathrm{Cauchy}}}
\def\Levy{\mathop{\mathrm{Levy}}}

\newcommand{\da}{\downarrow}
\newcommand{\ua}{\uparrow}
\newcommand{\uda}{\downarrow\mathrel{\mspace{-1mu}}\uparrow}
\newcommand{\dotp}{\bigcdot}
\newcommand{\llb}{\llbracket}
\newcommand{\rrb}{\rrbracket}

\newcommand{\pid}{\pi^\da}

\newcommand{\pidd}{\pi^{\da,\delta}}
\newcommand{\piud}{\pi^{\ua,\delta}}

\def\f#1#2{\textstyle{#1\over #2}}

\colorlet{darkblue}{blue!90!black}
\colorlet{darkred}{red!90!black}
\colorlet{dr}{red!90!black}

\def\EW{{\mathrm{EW}}}
\def\KPZ{{\mathrm{KPZ}}}
\def\Law{\mathop{\mathrm{Law}}}
\def\BC{{\mathrm{BC}}}
\def\law{\mathrm{\tiny law}}
\newcommand{\eqlaw}{\stackrel{\law}{=}}
\mathchardef\mhyphen="2D
\def\pvar{p\mathrm{\mhyphen var}}

\tikzset{
        dot/.style={thin,circle,fill=gray,draw=black,inner sep=0pt,minimum size=2mm},
        vertex/.style={thin,circle,fill=black,draw=black,inner sep=0pt,minimum size=1mm},
	}

\makeatletter 
\newcommand*{\bigcdot}{}
\DeclareRobustCommand*{\bigcdot}{%
  \mathbin{\mathpalette\bigcdot@{}}%
}
\newcommand*{\bigcdot@scalefactor}{.5}
\newcommand*{\bigcdot@widthfactor}{1.15}
\newcommand*{\bigcdot@}[2]{%
  \sbox0{$#1\vcenter{}$}
  \sbox2{$#1\cdot\m@th$}%
  \hbox to \bigcdot@widthfactor\wd2{%
    \hfil
    \raise\ht0\hbox{%
      \scalebox{\bigcdot@scalefactor}{%
        \lower\ht0\hbox{$#1\bullet\m@th$}%
      }%
    }%
    \hfil
  }%
}
\makeatother

\def\dash{\leavevmode\unskip\kern0.18em--\penalty\exhyphenpenalty\kern0.18em}
\def\slash{\leavevmode\unskip\kern0.15em/\penalty\exhyphenpenalty\kern0.15em}

\begin{document}

\title{The Brownian Castle}
\author{G. Cannizzaro$^{1,2}$ and M. Hairer$^1$}

\institute{Imperial College London, SW7 2AZ, UK  \and University of Warwick, CV4 7AL, UK\\
 \email{giuseppe.cannizzaro@warwick.ac.uk, m.hairer@imperial.ac.uk}}

\maketitle

\begin{abstract}
We introduce a $1+1$-dimensional temperature-dependent model such that the classical ballistic
deposition model is recovered as its zero-temperature limit. 
Its $\infty$-temperature version, which we refer to as the $0$-Ballistic Deposition ($0$-BD) model, 
is a randomly evolving interface which, surprisingly enough, does {\it not} belong to either 
the Edwards--Wilkinson (EW) or the Kardar--Parisi--Zhang (KPZ) universality class.
We show that $0$-BD has a scaling limit, 
a new stochastic process that we call {\it Brownian Castle} (BC) which, although it 
is ``free'', is distinct from EW and,
like any other renormalisation fixed point, is scale-invariant, in this case  
under the $1:1:2$ scaling (as opposed 
to $1:2:3$ for KPZ and $1:2:4$ for EW). 
In the present article, we not only derive its finite-dimensional 
distributions, but also provide a ``global'' construction of the Brownian Castle
which has the advantage of highlighting the fact that it admits backward characteristics given by 
the (backward) Brownian Web (see~\cite{TW,FINR}). 
Among others, this characterisation enables us to establish fine pathwise properties of BC and 
to relate these to special points of the Web. We prove that the Brownian Castle is 
a (strong) Markov and Feller process on a suitable space of c\`adl\`ag functions 
and determine its long-time behaviour. 
At last, we give a glimpse to its universality by proving the convergence of $0$-BD to BC in a rather strong sense.

%
%
%
%
%
%
%
%
%
%
%
%
%
%
%
%
%
%
%
\end{abstract}

\setcounter{tocdepth}{2}       
\tableofcontents

\section{Introduction}

The starting point for the investigation presented in this article is the one-dimensional 
Ballistic Deposition (BD) model \cite{BDOriginal}
and the long-standing open question concerning its large-scale behaviour. 
Ballistic Deposition (whose precise definition will be given below) 
is an example of random interface in $(1+1)$-dimensions, i.e.\ a 
map $h\colon\R_+\times A\to\R$, $A$ being a subset of $\R$, whose evolution is driven by a stochastic forcing. 
In this context, two universal behaviours, so-called universality classes, 
have generally been considered, the Kardar--Parisi--Zhang (KPZ) class, to which BD is 
presumed to belong \cite{MR1600794,Jeremy}, and the Edwards--Wilkinson (EW) class.
Originally introduced in~\cite{KPZOrig}, the first is conjectured to capture the large-scale fluctuations of all those models exhibiting some smoothing mechanism, slope-dependent growth speed, 
and short range randomness. 
The ``strong KPZ universality conjecture'' states that for height functions $h$ in this loosely defined class, 
the limit as $\delta\to0$ of 
$\delta^{1} h(\cdot/\delta^{3}, \cdot/\delta^2)-C/\delta^2$, where $C$ is a model dependent constant, exists 
(meaning in particular that the scaling exponents 
of the KPZ class are $1:2:3$), and is given by $h_\KPZ$, 
a universal (model-independent) stochastic process referred to as the ``KPZ fixed point'' 
(see~\cite{KPZfp} for the recent construction of this process as the scaling limit of TASEP). 
%
%
%
%
%
%
%
%
%
%
%
%
%
%
If an interface model satisfies these features but does not display any slope-dependence, 
then it is conjectured to belong to the EW universality class~\cite{EW}, whose scaling exponents are $1:2:4$ 
and whose universal fluctuations are Gaussian, given by the solutions $h_\EW$ to the (additive) 
stochastic heat equation 
\begin{equ}[e:EW]
\d_t h_\EW = {1\over 2} \d_x^2 h_\EW + \xi\;,
\end{equ}
with $\xi$ denoting space-time white noise\footnote{The choice of constants $1/2$ and $1$ appearing in \eqref{e:EW}
is no loss of generality as it can be enforced by a simple fixed rescaling.}.

That said, there is a paradigmatic model in the KPZ universality class which plays a distinguished role. 
This model is a singular stochastic PDE, the KPZ equation, which can be formally written as 
\begin{equ}[e:KPZ]
\partial_t h =\frac12\d_x^2 h+ \frac14(\partial_x h)^2 + \xi\,.
\end{equ}
(The proof that it does indeed converge to the KPZ fixed point under the KPZ scaling
was recently given in \cite{QS1,Virag}.)
The importance of~\eqref{e:KPZ} lies in the fact that its solution is expected to be universal itself in view of 
the so-called ``weak KPZ universality conjecture'' \cite{MR1462228,KPZJeremy} 
which, loosely speaking, can be stated as follows. 
Consider any (suitably parametrised) continuous
one-parameter family $\eps \mapsto h_\eps$ of interface growth models with the following properties
\begin{itemize}[noitemsep, label=-]
\item the model $h_0$ belongs to the EW universality class,
\item for $\eps > 0$, the model $h_\eps$ belongs to the KPZ universality class. 
\end{itemize}
Then it is expected that there exists a choice of constants $C_\eps$ such that
\begin{equ}[e:weakKPZ]
\lim_{\eps \to 0} \eps^{1/2} h_\eps(\eps^{-2} t, \eps^{-1} x) - C_\eps  =  h\;,
\end{equ}
with $h$ solving~\eqref{e:KPZ}.

One can turn this conjecture on its head and take a result of the form \eqref{e:weakKPZ}
as suggestive evidence that the models $h_\eps$ are indeed in the KPZ universality class for
fixed $\eps$.
What we originally intended to do with the Ballistic Deposition model was exactly this: 
introduce a one-parameter family 
of interface models to which BD belongs and prove a limit of the type~\eqref{e:weakKPZ}. 
\medskip
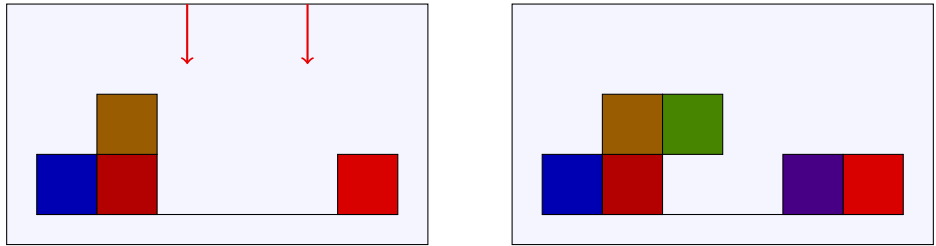
\begin{figure}[ht]
\begin{center}
\begin{tikzpicture}[scale=0.8]
\draw[thin,fill=blue!4] (-0.5,-0.5) rectangle (6.5,3.5);

\draw (0,0) -- (6,0);
\draw[black,fill=blue!70!black] (0,0) rectangle ++(1,1);
\draw[black,fill=red!70!black] (1,0) rectangle ++(1,1);
\draw[black,fill=red!85!black] (5,0) rectangle ++(1,1);
\draw[black,fill=red!40!yellow!60!black] (1,1) rectangle ++(1,1);

	\draw[->,thick,darkred] (4.5,3.5) -- (4.5,2.5);
	\draw[->,thick,darkred] (2.5,3.5) -- (2.5,2.5);
\end{tikzpicture}\hspace{1cm}
\begin{tikzpicture}[scale=0.8]
\draw[thin,fill=blue!4] (-0.5,-0.5) rectangle (6.5,3.5);

\draw (0,0) -- (6,0);
\draw[black,fill=blue!70!black] (0,0) rectangle ++(1,1);
\draw[black,fill=red!70!black] (1,0) rectangle ++(1,1);
\draw[black,fill=red!85!black] (5,0) rectangle ++(1,1);
\draw[black,fill=red!40!yellow!60!black] (1,1) rectangle ++(1,1);

	\draw[black,fill=blue!65!red!80!black] (4,0) rectangle ++(1,1);
	\draw[black,fill=green!65!red!80!black] (2,1) rectangle ++(1,1);
\end{tikzpicture}
\end{center}
\caption{Example of two ballistic deposition update events. The two arrows indicate the locations of the two clocks that ring
and the right picture shows the evolution of the left picture after these two
update events.}\label{fig:BD}
\end{figure}

The BD model is a Markov process $h$ taking values in $\Z^\Z$ 
and informally defined as follows. Take a family of i.i.d.\ exponential clocks (with rate $1$) 
indexed by $x \in \Z$ and, whenever the clock located at $x$ rings, the value
of $h(x)$ is updated according to the rule
\begin{equ}[e:BD]
h(x) \mapsto \max\{h(x-1), h(x)+1, h(x+1)\}\;.
\end{equ}
This update rule is usually interpreted as a ``brick'' falling down at site $x$ and then either sticking
to the topmost existing brick at sites $x-1$ or $x+1$, or coming to rest on top of the
existing brick at site $x$. See Figure~\ref{fig:BD} for an example illustrating two steps of
this dynamic.
The result of a typical medium-scale simulation is shown in Figure~\ref{fig:BDsim},
suggesting that $x \mapsto h(x)$ is locally Brownian. 

\begin{figure}[h]
\begin{center}
\includegraphics[width=14cm]{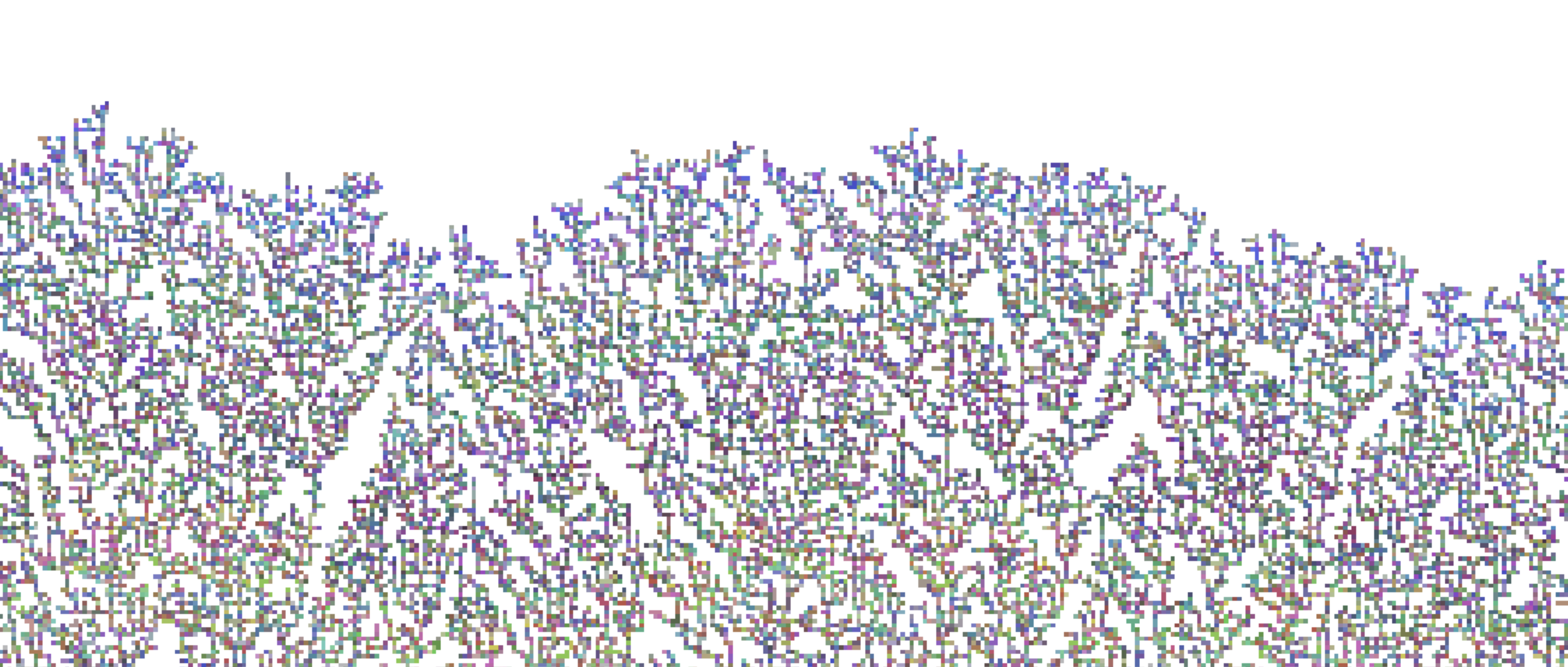}
\end{center}
\caption{Medium-scale simulation of the ballistic deposition process.}\label{fig:BDsim}
\end{figure}

A natural one-parameter family containing ballistic deposition is given
by interpreting the maximum appearing in \eqref{e:BD} as a ``zero-temperature'' limit and, for $\beta\geq 0$, to consider 
instead the update rule
\begin{equ}[e:betaBD]
h(x) \mapsto y \in \{h(x-1), h(x)+1, h(x+1)\}\;, \quad \P(y = \bar y) \propto e^{\beta \bar y}\;.
\end{equ}
As $\beta \to \infty$, this does indeed reduce to \eqref{e:BD}, while $\beta = 0$ corresponds
to a natural uniform reference measure for ballistic deposition. It is then legitimate
to ask whether \eqref{e:weakKPZ} holds if we take for $h_\eps$ the process just described with
a suitable choice of $\beta=\beta(\eps)$.

\begin{figure}[ht]
\begin{center}
\includegraphics[width=14cm]{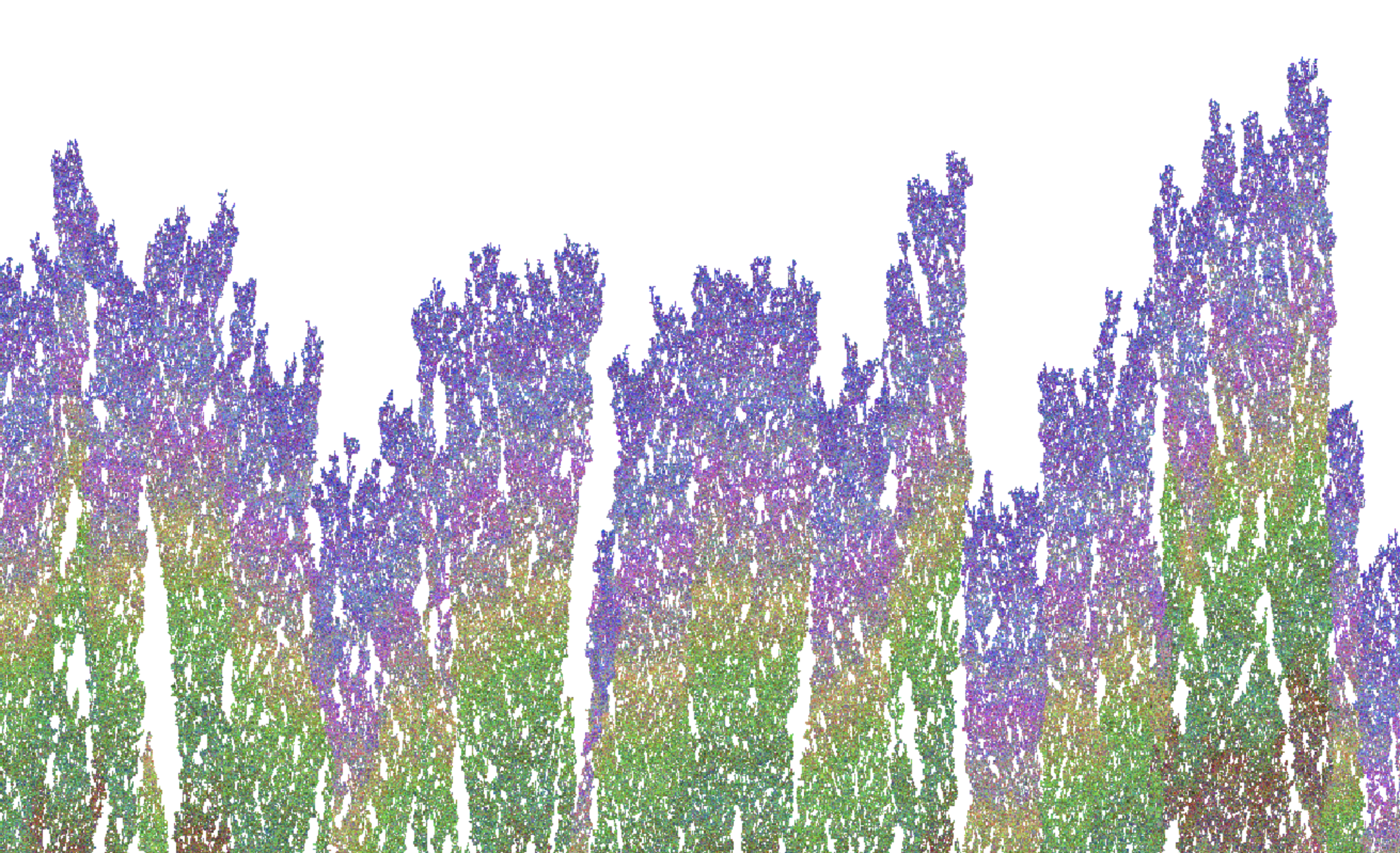}
\end{center}
\caption{Large-scale simulation of $\beta=0$ ballistic deposition.}\label{fig:BC}
\end{figure}
Surprisingly, this is \textit{not} the case. The reason however is not that ballistic deposition
isn't in the KPZ universality class, but that its $\beta = 0$ version (which we will refer to as $0$-Ballistic Deposition model) 
does not belong to the Edwards-Wilkinson universality class. Indeed, Figure~\ref{fig:BC} shows
what a typical large-scale simulation of this process looks like. It is clear from this picture
that its large-scale behaviour is not Brownian; in fact, it does not even appear to be continuous!
The aim of this article is to describe the scaling limit of this process, which we
denote by $h_\BC$ and call the ``Brownian Castle'' in view of the turrets and crannies apparent in the simulation.

\subsection{The $0$-Ballistic Deposition and its scaling limit}

In order to understand how to characterise the Brownian Castle, it is convenient to take a step back 
and examine more closely the $0$-Ballistic Deposition model. 
In view of~\eqref{e:betaBD}, the dynamics of $0$-BD is driven by three independent  
Poisson point processes $\mu^L$, $\mu^R$ and $\mu^\bullet$ on $\R\times\Z$, whose intensity is 
$\lambda/2$ for $\mu^R$ and $\mu^L$ and $\lambda$ for $\mu^\bullet$, $\lambda$ being 
such that for every $k\in\Z$, $\lambda(\dd t, k)$ is the Lebesgue measure on $\R$.\footnote{This is
actually a slightly different model from that described in~\eqref{e:betaBD} where 
the three processes have the same intensity. The present choice is so that as many constants as possible take 
the value $1$ in
the limit, but other (symmetric) choices yield the same limit modulo a simple rescaling. }
Each event of $\mu^L$, $\mu^R$ and $\mu^\bullet$ is responsible of one of the three 
possible updates $h_{\bd}(x) \mapsto y$ of the height function, namely $\mu^L$ yields 
$y=h_{\bd}(x+1)$, $\mu^R$ yields $y=h_{\bd}(x-1)$, and $\mu^\bullet$ yields $y= h_{\bd}(x)+1$. 
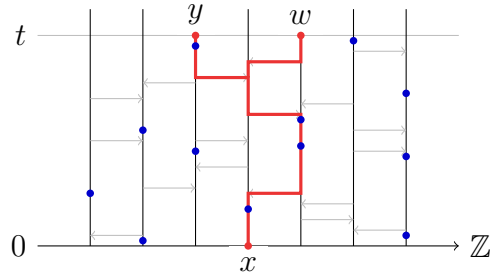
\begin{figure}[h]
\setlength{\unitlength}{0.27cm}
\begin{center}
\begin{tikzpicture}[scale=0.7,baseline=0.85cm]
\draw[->] (0,0) -- (8,0);
\draw[black!25] (0,4) -- (8,4);
\foreach \x in {1,...,7}
{
	\draw (\x,0) -- (\x,4.5);
}
\foreach \x/\y in {1/2,1/2.8,2/1.1,3/2,3/3.2,4/2.5,5/0.5,6/1.8,6/2.2,6/3.7}
{
	\draw[black!25,->] (\x,\y) -- (\x+1,\y);
}
\foreach \x/\y in {2/0.2,3/3.1,4/1.5,5/1,5/3.5,6/0.8,6/2.7,7/0.3}
{
	\draw[black!25,->] (\x,\y) -- (\x-1,\y);
}
\draw[very thick,dr!80] (4,0) -- (4,1) -- (5,1) -- (5,2.5) -- (4,2.5) -- (4,3.2) -- (3,3.2) -- (3,4);
\draw[very thick,dr!80] (4,3.2) -- (4,3.5) -- (5,3.5) -- (5,4);
\foreach \x/\y in {1/1,2/2.2,2/0.1,3/3.8,3/1.8,4/0.7,5/1.9,5/2.4,6/3.9,7/0.2,7/1.7,7/2.9}
{
	\node[inner sep=0pt,minimum size=1mm] at (\x,\y) [circle,fill=blue!80!black] {};
}

\node[fill=white,below] at (4,0) {$x$};
\node[fill=white,above] at (3,4) {$y$};
\node[fill=white,above] at (5,4) {$w$};
\node[left] at (0,0) {$0$};
\node[left] at (0,4) {$t$};
\node[right] at (8,0) {$\Z$};
\node[inner sep=0pt,minimum size=1mm] at (4,0) [circle,fill=dr!80] {};
\node[inner sep=0pt,minimum size=1mm] at (3,4) [circle,fill=dr!80] {};
\node[inner sep=0pt,minimum size=1mm] at (5,4) [circle,fill=dr!80] {};
\end{tikzpicture}
\end{center}
\vspace{-1em}\caption{Graphical representation of $0$-Ballistic Deposition. 
The red lines represent the coalescing paths $\pi^{\da,1}_{(t,y)}$ and $\pi^{\da,1}_{(t,w)}$. }\label{f:PPEvalMap}
\end{figure}
Given a realisation of these processes, we can graphically represent them as in Figure~\ref{f:PPEvalMap}, i.e.\
events of $\mu^L$ and $\mu^R$
are drawn as left / right pointing arrows, while for those of $\mu^\bullet$ are drawn as dots on $\R\times\Z$. 

Assuming that the configuration of $0$-BD at time $0$ is $h_0\in\Z^\Z$,
it is easy to see that for any $z=(t,y)\in\R_+\times\Z$, 
the value $h_{\bd}(z)$ can be obtained by going backwards following the arrows along 
the unique path $\pi^{\da}_z$ starting at $z$ and 
ending at a point in $\{0\}\times\Z$, say $(0,x)$ (the red line in Figure~\ref{f:PPEvalMap}), 
and adding to $h_0(x)$ the number of dots that are met along the way 
(in Figure~\ref{f:PPEvalMap}, $h_{\bd}(t,y)=h_0(x)+4$).

In order to obtain order one large-scale fluctuations for $h_{\bd}$, we clearly need 
to rescale space and time diffusively to ensure convergence of the random walks $\pi^{\da}_z$ to Brownian motions.
The size of the fluctuations should equally be scaled in a diffusive relation with time in order to have a 
limit for the fluctuations of the Poisson processes obtained by ``counting the number of dots''. 
In other words, the scaling exponents governing the large-scale 
fluctuations should indeed be $1:1:2$. 

The previous considerations immediately enable us to deduce the finite-dimensional distributions 
of the scaling limit of $0$-BD and consequently lead to the following definition of the Brownian Castle. 

\begin{definition}\label{def:bcFD}
Given $\sh_0\in D(\R,\R)$\footnote{Here $D(\R,\R)$ denotes all c\`adl\`ag functions from $\R$ to $\R$.},
we define the Brownian Castle (BC) starting from $\sh_0$ as the process $h_\bc:\R_+\times\R\to\R$ with
finite-dimensional distributions at space-time points $\{(t_i,x_i)\}_{i\le k}$ with $t_i > 0$ given as follows.
Consider $k$ coalescing Brownian motions $B_i$, running backwards in time and such that each $B_i$ is defined on $[0,t_i]$
with terminal value $B_i(t_i) = x_i$. For each edge $e$ of the resulting rooted forest, consider independent 
Gaussian random variables $\xi_e$ with variance equal to the length $\ell_e$ of the time interval corresponding to $e$.
We then set $h_\bc(t_i,x_i) = \sh_0(B_i(0)) + \sum_{e \in E_i} \xi_e$, where $E_i$ denotes the set of
edges that are traversed when joining the $i$th leaf to the root of the corresponding coalescence tree.
See Figure~\ref{fig:exBC} for a graphical description.
\end{definition}

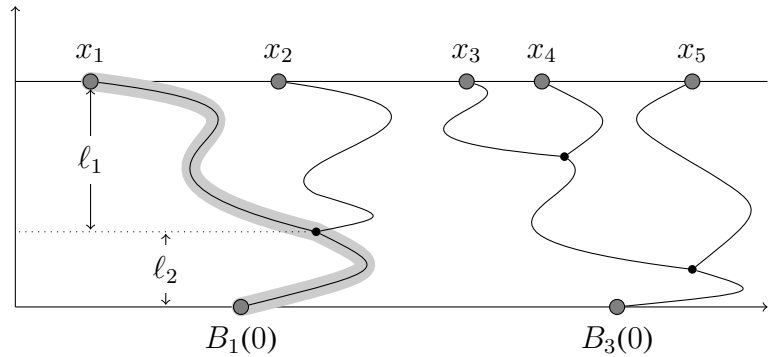
\begin{figure}
\begin{center}
\begin{tikzpicture}
\draw[->] (0,0) -- (10,0);
\draw (0,3) -- (10,3);
\draw[->] (0,0) -- (0,4);

\draw[line width=2.5mm,line cap=round,black!20] (1,3) .. controls (5,2.5) and (0,2) .. (4,1) .. controls (5,0.5) .. (3,0);

\node[dot,label=above:{$x_1$}] (x1) at (1,3) {};
\node[dot,label=above:{$x_2$}] (x2) at (3.5,3) {};
\node[dot,label=above:{$x_3$}] (x3) at (6,3) {};
\node[dot,label=above:{$x_4$}] (x4) at (7,3) {};
\node[dot,label=above:{$x_5$}] (x5) at (9,3) {};

\node[dot,label=below:{$B_1(0)$}] (y1) at (3,0) {};
\node[dot,label=below:{$B_3(0)$}] (y3) at (8,0) {};

\node[vertex] (j1) at (4,1) {};
\node[vertex] (j2) at (7.3,2) {};
\node[vertex] (j3) at (9,0.5) {};
\draw (x1) .. controls (5,2.5) and (0,2) .. (j1) .. controls (5,0.5) .. (y1);
\draw (x2) .. controls (7,2.5) and (3,2) .. (4,1.5) .. controls (5,1.2) .. (j1);

\draw (x3) .. controls (7,2.7) and (4,2.3) .. (j2) .. controls (8,1.5) and (5,1) .. (j3) 
	.. controls (10,0.2) .. (y3);
\draw (x4) .. controls (8,2.5) .. (j2);
\draw (x5) .. controls (8,2.5) .. (9.2,1.8)  .. controls (10,1.3) .. (j3);
\draw[dotted] (j1) -- +(-4,0);
\draw[<->] (x1)  -- (1,1.03) node[pos=0.5,fill=white] {$\ell_1$};
\draw[<->] (2,0.97) -- (2,0.03) node[pos=0.5,fill=white] {$\ell_2$};
\end{tikzpicture}
\end{center}
\caption{Example of construction of the $5$-point distribution of the Brownian
Castle.}\label{fig:exBC}
\end{figure}

The existence of such a process $h_\bc$ is guaranteed by Kolmogorov's extension theorem. 
One goal of the present article is to provide a finer description of the Brownian Castle which 
allows to deduce some of its pathwise properties and to show that $0$-BD converges to it in a topology
that is significantly stronger than just convergence of $k$-point distributions. 

\subsection{The Brownian Castle: a global construction and main results}

As Definition~\ref{def:bcFD} and the above discussion suggest, any global construction of the Brownian Castle must 
comprise two components. On the one hand, since we want to define it at {\it all} points simultaneously, 
we need a family of backward coalescing trajectories 
starting from every space-time point
and each distributed as a Brownian motion. On the other, we need a stochastic process indexed by the points 
on these trajectories whose increment between, say, $z_1,\,z_2\in\R^2$, is a Gaussian random variable with variance 
given by their {\it ancestral distance}, i.e. by the time it takes for the trajectories 
started from $z_1$ and $z_2$ to meet. 

The first of these components is the (backward) Brownian Web  
which was originally constructed in~\cite{TW} and further studied in~\cite{FINR}. 
In the present context, it turns out to be more convenient to work with the 
characterisation provided by~\cite{CHbwt}. The latter has the advantage of highlighting 
the {\it tree structure} of the Web 
which in turn determines the distribution of the increments of the Gaussian process indexed by it. 

The starting point in our analysis is to construct a random couple 
$\chi_\bc\eqdef(\zeta^\da_\bw, B_\bc)$ (and an appropriate Polish space in which it lives), 
whose first component $\zeta^\da_\bw=(\ST^\da_\bw,\ast^\da_\bw,d^\da_\bw, M^\da_\bw)$ 
is the Brownian Web Tree of~\cite{CHbwt}. The terms in $\chi_\bc$ can heuristically be described as
\begin{itemize}[noitemsep,label=-]
\item $(\ST^\da_\bw,\ast^\da_\bw, d^\da_\bw)$ is an pointed $\R$-tree (see Definition~\ref{def:Rtree}) 
which should be thought of as the set of ``placeholders'' for 
the points in the trajectories, and whose elements are morally of the form $( s, \pi^\da_z)$, 
where $\pi^\da_z$ is a backward path in the Brownian Web $\cW$ from $z=(t,x)\in\R^2$ (there can be more than $1$!)
and $s<t$ is a time, and in which the distance $d^\da_\bw$ is the {\it ancestral distance} given by 
\begin{equ}[e:AncestralD]
d_\bw^\da((t,\pid_{z}),(s,\pid_{z'}))\eqdef (t+s)-2\tau^\da_{t,s}(\pid_{z},\pid_{z'})\;,
\end{equ}
where the coalescence time $\tau^\da_{t,s}$ is given by $\tau^\da_{t,s}(\pid_{z},\pid_{z'})\eqdef\sup\{r<t\wedge s\,:\, \pid_{z}(r)=\pid_{z'}(r)\}$,
\item $M^\da_\bw$ is the {\it evaluation map} which associates to the abstract placeholder in $\ST^\da_\bw$ 
the actual space-time point 
in $\R^2$ to which it corresponds, i.e. $M^\da_\bw\colon (s, \pi^\da_z)\mapsto (s, \pi^\da_z(s))$,
\item $B_\bc$ is the {\it branching map}, which corresponds to 
the Gaussian process indexed by $\ST^\da_\bw$ and such that 
\begin{equ}
\E[(B_\bc(t,\pid_{z})-B_\bc(s,\pid_{z'}))^2]=d_\bw^\da((t,\pid_{z}),(s,\pid_{z'}))\,.
\end{equ}
\end{itemize}
With such a couple at hand, we would like to define the Brownian Castle starting at $\sh_0$ by setting 
\begin{equ}[e:FormalBC]
\sh_\bc(z)\eqdef \sh_0(\pi_z^\da(0))+B_\bc(t,\pi^\da_z)-B_\bc(0,\pi^\da_z)\,,\qquad\text{for all $z\in\R_+\times\R$.}
\end{equ}
The above definition implicitly relies on the fact that 
we are assigning to every point $z\in\R^2$ a point in $\ST^\da_\bw$ (and consequently a path $\pid_z\in\cW$), but, 
as it turns out (see~\cite[Theorem 3.24]{CHbwt}), in the Brownian Web Tree
there are ``special points'' from which more than one path originates. 
Since anyway this number is finite (at most $3$), we can always pick 
the right-most trajectory and ensure well-posedness of~\eqref{e:FormalBC} 
(see the definition of the tree map in Section~\ref{sec:TreeMap} that makes this assignment rigorous). 

The special points of the Brownian Web Tree are particularly relevant for the Brownian Castle because
they are the points at which the discontinuities generating the turrets and crannies 
in Figure~\ref{fig:BDsim} can be located. Thanks to~\eqref{e:FormalBC}, we will not only be able to 
detect these points but also track the space-time behaviour 
of the (dense) set of discontinuities of $h_\BC$ (see Section~\ref{sec:BC}). 

In the following theorem, we loosely state some of the main results concerning the Brownian Castle 
which can be obtained by virtue of the construction above. 

\begin{theorem}\label{thm:main}
The Brownian Castle in Definition~\ref{def:bcFD} admits a version $\sh_\bc$ such that for all $\sh_0\in D(\R,\R)$, 
$t\mapsto \sh_\bc(t,\cdot)$ is a right-continuous map with values in $D(\R,\R)$ 
endowed with the Skorokhod topology $d_\Sk$ in~\eqref{e:Sk}, which is continuous except for 
a countable subset of $\R_+$,  
but admits no version which is c\`adl\`ag in both space and time. 

$\sh_\bc$ is a time-homogeneous $D(\R,\R)$-valued Markov process, 
satisfying both the strong Markov and the Feller properties, 
which is invariant under the $1:1:2$ scaling, 
i.e. if $\sh_\bc^{i}$ with $i \in \{1,2\}$ are two instances of the Brownian Castle
with possibly different initial conditions at time $0$, then, for all $\lambda>0$, one has the equality in law
\begin{equ}[e:eqlawBC]
\sh_\bc^{1}(t,x) \eqlaw \lambda^{-1} \sh_\bc^{2}(\lambda^2 t, \lambda x),\qquad t \ge 0,\,x\in \R,
\end{equ}
(viewed as an equality between space-time processes) provided that \eqref{e:eqlawBC}
holds as an equality between spatial processes at time $t=0$. 

Moreover, when quotiented by vertical translations, $\sh_\bc(t,\cdot)$ converges in law, as $t \to \infty$, 
to a stationary process whose increments are Cauchy but which is singular with respect to the Cauchy process. 
\end{theorem}

A more precise formulation of this theorem, together with its proof, is split in the various statements 
contained in Section~\ref{s:bc}. 

\begin{remark}
When we say that ``a process $h$ admits a version having property $P$'', we mean that there exists a standard probability
space $\Omega$ endowed with a collection of random variables $\mathbf{h}(z)$ such that
for any finite collection $\{z_1,\ldots ,z_k\}$, the laws of $(h(z_i))_{i\le k}$ and
$(\mathbf{h}(z_i))_{i\le k}$ coincide and furthermore $\mathbf{h}^{-1}(P) \subset \Omega$ 
is measurable and of full measure. 
\end{remark}

The second task of the present article is to show that the $0$-Ballistic Deposition model indeed converges to it. 
Thanks to the heuristics presented above, in order to expect any meaningful limit, we need to recentre and 
rescale the $0$-BD height function $h_{\bd}$ according to the $1:1:2$ scaling, so we set 
\begin{equ}[e:Scaled]
h^\delta_{\bd}(t,x)\eqdef \delta\Big(h_{\bd}\Big(\frac{t}{\delta^2},\frac{x}{\delta}\Big)-\frac{t}{\delta^2}\Big)\,,\qquad\text{for all $z\in\R_+\times\R$.}
\end{equ}
Now, given the way in which Definition~\ref{def:bcFD} was derived, convergence of $h_{\bd}^\delta$ to $h_\bc$ 
in the sense of finite-dimensional distributions 
should not come as a surprise since it is an almost immediate consequence of Donsker's invariance principle. 
That said, we aim at investigating a stronger form of convergence which relates $0$-BD and BC as 
space-time processes. The major obstacle here is that 
Theorem~\ref{thm:main} explicitly asserts that the Brownian Castle $h_\bc$ 
does not live in any ``reasonable'' space which is Polish and in which point evaluation is a measurable operation, 
so that {\it a priori} it is not even clear in what sense such a convergence should be stated. 
It is at this point that our construction, summarised by the expression in~\eqref{e:FormalBC}, comes once more into play. 

As we have seen above, the version of the Brownian Castle $\sh_\bc$ given in~\eqref{e:FormalBC}, 
is fully determined by the couple $\chi_\bc\eqdef(\zeta^\da_\bw, B_\bc)$, 
which in turn was inspired by the graphical representation 
of the $0$-Ballistic Deposition model illustrated in Figure~\ref{f:PPEvalMap}. 
For any realisation of the Poisson random measures 
$\mu^L$, $\mu^R$ and $\mu^\bullet$ suitably rescaled and (for the latter) compensated, 
it is possible to build $\chi^\delta_{\bd}\eqdef(\zeta^\da_\delta, N_\delta)$, 
in which $\zeta^\da_\delta\eqdef(\ST_\delta^\da,\ast^\da_\delta, d_\delta^\da, M_\delta^\da)$ 
is the Discrete Web Tree of~\cite[Definition 4.1]{CHbwt} and
encodes the family of coalescing
backward random walks $\pi^{\da,\delta}$ naturally associated to the random measures $\mu^L$ and $\mu^R$, 
while $N_\delta$ is the compensated Poisson point process indexed by $\ST_\delta^\da$ and induced by $\mu^\bullet$ 
(the precise construction of $\chi^\delta_{\bd}$ can be found in Section~\ref{sec:graphical}). 
Given any initial condition $\sh^\delta_0\in D(\R,\R)$, we now set, analogously to~\eqref{e:FormalBC},
\begin{equ}[e:FormalBD]
\sh^\delta_{\bd}(z)\eqdef \sh^\delta_0(\pi_{\bar z}^{\da,\delta}(0))+N_\delta(t,\pi_{\bar z}^{\da,\delta})-N_\delta(0,\pi_{\bar z}^{\da,\delta})\,,\qquad\text{for all $z\in\R_+\times\R$,}
\end{equ}
where, for $z=(t,x)$, $\bar z=(t,\delta\lfloor x/\delta\rfloor)\in\R_+\times(\delta\Z)$. Notice that $\sh^\delta_{\bd}$ 
is a c\`adl\`ag (in both space and time) version of $h_{\bd}^\delta$ started from $\sh_0^\delta$, in the sense that 
all of their $k$-point (in space-time) marginals coincide. 
In force of the previous construction, we are able to state the following theorem whose proof can be found at the 
end of Section~\ref{sec:0BD}. 

\begin{theorem}\label{thm:convTime}
Let $\{\sh_0^\delta\}_\delta,\,\sh_0\subset D(\R,\R)$ be such that $d_{\Sk}(\sh_0,\sh_0^\delta)\to 0$. 
Then, for every sequence $\delta_n \to 0$ there exists a version of $\sh^\delta_{\bd_n}$ and $\sh_\bc$
for which, almost surely, there exists a countable set of times $D$ such that for every $T\in\R_+\setminus D$ 
\begin{equ}[e:ConvDC]
\lim_{\delta\to0} d_{\Sk}(\sh^\delta_{\bd}(T,\cdot),\sh_\bc(T,\cdot))=0\,.
\end{equ}
Here, $\sh^\delta_{\bd}$ starts from $\sh_0^\delta$ and is defined in~\eqref{e:FormalBD}
while $\sh_\bc$ is the Brownian Castle started from $\sh_0$. 
\end{theorem}

This theorem also provides information concerning the nature 
and evolution of the discontinuities of the Brownian Castle. Indeed, for~\eqref{e:ConvDC} to hold, it 
cannot be the case that many small discontinuities of $0$-BD add up and 
ultimately create a large discontinuity for BC. Instead, the statement shows that
the major discontinuities of the former converge to those of the latter. 

To see this phenomenon and prove Theorem~\ref{thm:convTime}, 
the main ingredient is the convergence 
of $\chi_{\bd}^\delta$ to $\chi_\bc$ (and that of the dual of the Discrete Web Tree to the dual of 
the Brownian Web Tree as stated in~\cite[Theorem 4.5]{CHbwt}). 
The topology in which such a convergence holds (see Section~\ref{sec:Trees}) 
is chosen in such a way that 
the convergence of the evaluation maps morally provides a control over the sup norm 
distance of discrete and continuous backward trajectories and is therefore similar in spirit to that in, e.g.,~\cite{FINR}, 
while that of the trees guarantees that couples of distinct discrete and continuous paths which are close 
also coalesce approximately at the same time. This is a crucial point (which moreover distinguishes our work 
from the previous ones) since it is at the basis of the convergence of $N_\delta$ to $B_\bc$ and ultimately
ensures that of $\sh^\delta_{\bd}$ to $\sh_\bc$. 

%
%
%
%
%
%

\subsection{The BC universality class and further remarks}

Over the last two decades, the KPZ (and EW) universality class has been at the heart 
of an intense mathematical interest because of the challenges it posed and 
the numerous physical systems which abide its laws (see~\cite{QS} for a review and~\cite{ACQ,SS, KPZfp, Virag, QS1} among  many other recent results). 
This article and the results stated herein establish the existence of a {\it new} universality class, 
which we will refer to as the {\it BC universality class}, by characterising a novel scale-invariant stochastic 
process, the Brownian Castle, which encodes the fluctuations of models in this class and arises as the 
scaling limit of a microscopic random system,
the $0$-Ballistic Deposition model. It is natural to wonder what are the features a model should
exhibit in order to belong to it. 
Given the analysis of the Brownian Castle outlined above, it is reasonable to expect that 
any interface model which displays both {\it horizontal} and {\it vertical fluctuations} but {\it no smoothing} 
is an element of the BC universality class. The first type of fluctuations is responsible for the 
(coalescing) Brownian characteristics in the limit, while the second determines the Brownian motion indexed by them. 
A model which possesses these features and is somewhat paradigmatic for the class 
is the random transport equation given by 
\begin{equ}[e:BCeq]
\partial_t h=\eta\,\partial_x h  +\mu\;,
\end{equ}
where $\eta$ and $\mu$ are two space-time stationary random fields, 
the first being responsible for the horizontal / lateral fluctuations and the second for the vertical ones. 
We conjecture that, provided the noises are sufficiently mixing so that some form of functional central limit theorem 
applies, under the $1:1:2$ scaling the solution of~\eqref{e:BCeq} converges (in a weak sense) to the Brownian Castle. 
We conclude this introduction by pointing out some aspects of the construction of the Brownian Castle 
and the description of the $0$-Ballistic Deposition model, commenting on their relation with the existing literature. 

The importance of the Brownian Web lies in its connection with many interesting 
physical and biological systems (population genetics models, drainage networks, random walks in random environment...)
and a thorough account on the advances of the research behind it can be found in~\cite{SSS}. 
One of the most notable generalisations of the Brownian Web
is the Brownian Net~\cite{SS} (and the stochastic flows therein~\cite{SSSflows}), which arises as the scaling limit 
of a collection of coalescing random walks that in addition have a small probability to branch. 
It is then natural 
to wonder if a construction similar to that carried out here is still possible starting from the Brownian Net, 
and what the corresponding ``Castle'' would be in this context. We believe that such considerations allow
to build crossover processes connecting the BC and EW universality classes, namely such that 
their small scale statistics are BC, while their large scale fluctuations are EW. 

From the perspective of discrete interacting systems, let us also mention 
that the graphical representation of the $0$-Ballistic Deposition is in itself not new. 
A picture analogous to Figure~\ref{f:PPEvalMap}, can be found in~\cite{NoisyVoter}, where the authors introduce 
the so-called noisy voter model. The latter can be obtained by the usual voter model (see~\cite{Li} for the definition), 
whose graphical representation is the same as that in Figure~\ref{f:PPEvalMap} but without {\color{blue}$\bullet$}, 
by adding spontaneous flipping, illustrated by the realisation of $\mu^\bullet$.  
Let us remark that not only the meaning but also the limiting procedure involving $\mu^\bullet$ is different in the two cases. 
In the present setting the intensity of $\mu^\bullet$ is fixed, while
it is sent to $0$ at a suitable rate in~\cite{FINRb}.

\subsection{Outline of the paper}

In Section \ref{sec:Trees}, we recall the main definitions related to $\R$-trees and 
the topology and construction of the Brownian Web Tree $\zeta^\da_\bw$ given in~\cite{CHbwt}. 

In Section~\ref{sec:BST}, we build the couple $\chi_\bc$. 
We introduce, for $\alpha,\,\beta\in(0,1)$, 
the space $\T^{\alpha,\beta}_\bsp$ (and its ``characteristic'' subset 
$\Ch^{\alpha,\beta}_\bsp$) in which the couple $\chi_\bc$ lives, and 
define the metric that makes it Polish (Section~\ref{sec:bsp}). 
We then provide conditions under which a stochastic process $X$
indexed by a spatial tree $\zeta$ admits a H\"older continuous modification, construct both Gaussian  
and Poisson processes indexed by a generic spatial tree, and prove that the map 
$\zeta\mapsto\mathrm{Law}(\zeta,X)$ is continuous. (Sections~\ref{sec:MEC}-\ref{sec:MECc}.) 
At last, combining this with the results of Section~\ref{sec:Trees} we construct the law of 
$\chi_\bc$ on $\Ch^{\alpha,\beta}_\bsp$ in Section~\ref{sec:BCM}. 

Section \ref{s:bc} is devoted to the Brownian Castle and its periodic counterpart, and 
contains the proof of Theorem~\ref{thm:main}. In Section~\ref{sec:BC}, 
we first define the version formally given in~\eqref{e:FormalBC} and 
determine its continuity properties (among which the finite $p$-variation for $p>1$), 
then we study the location and structure of its discontinuities 
and analyse their relation with the special points of the Brownian Web. 
Afterwards, in Section~\ref{sec:BCprocess}, we prove the Markov, strong Markov, Feller and strong Feller properties 
(the latter holds only in the periodic case) and study its long-time behaviour. 
In Section \ref{sec:BCdist}, we derive the distributional 
properties of the Brownian Castle (scale invariance and multipoint distributions) and show that 
although its invariant measure has increments that are Cauchy distributed and has finite 
$p$-variation for any $p>1$, it is singular with 
respect to the law of the Cauchy process. 

In Section~\ref{sec:0BD}, we turn our attention to the $0$-Ballistic Deposition model. At first, 
we associate the triplet $\chi^\delta_{\bd}$ to it and show that the latter converges to $\chi_\bc$ 
(Sections~\ref{sec:graphical} and~\ref{sec:0BDconv}) and then (Section~\ref{sec:Conv}), 
we prove Theorem~\ref{thm:convTime}. Finally, the appendix collects a number of relatively straightforward
technical results.

\subsection*{Notations} 

We will denote by $|\cdot|_e$ the usual Euclidean norm on $\R^d$, $d\geq 1$, and 
adopt the short-hand notations $|x|\eqdef|x|_e$ and $\|x\|\eqdef|x|_e$ for $x\in\R$ and $\R^2$ respectively. 
Let $(\ST,d)$ be a metric space. We define the Hausdorff distance $d_H$ between two non-empty subsets $A,\,B$ of $\ST$ as 
\begin{equ}[e:Hausdorff]
d_H(A,B)\eqdef\inf\{\eps\colon A^\eps\subset B\text{ and } B^\eps\subset A\}
\end{equ}
where $A^\eps$ is the $\eps$-fattening of $A$, i.e. $A^\eps=\{\sz\in\ST\colon \exists\, \sw\in A\text{ s.t. } d(\sz,\sw)< \eps\}$.

Let $(\ST,d,\ast)$ be a pointed metric space, i.e. $(\ST,d)$ is as above and $\ast \in \ST$, 
and let $M\colon \ST\to\R^d$ be a map. For $r>0$ and $\alpha\in(0,1)$, we define the $\sup$-norm 
and $\alpha$-H\"older norm of $M$ restricted to a ball of radius $r$ as
\begin{equ}
\|M\|^{(r)}_\infty\eqdef\sup_{\sz\in B_d(\ast,r]}|M(\sz)|_e\,,\qquad \|M\|^{(r)}_\alpha\eqdef\sup_{\substack{\sz,\sw\in B_d(\ast,r]\\d(\sz,\sw)\leq 1}}\frac{|M(\sz)-M(\sw)|_e}{d(\sz,\sw)^\alpha}\,.
\end{equ}
where $B_d(\ast,r]\subset\ST$ is the closed ball of radius $r$ centred at $\ast$, and, for $\delta>0$, 
the modulus of continuity as
\begin{equation}\label{e:MC}
\omega^{(r)}(M,\delta)\eqdef \sup_{\substack{\sz,\sw\in B_d(\ast,r]\\d(\sz,\sw)\leq \delta}} |M(\sz)-M(\sw)|_e\,.
\end{equation}
In case $\ST$ is compact, in all the quantities above, the suprema are taken over the whole space $\ST$ and 
the dependence on $r$ of the notation will be suppressed. 
Moreover, we say that a function $M$ is (locally) {\it little $\alpha$-H\"older continuous} if for all $r>0$, 
$\lim_{\delta\to 0} \delta^{-\alpha} \omega^{(r)}(M,\delta)=0$. 

Let $I\subseteq \R$ be an interval and $(\CX,d)$ be a complete separable metric space. 
We denote the space of c\`ad\`ag functions on $I$ with values in $\CX$ as $D(I,\CX)$ and, for $f\in D(I,\CX)$, 
the set of discontinuities of $f$ by $\Disc(f)$. 
We will need two different metrics on $D(I,\CX)$, corresponding to the so-called J1 (or Skorokhod) and M1 
topologies. For the first, let $\Lambda(I)$ be the space of strictly increasing continuous homeomorphisms on 
$I$ such that 
\begin{equ}
\gamma(\lambda)\eqdef \sup_{t\in I} |\lambda(t)-t|\vee \sup_{\substack{s,t\in I\\s<t}}\left| 
\log\left(\frac{\lambda(t)-\lambda(s)}{t-s}\right)\right|<\infty\,.
\end{equ}
Then, for $\lambda\in\Lambda(I)$ and $f,g\in D(I,\CX)$ we set $d^I_\lambda(f,g)\eqdef 1\vee \sup_{s\in I}d(f(s),g(\lambda(s)))$, 
so that the Skorokhod metric is given by 
\begin{equ}[e:Sk]
d_\Sk(f,g)\eqdef\inf_{\lambda\in\Lambda(I)} \gamma(\lambda)\vee d^I_\lambda(f,g)\;,\quad d_\Sk(f,g)\eqdef\inf_{\lambda\in\Lambda} \gamma(\lambda)\vee \int_0^\infty e^{-t}\,d^{[-t,t]}_\lambda(f,g)\, \dd t\;,
\end{equ}
where in the first case $I$ is assumed to be bounded. 
For the M1 metric instead, we restrict to the case of $\CX=\R_+\eqdef[0,\infty)$. Given $f\in D(I,\R_+)$, denote by 
$\Gamma_f$ its completed graph, i.e. the graph of $f$ to which all the vertical segments joining the points 
of discontinuity are added, and order it by saying that $(x_1,t_1)\leq (x_2,t_2)$ if either $t_1<t_2$ 
or $t_1=t_2$ and $|f(t_1^-)-x_1|\leq |f(t_1^-)-x_2|$. 
Let $P_f$ be the set of all parametric representations of $\Gamma_f$, which is the set of all non-decreasing 
(with respect to the order on $\Gamma_f$) functions $\sigma_f\colon I\to \Gamma_f$. 
Then, if $I$ is bounded, we set 
\begin{equ}
\hat d^\com_{\M}(f,g)\eqdef 1\vee\inf_{\sigma_{f},\sigma_g} \|\sigma_{f}-\sigma_{g}\|
\end{equ}
and $d^\com_{\M}(f,g)$ to be the topologically equivalent metric with respect to which 
$D(I,\R_+)$ is complete (see~\cite[Section 8]{Whitt} for more details). If instead $I=[-1,\infty)$, we define
\begin{equ}[def:M1metric]
d_{\M}(f,g)\eqdef \int_0^\infty e^{-t} \big( 1\wedge d_\M^{\com}(f^{(t)},g^{(t)})\big)\,\dd t
\end{equ}
where $f^{(t)}$ is the restriction of $f$ to $[-1, t]$. 

For $p>0$, we say that a function $f\colon \R\to\CX$, $(\CX,d)$ being a metric space, has 
finite $p$-variation if for every bounded interval $I\subset \R$
\begin{equ}[e:pvar]
\|f\|_{\pvar,I}\eqdef \left(\sup_{t_k\in D_I} d(f(t_k),f(t_{k-1}))^p\right)^{1/p}<\infty
\end{equ}
where $D_I$ ranges over all finite partitions of $I$. 

The Wasserstein distance of two probability measures $\mu,\nu$ on a complete separable metric space $(\ST,d)$ is defined as 
\begin{equ}[def:Wass]
\cW(\mu,\nu)\eqdef \inf_{g\in\Gamma(\mu,\nu)} \Exp_\gamma[d(X,Y)]
\end{equ}
where $\Gamma(\mu,\nu)$ denotes the set of couplings of $\mu$ and $\nu$, the expectation is taken with respect to $g$ 
and $X,Y$ are two $\ST$-valued random variables distributed according to $\mu$ and $\nu$ respectively. 

At last, we will write $a\lesssim b$ if there exists a constant $C>0$ such that $a\leq C b$  
and $a\approx b$ if $a\lesssim b$ and $b\lesssim a$.

\subsection*{Acknowledgements}

{\small
The authors would like to thank Rongfeng Sun and Jon Warren for useful discussions. 
GC would like to thank the Hausdorff Institute in Bonn for the kind hospitality during the programme 
``Randomness, PDEs and Nonlinear Fluctuations'', where he carried out part of this work. 
GC gratefully acknowledges financial support via the EPSRC grant EP/S012524/1.
MH gratefully acknowledges financial support from the Leverhulme trust via a Leadership Award, 
the ERC via the consolidator grant 615897:CRITICAL, and the Royal Society via a research professorship. 
}

\section{The Brownian Web Tree and its topology}\label{sec:Trees}

In this section, we provide the basic definitions and notations on (characteristic) $\R$-trees 
and outline the construction and main properties of the (Double) Brownian Web Tree derived in~\cite{CHbwt}. 

\subsection{Characteristic $\R$-trees in a nutshell}\label{sec:SpTrees}

We begin by recalling the notion of $\R$-tree (see~\cite[Definition 2.1]{DL}), 
degree of a point, segment, ray and end. 

\begin{definition}\label{def:Rtree}
A metric space $(\ST,d)$ is an $\R$-tree if for every $\sz_1,\sz_2\in\ST$ 
\begin{enumerate}[noitemsep]
\item there is a unique isometric map $f_{\sz_1,\sz_2} :[0, d(\sz_1,\sz_2)]\to\ST$ such that $f_{\sz_1,\sz_2}(0)=\sz_1$ and 
	$f_{\sz_1,\sz_2}(d(\sz_1,\sz_2))=\sz_2$,
\item for every continuous injective map $q\colon [0,1] \to \ST$ such that $q(0)=\sz_1$ and $q(1)=\sz_2$, one has 
\begin{equ}
q([0,1])=f_{\sz_1,\sz_2}([0,d(\sz_1,\sz_2)])\,.
\end{equ}
\end{enumerate}
A {\it pointed $\R$-tree} is a triple $(\ST,\ast,d)$ such that $(\ST,d)$ is an $\R$-tree and $\ast \in \ST$. 
\end{definition}

Given $\sz\in\ST$, the number of connected components of $\ST\setminus\{\sz\}$ is the {\it degree of $\sz$}, 
$\deg(\sz)$ in short. 
A point of degree $1$ is an {\it endpoint}, of degree $2$, 
an {\it edge point} and  if the degree is $3$ or higher, a {\it branch point}. 

\begin{definition}\label{def:end}
Let $(\ST,d)$ be an $\R$-tree and, for any $\sz_1,\,\sz_2\in\ST$, $f_{\sz_1,\sz_2}$ the isometric 
map in Definition~\ref{def:Rtree}. 
We call the range of $f_{\sz_1,\sz_2}$, {\it segment joining $\sz_1$ and $\sz_2$} and denote it by $\llb\sz_1,\sz_2\rrb$. 
For $\sz\in\ST$, a segment having $\sz$ as an endpoint is said to be a {\it $\ST$-ray from $\sz$} 
if it is maximal for inclusion. 
The {\it ends} of $\ST$ are the equivalence classes of $\ST$-rays with respect to the equivalence 
relation according to which different $\ST$-rays are equivalent if their intersection is again a $\ST$-ray. 
An end is {\it closed} if it is an endpoint of $\ST$ and {\it open} otherwise, and, for $\dagger$ an open end, 
we indicate by $\llb\sz,\dagger\rangle$ the unique $\ST$-ray from $\sz$ representing $\dagger$. 
$\dagger$ is said to be an {\it open end with (un-)bounded rays} if for every $\sz\in\ST$, the map 
$\iota_\sz: \llbracket\sz,\dagger\rangle\to\R_+$ given by 
\begin{equation}\label{e:iota}
\iota_\sz(\sw)=d(\sz,\sw)\,,\qquad\sw\in\llb\sz,\dagger\rangle
\end{equation}
is (un-)bounded. 
\end{definition}

Throughout the article, we will work with (subsets of) the space of {\it spatial $\R$-trees}, 
consisting of $\R$-trees embedded into $\R^2$ via a map, called the evaluation map. 

\begin{definition}\label{def:CharTree}
Let $\alpha\in(0,1)$ and consider the quadruplet $\zeta=(\ST,\ast, d, M)$ where 
\begin{itemize}[noitemsep, label=-]
\item $(\ST,\ast,d)$ is a complete and locally compact pointed $\R$-tree, 
\item $M$, the {\it evaluation map}, is a locally little $\alpha$-H\"older continuous proper\footnote{Namely such that $\lim_{\eps \to 0} \sup_{\sz \in K} \sup_{d(\sz,\sz') \le \eps} \|M(\sz)-M(\sz')\| / d(\sz,\sz')^\alpha = 0$ for every compact $K$ and the preimage of every compact set is compact. } map 
from $\ST$ to $\R^2$. 
\end{itemize}
The space of {\it pointed $\alpha$-spatial $\R$-trees} $\T^\alpha_\Sp$ is the set of equivalence 
classes of quadruplets as above with respect to the equivalence relation that  identifies $\zeta$ and $\zeta'$ 
if there exists a bijective isometry $\phi:\ST\to\ST'$ such that $\phi(\ast)=\ast'$ and $M'\circ\phi\equiv M$, in short 
(with a slight abuse of notation) $\varphi(\zeta)=\zeta'$. 

Elements $\zeta=(\ST,\ast, d, M)\in\T^\alpha_\Sp$ are further said to be {\it characteristic}, 
and their space denoted by $\Ch^\alpha_\Sp$, if the evaluation map $M$ satisfies 
\begin{enumerate}[noitemsep, label=(\arabic*)]
\item\label{i:Back} {\it (Monotonicity in time)} for every $\sz_0,\sz_1\in\ST$ and $s \in [0,1]$ one has
\begin{equation}\label{e:Back}
M_t(\sz_s)=
\bigl(M_t(\sz_0)-s \,d(\sz_0,\sz_1)\bigr) \vee \bigl(M_t(\sz_1)-(1-s) \,d(\sz_0,\sz_1)\bigr)\;.
\end{equation}
where $\sz_s$ is the unique element of $\llb\sz_0,\sz_1\rrb$ with
$d(\sz_0,\sz_s) = s\,d(\sz_0,\sz_1)$,
\item\label{i:MonSpace} {\it (Monotonicity in space)} for every $s< t$, interval $I = (a,b)$
and any four elements $\sz_0,\bar \sz_0, \sz_1, \bar \sz_1$ 
such that $M_t(\sz_0)=M_t(\bar \sz_0) = t$, $M_t(\sz_1)=M_t(\bar \sz_1) = s$, $M_x(\sz_0)< M_x(\bar \sz_0)$,
and $M(\llb\sz_0,\sz_1\rrb), M(\llb\bar \sz_0,\bar\sz_1\rrb) \subset [s,t] \times (a,b)$, we have
\begin{equation}\label{e:MonSpace}
M_x(\sz_s)\leq M_x(\bar \sz_{s})
\end{equation}
for every $s \in [0,1]$,
\item\label{i:Spread} for all $z=(t,x)\in\R^2$, $M^{-1}(\{t\}\times[x-1,x+1])\neq\emptyset$.  
\end{enumerate}
The space of those characteristic trees for which, 
in~\ref{i:Back},~\eqref{e:Back} holds with $\vee$ instead of $\wedge$ will be denoted instead by 
$\hat\Ch^\alpha_\Sp$
\end{definition}

\begin{remark}\label{rem:Periodic}
Let $\T\eqdef\R/\Zint$ be the torus of size $1$ endowed with the usual periodic metric
$d(x,y)= \inf_{k \in \Z} |x-y+k|$. 
When $M$ is $\R\times\T$-valued, we will say that $\zeta$ is {\it periodic} and denote the space 
of periodic (characteristic) pointed $\alpha$-spatial $\R$-trees by $\T^\alpha_{\Sp,\per}$ (or $\Ch^\alpha_{\Sp,\per}$). 
As in~\cite[Definition 2.19]{CHbwt}, we point out that~\ref{i:MonSpace} makes sense also in this 
case provided we restrict to intervals $(a,b)$ that do not wrap around the torus. 
\end{remark}

In order to introduce a metric on $\T^\alpha_\Sp$ (and $\Ch^\alpha_\Sp$), 
recall that a correspondence $\CC$ between two metric spaces $(\ST,d)$, $(\ST',d')$ is a subset of 
$\ST\times\ST'$ such that for all $\sz\in\ST$ there exists at least one $\sz'\in\ST'$ for which $(\sz,\sz')\in\CC$ and vice versa. 
Its {\it distortion} is given by 
\begin{equ}
\dis \CC\eqdef \sup\{|d(\sz,\sw)-d'(\sz',\sw')|\,:\, (\sz,\sz'),\,(\sw,\sw')\in\CC\}\;. 
\end{equ}
Let $\zeta=(\ST,\ast,d,M)$ and $\zeta'=(\ST',\ast',d',M')\in\T^\alpha_\Sp$ be such that 
both $\ST$ and $\ST'$ are compact, and $\CC$ be a correspondence between them. 
We set 
\begin{equation}\label{e:MetC}
\begin{split}
\uDelta^{\com,\CC}_\Sp(\zeta,\zeta') &\eqdef \frac{1}{2}\dis \CC+\sup_{(\sz,\sz')\in\CC}\|M(\sz)-M'(\sz')\|\\
&\quad +\sup_{n\in\N}\,\,2^{n\alpha}\sup_{\substack{(\sz,\sz'),(\sw,\sw')\in\CC\\d(\sz,\sw),d'(\sz',\sw')\in\CA_n}}
\|\delta_{\sz,\sw}M-\delta_{\sz',\sw'}M'\|
\end{split}
\end{equation}
where $\CA_n\eqdef(2^{-n},2^{-(n-1)}]$ for $n\in\N$, and $\delta_{\sz,\sw}M\eqdef M(\sz)-M(\sw)$. In the above, 
we adopt the convention that if there exists no pair of couples 
$(\sz,\sz'),(\sw,\sw')\in\CC$ such that $d(\sz,\sw)\in\CA_n$, then the increment of $M$ is removed 
and the supremum is taken among all $\sz',\sw'$ such that $d'(\sz',\sw')\in\CA_n$ and vice versa\footnote{If instead we 
adopted the more natural convention $\sup \emptyset = 0$, then the triangle inequality might fail, e.g. when comparing a generic 
spatial tree to the trivial tree made of only one point. }
We can now define 
\begin{equation}\label{e:MetricCompact}
\Delta^\com_\Sp(\zeta,\zeta')\eqdef \uDelta^\com_\Sp(\zeta,\zeta')+d_\M(b_\zeta,b_{\zeta'})
\end{equation}
where 
\begin{itemize}[noitemsep,label=-]
\item the first term is
\begin{equation}\label{e:MetricC}
\uDelta^\com_\Sp(\zeta,\zeta')\eqdef\inf_{\CC\,:\,(\ast,\ast')\in\CC} \uDelta^{\com,\CC}_\Sp(\zeta,\zeta')\;,
\end{equation}
\item the map $b_\zeta$ is the {\it properness map}, which is $0$ for $r<0$, while 
for $r\ge 0$ is defined as
\begin{equation}\label{def:PropMap}
b_\zeta(r)\eqdef \sup_{\sz\,:\,M(\sz)\in\Lambda_r} d(\ast,\sz)\;,
\end{equation}
for $\Lambda_r\eqdef [-r,r]^2\subset\R^2$ and $\Lambda_r=\Lambda_r^\per\eqdef [-r,r]\times\T$, in the periodic case,
\item $d_{M_1}$ is the metric on the space of c\`adl\`ag functions given in~\eqref{def:M1metric}. 
\end{itemize}
Let us point out that by~\cite[Lemma 2.9]{CHbwt}, the properness map is non-decreasing and 
c\`adl\`ag so that the second summand in~\eqref{e:MetricCompact} is meaningful.

Since the elements of $\T^\alpha_\Sp$ are locally compact, by~\cite[Theorem 2.6(a)]{CHbwt} we can 
generalise the definition of $\Delta_\Sp^\com$ to the non-compact case. For $\zeta,\,\zeta'\in\T^\alpha_\Sp$, 
we set
\begin{equation}\label{e:Metric}
\begin{split}
\Delta_\Sp(\zeta,\zeta')&\eqdef\int_0^{+\infty} e^{-r}\, \Big[1\wedge \uDelta^{\com}_\Sp(\zeta^{(r)},\zeta'^{\,(r)})\Big]\,\dd r+ d_\M(b_\zeta,b_{\zeta'})\\
&=:\uDelta_\Sp(\zeta,\zeta')+d_\M(b_\zeta,b_{\zeta'}).
\end{split}
\end{equation} 
where, for $r>0$, 
\begin{equation}\label{e:rRestr}
\zeta^{(r)}\eqdef (\ST^{(r)}, \ast, d, M)
\end{equation} 
$\ST^{(r)}\eqdef B_d(\ast,r]$ being the closed ball of radius $r$ in $\ST$. 

\begin{proposition}\label{p:Compactness}
For any $\alpha\in(0,1]$, 
\begin{enumerate}[noitemsep, label=(\roman*)]
\item the space $(\T_\Sp^\alpha,\Delta_\Sp)$ is a complete, separable metric space, 
\item a subset $\CA=\{\zeta_a=(\ST_a,\ast_a,d_a,M_a)\,:\,a\in A\}$, $A$ being an index set, of $\T^\alpha_\Sp$ 
is relatively compact if and only if for every $r>0$ and $\eps>0$ there exist
\begin{enumerate}[noitemsep]
\item a finite integer $N(r;\eps)$ such that 
\begin{equ}[e:EpsNet]
\sup_a \cN_{d_a}(\ST_a^{(r)},\eps)\leq N(r;\eps)
\end{equ}
where $\cN_{d_a}(\ST_a^{(r)},\eps)$ is the cardinality of the minimal $\eps$-net in $\ST_a^{(r)}$ 
with respect to the metric $d_a$,
\item a finite constant $C=C(r)>0$ and $\delta=\delta(r,\eps)>0$ such that
\begin{equation}\label{e:Equicont}
\sup_{a\in A}\|M_a\|_{\infty}^{(r)}\leq C\qquad\text{and}\qquad\sup_{a\in A}\delta^{-\alpha}\omega^{(r)}(M_a,\delta)<\eps\,,
\end{equation}
\item a finite constant $C'=C'(r)>0$ such that
\begin{equation}\label{e:Comb}
\sup_a b_{\zeta_a}(r)\leq C'\,,
\end{equation} 
\end{enumerate}
\item $\Ch^\alpha_\Sp$ is closed in $\T^\alpha_\Sp$. 
\end{enumerate}
\end{proposition}
\begin{proof}
The first two points were shown in~\cite[Theorem 2.13 and Proposition 2.16]{CHbwt}, 
while the last is a consequence of~\cite[Lemma 2.22]{CHbwt}. 
\end{proof}

An important feature satisfied by any characteristic tree $\zeta=(\ST,\ast,d,M)\in\Ch^\alpha_\Sp$ 
and shown in~\cite[Proposition 2.23]{CHbwt}, 
is that $\ST$ possesses a 
unique open end $\dagger$ with unbounded rays (see Definition~\ref{def:end}) such that 
for every $\sz\in\ST$ and every $\sw\in\llb\sz,\dagger\rangle$, one has
\begin{equation}\label{e:Ray}
M_t(\sw)=M_t(\sz)-d(\sz,\sw)\;.
\end{equation}
It is by virtue of the previous property that we can introduce the {\it radial map} 
which allows to move along rays in the $\R$-tree. 

\begin{definition}\label{def:RadMap}
Let $\alpha\in(0,1]$, $\zeta=(\ST,\ast,d,M)\in\Ch^\alpha_\Sp$ and $\dagger$ the open end with unbounded 
rays such that~\eqref{e:Ray} holds. The {\it radial map} $\rho$ associated to $\zeta$ 
is defined as 
\begin{equation}\label{e:RadMap}
\rho(\sz,s)\eqdef \iota_\sz^{-1}(M_t(\sz)-s)\,,\qquad\text{for $\sz\in\ST$ and $s\in(-\infty, M_t(\sz)]$}
\end{equation}
where $\iota_\sz$ was given in~\eqref{e:iota}. If instead $\zeta\in\hat\Ch^\alpha_\Sp$, 
the radial map $\hat\rho$ satisfies
$\hat\rho(\sz,s)\eqdef \iota_\sz^{-1}(s-M_t(\sz))$, where this time $s\in[M_t(\sz),+\infty)$. 
\end{definition}

\subsection{The tree map}\label{sec:TreeMap}

In this subsection, we introduce a map, the {\it tree map}, which serves as an inverse of the evaluation map. 
Since the evaluation map is not necessarily bijective, we need to determine a way to assign to a point in $\R^2$, 
one in the tree so that certain continuities properties can be deduced from those of the evaluation map. 
We begin with the following definition. 

\begin{definition}\label{def:Right-most}
Let $\alpha\in(0,1)$, $\zeta=(\ST,\ast,d,M)\in\Ch^\alpha_\Sp$ and 
$\rho$ be the radial map associated to $\zeta$ given as in Definition~\ref{def:RadMap}. 
For $z=(t,x)\in\R^2$, we say that 
$\sz$ is a {\it right-most point} for $z$ if $M_t(\sz) = t$ and
\begin{equ}[e:Right-most]
 M_x(\rho(\sz,s))=\sup\{M_x(\rho(\sw,s))\,:\,M_t(\sw) = t\;\&\; M_x(\sw) \le x\}\,,
\end{equ}
for all $s<t$. {\it Left-most points} are defined as in~\eqref{e:Right-most} but replacing $\sup$ with $\inf$ and 
$M_x(\sw) \le x$ with $M_x(\sw) \ge x$. If there is a unique 
right-most (resp. left-most) point we will denote it by $\sz_\rr$ (resp. $\sz_\rl$). 
\end{definition}

\begin{remark}\label{rem:Right-mostUA}
For $\zeta\in\hat\Ch^\alpha_\Sp$, 
a point is said to be a right-most (or left-most) point if~\eqref{e:Right-most} holds for all $s>t$. 
\end{remark}

For $z\in\R^2$ and an arbitrary characteristic tree, right-most and left-most point are not necessarily uniquely defined. 
As we will see below, for elements in the measurable subset of $\Ch^\alpha_\Sp$ given in~\cite[Definition 2.29]{CHbwt}, 
this is indeed the case. 

\begin{definition}\label{def:TreeCond}
Let $\alpha\in(0,1)$. We say that $\zeta=(\ST,\ast,d,M)\in\Ch^\alpha_\Sp$ satisfies the {\it tree condition} 
if 
\begin{enumerate}[label=($\mathfrak{t}$)]
\item\label{i:TreeCond} for all $\sz_1,\sz_2\in\ST$, if $M(\sz_1)=M(\sz_2)=(t,x)$ and there exists $\eps>0$ such that 
$M(\rho(\sz_1,s))=M(\rho(\sz_2,s))$ for all $s\in[t-\eps,t]$, then $\sz_1=\sz_2$.
\end{enumerate}
We denote by $\Ch^\alpha_\Sp(\ft)$, the subset of $\Ch^\alpha_\Sp$ whose elements satisfy~\ref{i:TreeCond}. 
\end{definition}

\begin{lemma}\label{l:Right-most}
Let $\alpha\in(0,1)$ and $\zeta=(\ST,\ast,d,M)\in\Ch^\alpha_\Sp(\ft)$. 
Then, for all $z \in\R^2$ there exist unique left-most and right-most points. 
\end{lemma}
\begin{proof}
Since left-most and right-most points are exchanged under $M_x \mapsto -M_x$, we only need to
consider right-most points.
Let $z=(t,x)\in\R^2$, $\dagger$ be the unique open end such that~\eqref{e:Ray} holds and $\rho$ be 
$\zeta$'s radial map given in~\eqref{e:RadMap}.
Note first that we can assume without loss of generality that $z \in M(\ST)$ since the
right-most points for $z$ equal those of $\bar z = (t, \bar x)$, where
$\bar x = \sup\{y \le x\,:\, (t,y) \in M(\ST)\}$. Since $M$ is proper, it is closed and therefore
$\bar z \in  M(\ST)$.

Let $\{s_n\}_n\subset\R$ be a sequence such that $s_n\ua t$ and 
$A_n\eqdef\{\rho(\sw,s_n)\,:\,\sw\in M^{-1}(z)\}$. $A_n$ is finite since 
$M^{-1}(z)$ is totally disconnected, thanks to~\eqref{e:Ray},  and is compact by properness of the evaluation map. 
In particular this implies that also the number of paths connecting points in $A_n$ with those in $A_{n+1}$ is finite. 
We inductively construct a sequence $\{\sw_n\}_n\in M^{-1}(z)$ as follows. 
Let $\sw_1$ be one of the points for which $M_x(\rho(\sw_1,s))\geq M_x(\rho(\sw,s))$ 
for all $s\in[s_0,s_1]$ and $\sw\in M^{-1}(z)$. 
Assume we picked $\sw_{n}$. If $M_x(\rho(\sw_{n},s))\geq M_x(\rho(\sw,s))$ for all $s\in[s_n,s_{n+1}]$ and 
all $\sw\in M^{-1}(z)$ then set $\sw_{n+1}\eqdef\sw_{n}$. Otherwise choose any $\sw_{n+1}$ so that 
$M_x(\rho(\sw_{n+1},s))$ coincides with the right hand side of~\eqref{e:Right-most} for all $s\in[s_n,s_{n+1}]$. 
Notice that in the first case $d(\sw_{n},\sw_{n+1})=0$. In the other instead, there exists $\bar s\in [s_n,s_{n+1}]$ such that
$M_x(\rho(\sw_n,\bar s)<M_x(\rho(\sw_{n+1},\bar s)$ hence by monotonicity in space 
$M_x(\rho(\sw_n,s)\leq M_x(\rho(\sw_{n+1},s)$ for any $s\leq \bar s$. 
Moreover, for $s\in[s_{n-1},s_n]$, $M_x(\rho(\sw_n,s))\geq M_x(\rho(\sw_{n+1},s))$ by construction, 
and therefore $M_x(\rho(\sw_n,s))= M_x(\rho(\sw_{n+1},s))$ for $s\in[s_{n-1},s_n]$.~\ref{i:TreeCond} then implies that 
$\rho(\sw_n,s_n)= \rho(\sw_{n+1},s_n)$ and we conclude that $d(\sw_{n},\sw_{n+1})\leq2(t-s_n)$. 
Hence, the sequence $\{\sw_n\}_n$ is Cauchy and converges to a unique limit $\sz\in M^{-1}(z)$. Since, for any 
$n$, $d(\sw_{n-1},\sz)\leq2(t-s_n)$ we necessarily have $\rho(\sz,s)=\rho(\sw_{n},s)$ for all $s\in[s_{n-1},s_n]$ which 
implies that $\sz$ is a right-most point. Now, if there existed another one, say $\bar \sz$, then by definition, 
$\rho(\bar\sz,\cdot)\equiv\rho(\sz,\cdot)$ on any subinterval $I\subset (-\infty,t)$ 
therefore, by~\ref{i:TreeCond}, $d(\sz,\bar\sz)=0$. 
\end{proof}

Thanks to the above lemma, we are ready for the following definition.

\begin{definition}\label{def:TreeM}
Let $\alpha\in(0,1)$ and $\Ch^\alpha_\Sp(\ft)$ (resp. $\hat\Ch^\alpha_\Sp(\ft)$) 
be the subset of $\Ch^\alpha_\Sp$ (resp. $\Ch^\alpha_\Sp(\ft)$) whose elements satisfy~\ref{i:TreeCond}. 
For $\zeta\in\Ch^\alpha_\Sp(\ft)$, we define 
the {\it tree map} $\sT$ associated to $\zeta$ as 
\begin{equ}[e:TreeM]
\sT(z)\eqdef \sz_\rr\,,\qquad\text{for all $z\in M(\ST)$}
\end{equ}
where $\sz_\rr$ is the unique right-most point (see Definition~\ref{def:Right-most}). 
\end{definition}

\begin{remark}
In the previous definition we could have analogously picked the left-most point. The choice above was made so that, 
under suitable assumptions on the evaluation map (see Proposition~\ref{p:ContT}), the tree map is c\`adl\`ag. 
\end{remark}

The following proposition  determines the continuity properties of the tree map. 

\begin{proposition}\label{p:ContT}
Let $\alpha\in(0,1)$ and $\zeta\in\Ch^\alpha_\Sp(\ft)$. 
Then, for every $t\in\R$, $x\mapsto \sT(t,x)$ is c\`adl\`ag. 
%
\end{proposition}

Before proving the previous, we state and show the following lemma which contains more precise information 
regarding the roles of left-most and right-most point in the continuity properties of the tree map.

\begin{lemma}\label{l:ContT}
Let $\alpha\in(0,1)$ and $\zeta\in\Ch^\alpha_\Sp(\ft)$. 
Let $z=(t,x)\in\R^2\cap M(\ST)$ and assume there exists a sequence $\{z_n=(t,x_n)\}_n\subset\R^2\cap M(\ST)$ 
converging to it. If 
\begin{enumerate}[noitemsep]
\item $x_n\da x$ then $\lim_n \sT(z_n)=\sT(z)$,
\item $x_n\ua x$ then $\lim_n \sT(z_n)$ exists and coincides with $\sz_\rl$,
\end{enumerate}
\end{lemma}
\begin{proof}
Let $\zeta=(\ST,\ast,d,M)\in\Ch^\alpha_\Sp$. We will prove only 1. since the other can be shown similarly. 
Let $z$ and $\{z_n\}_n$ be as in the statement. 
Since $M$ is proper and continuous, the sequence $\{\sz^n_\rr\eqdef \sT(z_n)\}_n$ converges 
along subsequences and any limit point is necessarily in $M^{-1}(z)$. 
But now, monotonicity in space and~\eqref{e:Ray} imply that 
\begin{equ}
d(\sz,\sz_\rr^n)= d(\sz,\sz_\rr)\vee d(\sz_\rr,\sz^n_\rr)
\end{equ}
hence, by uniqueness of $\sz_\rr$, the result follows. 
\end{proof}

\begin{proof}[of Proposition~\ref{p:ContT}]
Let $\zeta=(\ST,\ast,d,M)\in\Ch^\alpha_\Sp$ and recall that by definition $M$ is both proper and continuous. 
Assume first that $z=(t,x)\notin M(\ST)$. Then, 
there exists $\eps>0$ such that $\{t\}\times[x-\eps,x+\eps]\cap M(\ST)=\emptyset$. 
Hence, all $w\in \{t\}\times[x-\eps,x+\eps]$ have the same right-most point so that $\sT$ is constant there. 

If $z=(t,x)\in M(\ST)$ and, for some $\eps>0$, $\{t\}\times(x,x+\eps)\cap M(\ST)=\emptyset$ then $\sT$ 
is constantly equal to $\sT(z)$ on $\{t\}\times[x,x+\eps)$ and therefore it is continuous from the right. 
If $z=(t,x)\in M(\ST)$ and, for some $\eps>0$, $\{t\}\times(x-\eps,x)\cap M(\ST)=\emptyset$, since by property~\ref{i:Spread}
of characteristic trees $M^{-1}(\{t\}\times[x-\eps-2,x-\eps])\neq\emptyset$ and $M^{-1}(\{t\}\times[x-\eps-2,x-\eps])$ 
is closed, there exists $\bar z\eqdef \sup\{w\in \{t\}\times[x-\eps-2,x-\eps])\cap M(\ST)\}$. But then, 
$\sT$ is constantly equal to $\sT(\bar z)$ on $\{t\}\times(x-\eps,x)$ which implies that $\lim_{y\ua x}\sT(t,y)$ exists. 

The case of $z$ being an accumulation point in $\{t\}\times[x-1,x+1]\cap M(\ST)$ was covered in Lemma~\ref{l:ContT}, 
so that the proof is concluded. 
\end{proof}

\begin{remark}\label{rem:ContT}
Notice that in view of the proof of Proposition~\ref{p:ContT} and Lemma~\ref{l:ContT}, for  
$\zeta=(\ST,\ast,d,M)\in\Ch^\alpha_\Sp$, $\sT$ is continuous {\it only} at those points $z$ such that the cardinality of 
$M^{-1}(z)$ is less or equal to $1$. 
\end{remark}

\subsection{The (Double) Brownian Web tree}\label{sec:BW}

As pointed out in the introduction, a major role in the definition of the Brownian Castle 
is played by the Brownian Web and, in particular its characterisation as a 
characteristic spatial $\R$-tree. 
In this subsection, we recall the main statements in~\cite{CHbwt}, 
where such a characterisation was derived, focusing 
on those results which are instrumental to the present paper. 
We begin with the following theorem (see~\cite[Theorem 3.8 and Remark 3.9]{CHbwt}) 
which establishes the existence of 
the Brownian Web Tree and uniquely identifies its law in $\Ch^\alpha_\Sp$. 

\begin{theorem}\label{thm:BW}
Let $\alpha<\tfrac12$. There exists a $\Ch^\alpha_\Sp$-valued random variable 
$\zeta^\da_\bw=(\ST^\da_\bw,\ast^\da_\bw,d^\da_\bw,M^\da_\bw)$ with radial map $\rho^\da$, 
whose law is uniquely characterised by the following properties 
\begin{enumerate}
\item for any deterministic point $w=(s,y)\in\R^2$ there exists almost surely a unique point 
$\sw\in\ST^\da_\bw$ such that $M^{\da}_\bw(\sw)=w$, 
\item for any deterministic $n\in\N$ and $w_1=(s_1,y_1),\dots,w_n=(s_n,y_n)\in\R^2$, the joint distribution 
of $( M^{\da}_{\bw,x}(\rho^\da(\sw_i,\cdot)))_{i=1,\dots,n}$,  
where $\sw_1,\dots,\sw_n$
are the points determined in 1., 
is that of $n$ coalescing backward Brownian motions starting at $w_1,\dots, w_n$, 
\item for any deterministic countable dense set $\cD$ such that $0\in\cD$, let $\sw$ be the point determined in 1.
associated to $w\in\cD$ and $\tilde\ast^\da$ that associated to $0$. Define
$\tilde\zeta_\infty^\da(\cD)=(\tilde\ST_\infty^\da(\cD),\tilde\ast^\da,d^\da,\tilde M_\infty^{\,\da,\cD})$ as
\begin{equation}\label{e:Coupling}
\begin{split}
\tilde\ST_\infty^\da(\cD)&\eqdef\{\rho^\da(\sw,t)\,:\,w=(s,y)\in\cD'\,,\,t\leq s\}\\
\tilde M_\infty^{\,\da,\cD}(\rho^\da(\sw,t))&\eqdef M_\bw(\rho^\da(\sw,t))
\end{split}
\end{equation}
and $d^\da$ to be the ancestral metric in~\eqref{e:AncestralD}. 
Let $\tilde\ST^\da(\cD)$ be the completion of $\tilde\ST_\infty^\da(\cD)$ under $d^\da$,  
$\tilde M^{\,\da,\cD}$ be the unique little $\alpha$-H\"older continuous extension of $\tilde M_\infty^{\,\da,\cD}$ and 
$\tilde\zeta^\da(\cD)\eqdef(\tilde\ST^\da(\cD),\ast^\da,d^\da, \tilde M^{\,\da,\cD})$. Then,  
$\tilde \zeta^\da(\cD)\eqlaw\zeta^\da_\bw$.
\end{enumerate}
The same statement holds upon taking 
the periodic version of all objects and spaces above and 
replacing the properties $1.$-$3.$ with $1_\per.$, $2_\per.$ and $3_\per.$ 
obtained from the former by adding the word 
``periodic'' before any instance of ``Brownian motion''. 
\end{theorem}

Thanks to the previous result, we can define the (periodic) Brownian Web Tree. 

\begin{definition}\label{def:BW}
Let $\alpha<\tfrac12$. We define  {\it backward Brownian Web 
Tree} and {\it periodic backward Brownian Web tree}, the $\Ch^\alpha_\Sp$ and $\Ch^\alpha_{\Sp,\per}$ 
random variables  $\zeta^\da_\bw=(\ST^\da_\bw,\ast^\da_\bw,d^\da_\bw,M^\da_\bw)$ and 
$\zeta^{\per,\da}_\bw=(\ST_\bw^{\per,\da},\ast_\bw^{\per,\da}, d_\bw^{\per,\da},M_\bw^{\per,\da})$ 
whose distributions is uniquely characterised by properties $1.$-$3.$ and 
$1_\per.$, $2_\per.$ and $3_\per.$ in Theorem~\ref{thm:BW}. 
\end{definition}

The following proposition states some important quantitative and qualitative properties of both the Brownian 
Web Tree and its periodic counterpart. 

\begin{proposition}\label{p:BW}
Let $\alpha<\tfrac12$ and $\zeta^\da_\bw$ be the Brownian Web Tree given as in Definition~\ref{def:BW}. 
Then, almost surely, for any fixed $\theta>\tfrac32$ and all $r>0$ there exists a constant $c=c(r)>0$ 
depending only on $r$ such that for all $\eps>0$ 
\begin{equ}[e:nMEC]
\cN_{d_\bw^\da}(\ST_\bw^{\da,\,(r)}, \eps)\leq c \eps^{-\theta}
\end{equ}
where $\cN_{d_\bw^\da}(\ST^{\da,\,(r)}, \eps)$ is defined as in~\eqref{e:EpsNet}. 
Moreover, almost surely $M^\da_\bw$ is surjective and~\ref{i:TreeCond} holds for $\zeta^\da_\bw$. 

All the properties above remain true in the periodic case. 
\end{proposition}
\begin{proof}
See~\cite[Proposition 3.2 and Theorem 3.8]{CHbwt}. 
\end{proof}

A key aspect of the Brownian Web Tree is that it comes naturally associated with a dual, 
which consists of a spatial $\R$-tree whose rays, when embedded into $\R^2$, 
are distributed as a family of coalescing {\it forward} Brownian motions. 
In the following theorem and the subsequent definition (see~\cite[Theorem 3.1, Remark 3.18 and Definition 3.19]{CHbwt}), 
we introduce the (periodic) Double Brownian Web Tree, a random couple of characteristic $\R$-trees 
made of the Brownian Web Tree and its dual, and clarify the relation between the two. 

\begin{theorem}\label{thm:DBW}
Let $\alpha<1/2$. There exists a $\Ch^\alpha_\Sp\times\hat\Ch^\alpha_\Sp$-valued random variable 
$\zeta^{\uda}_\bw\eqdef(\zeta^\da_\bw,\zeta^\ua_\bw)$, 
$\zeta^{\dotp}_\bw=(\ST^{\dotp}_\bw,\ast^{\dotp}_\bw,d^{\dotp}_\bw,M^{\dotp}_\bw)$, $\dotp\in\{\da,\ua\}$, whose 
law is uniquely characterised by the following properties
\begin{enumerate}[label=(\roman*)]
\item\label{i:Dist} Both $-\zeta^\ua_\bw\eqdef (\ST^\ua_\bw,\ast^\ua_\bw,d^\ua_\bw,-M^\ua_\bw)$ and $\zeta^\da_\bw$ are 
distributed as the backward Brownian Web tree in Definition~\ref{def:BW}.
\item\label{i:Cross} Almost surely, for any $\sz^\da\in\ST^\da_\bw$ and $\sz^\ua\in\ST^\ua_\bw$, the paths 
$M^\da_{\bw}(\rho^\da(\sz^\da,\cdot))$ and $M^\ua_{\bw}(\rho^\ua(\sz^\ua,\cdot))$ do not cross, 
where $\rho^\da$ (resp. $\rho^\ua$) is the radial map of $\zeta^\da_\bw$ (resp. $\zeta^\ua_\bw$). 
\end{enumerate}
Moreover, almost surely $\zeta^{\uda}_\bw\in\Ch^\alpha_\Sp(\ft)\times\hat\Ch^\alpha_\Sp(\ft)$ and 
$\zeta^\ua_\bw$ is determined by $\zeta^\da_\bw$ and vice-versa, meaning that 
the conditional law of $\zeta^\ua_\bw$ given $\zeta^\da_\bw$ is almost surely given 
by a Dirac mass. 

The above statement remains true in the periodic setting upon replacing every object and space 
with their periodic counterpart. 
\end{theorem}

\begin{definition}\label{def:DBW}
Let $\alpha<\tfrac12$. We define the {\it double Brownian Web tree} and {\it double periodic Brownian Web tree} 
as the $\Ch^\alpha_\Sp\times\hat\Ch^\alpha_\Sp$ and $\Ch^\alpha_{\Sp,\per}\times\hat\Ch^\alpha_{\Sp,\per}$-valued 
random variables $\zeta^{\uda}_\bw\eqdef(\zeta^{\da}_\bw,\zeta^{\ua}_\bw)$ and 
$\zeta^{\per,\uda}_\bw\eqdef(\zeta^{\per,\da}_\bw, \zeta^{\per,\ua}_\bw)$ given by 
Theorem~\ref{thm:DBW}. 
We will refer to $\zeta^\ua_\bw$ and $ \zeta^{\per,\ua}_\bw$ as the {\it forward} (or dual) and {\it forward periodic,
Brownian Web trees}.

We  denote their laws
by $\Theta^{\uda}_\bw(\dd(\zeta^\da\times\zeta^\ua))$ and $\Theta^{\per,\uda}_\bw(\dd(\zeta^\da\times\zeta^\ua))$, 
 with marginals $\Theta^{\da}_\bw(\dd\zeta)$, $\Theta^{\ua}_\bw(\dd\zeta)$ and 
$\Theta^{\per,\da}_\bw(\dd\zeta)$, $\Theta^{\per,\ua}_\bw(\dd\zeta)$ respectively. 
\end{definition}

As we will see in the next section, the continuity properties of the Brownian Castle are crucially 
connected to the cardinality and the degree of points of the inverse maps 
$(M^{\dotp}_\bw)^{-1}$ and $(M^{\per,\dotp}_\bw)^{-1}$, for $\dotp\in\{\ua,\da\}$. 
Based on these two features, it is possible to classify all points of $\R^2$ or $\R\times\T$ 
as was shown in~\cite[Theorem 3.24]{CHbwt} and we now summarise this classification. 

\begin{definition}\label{def:Type}
Let $\zeta_\bw^{\uda}=(\zeta^\da_\bw,\zeta^\ua_\bw)$ be the double Brownian Web tree. 
For $\dotp\in\{\ua,\da\}$, the type of a point $z\in\R^2$ for $\zeta^{\dotp}_\bw$ is $(i,j)\in\N^2$, where 
\begin{equ}
i=\sum_{i=1}^{|(M^{\dotp}_\bw)^{-1}(z)|}(\deg(\sz_i^{\dotp})-1)\quad\text{and}\quad j=|(M^{\dotp}_\bw)^{-1}(z)|\,.
\end{equ} 
Above, $\{\sz^{\dotp}_i\,:\,i\in\{1,\dots, |(M^{\dotp}_\bw)^{-1}(z)|\}\}=(M^{\dotp}_\bw)^{-1}(z)$. 
We define $S^\da_{i,j}$ (resp. $S^\ua_{i,j}$) as the subset of $\R^2$ containing all points of type $(i,j)$
for the forward (resp. backward) Brownian Web tree. 
For the periodic Brownian Web $\zeta_\bw^{\per,\uda}=(\zeta^{\per,\ua}_\bw,\zeta^{\per,\da}_\bw)$, 
the definition is the same as above and the set of all of points in $\R\times\T$ of type $(i,j)$ 
for the backward (resp. forward) 
periodic Brownian Web tree, will be denoted by 
$S^{\per,\da}_{i,j}$ (resp. $S^{\per,\ua}_{i,j}$).
\end{definition}

\begin{theorem}\label{thm:Types}
For the backward and backward periodic Brownian Web trees $\zeta^\da_\bw$ and $\zeta^{\da,\per}_\bw$, 
almost surely, every $z\in\R^2$ (resp. $\R\times\T$) is of one of the following types, 
all of which occur: $(0,1),\,(1,1),\,(2,1),\,(0,2),\,(1,2)$ and $(0,3)$. 
Moreover, almost surely, $S^\da_{0,1}$ has full Lebesgue measure, $S^\da_{2,1}$ and $S^\da_{0,3}$ 
are countable and dense and for every $t\in\R$
\begin{itemize}[noitemsep,label=-]
\item $S^\da_{0,1}\cap\{t\}\times\R$ has full Lebesgue measure in $\{t\}\times\R$,
\item
$S^\da_{1,1}\cap\{t\}\times\R$ and $S^\da_{0,2}\cap\{t\}\times\R$ are both countable 
and dense in $\{t\}\times\R$,
\item $S^\da_{2,1}\cap\{t\}\times\R$, $S^\da_{1,2}\cap\{t\}\times\R$ 
and $S^\da_{0,3}\cap\{t\}\times\R$ have each cardinality at most $1$. 
\end{itemize}
At last, for every deterministic $t$, $S^\da_{2,1}\cap\{t\}\times\R$, $S^\da_{1,2}\cap\{t\}\times\R$ 
and $S^\da_{0,3}\cap\{t\}\times\R$ are almost surely empty. 

Moreover, $S^\da_{(i,j)}=S^\ua_{(i',j')}$ for $(i,j)/(i',j')=(0,1)/(0,1)$, $(1,1)/(0.2)$, $(2,1)/(0,3)$, $(0,2)/(1,1)$, 
$(1,2)/(1,2)$ and $(0.3)/(0,3)$, and the same holds in the periodic case. 
\end{theorem}
\begin{proof}
The first part of the statement corresponds to~\cite[Theorem 3.24]{CHbwt}, while the last 
is a consequence of~\cite[Proposition 3.22]{CHbwt}. 
\end{proof}

\section{Branching Spatial Trees and the Brownian Castle measure}\label{sec:BST}

As mentioned in the introduction, the Brownian castle is not just given by an $\R$-tree $\ST$ 
realised on $\R^2$ by a map $M$, but furthermore comes with a stochastic process $X$ indexed by $\ST$, 
whose distribution depends on the metric structure of $\ST$. (In our case, we want $X$ to be a Brownian motion
indexed by $\ST$.)
How to do it in such such a way that $X$ admits a (H\"older) continuous modification is the topic of the second 
section but first, we want to introduce the space in which such an object lives.

\subsection{Branching Spatial $\R$-trees}\label{sec:bsp}

The space of branching spatial $\R$-trees corresponds to spatial $\R$-trees endowed 
with an additional (H\"older) continuous map, from the tree to $\R$, which, for us, 
will encode a realisation of a suitable stochastic process. The term {\it branching} is chosen because 
this extra map (read, process) should be thought of as {\it branching} at the points in which the branches of the tree coalesce. 

\begin{definition}\label{def:BPST}
Let $\alpha,\beta\in(0,1)$. The space of {\it $(\alpha,\beta)$-branching spatial pointed $\R$-trees} 
$\T^{\alpha,\beta}_\bsp$ is the set of couples $\chi=(\zeta, X)$ with $\zeta\in\T^\alpha_\Sp$ and 
$X\colon \ST\to \R$, the {\it branching map}, is a locally little $\beta$-H\"older continuous map.
In $\T^{\alpha,\beta}_\bsp$ two couples 
$\chi=(\zeta,X),\,\chi'=(\zeta',X')$ are indistinguishable if there exists a bijective 
isometry $\varphi:\ST\to\ST'$ such that $\varphi(\ast)=\ast'$, $M'\circ\varphi\equiv M$ 
and $X'\circ\varphi\equiv X$, in short $\varphi\circ\chi=\chi'$. 

If furthermore $\zeta\in\Ch^\alpha_\Sp$, we say that $\chi=(\zeta, X)$ is a {\it characteristic $(\alpha,\beta)$-branching 
spatial $\R$-tree} and denote the subset of $\T^{\alpha,\beta}_\bsp$ containing such $\R$-trees 
by $\Ch^{\alpha,\beta}_\bsp$. 
\end{definition}

Similar to what done in Section~\ref{sec:SpTrees}, we endow the space $\T^{\alpha,\beta}_\bsp$ with the metric 
\begin{equation}\label{def:bspMetric}
\begin{split}
\Delta_\bsp(\chi,\chi')&\eqdef\int_0^{+\infty} e^{-r}\, \big[1\wedge \uDelta_\bsp^\com(\chi^{(r)},\chi'^{(r)})\big]\dd r+d_\M(b_\zeta,b_{\zeta'})\\
&=:\uDelta_\bsp(\chi,\chi')+d_\M(b_\zeta,b_{\zeta'})\,.
\end{split}
\end{equation} 
for $\chi,\,\chi'\in\T^{\alpha,\beta}_\bsp$. 
Above, given $r>0$, $\chi^{(r)}$, $\chi'^{\,(r)}$ are defined as in~\eqref{e:rRestr} and 
\begin{equ}
\uDelta^\com_\bsp(\chi^{(r)},\chi'^{\,(r)})\eqdef\inf_{\CC:(\ast,\ast')\in\CC} \Delta_\bsp^{\com,\CC}(\chi^{(r)},\chi'^{\,(r)})\;,
\end{equ}
the infimum being taken over all correspondences $\CC\subset\ST^{(r)}\times\ST'^{(r)}$ such that $(\ast,\ast')\in\CC$, 
and for such a correspondence $\CC$
\begin{equation}\label{e:MetbC}
\begin{split}
\uDelta^{\com,\CC}_\bsp(\chi^{(r)},\chi'^{\,(r)})\eqdef& \uDelta^{\com,\CC}_\Sp(\zeta^{(r)},\zeta'^{\,(r)})+\sup_{(\sz,\sz')\in\CC}|X(\sz)-X'(\sz')|\\
&+\sup_{n\in\N}2^{n\beta}\sup_{\substack{(\sz,\sz'),(\sw,\sw')\in\CC\\d(\sz,\sw),d'(\sz',\sw')\in\CA_n}}
|\delta_{\sz,\sw}X-\delta_{\sz',\sw'}X'|\,.
\end{split}
\end{equation}

The following lemma determines the metric properties of $\Ch^{\alpha,\beta}_\bsp$ and gives a 
compactness criterion for its subsets. The proof, as well as the statement, are completely 
analogous to those for $\T^{\alpha}_\Sp$ given in~\cite[Theorem 2.13 and Proposition 2.16]{CHbwt}, 
so we refer the reader to the aforementioned reference for further details. 

\begin{lemma}\label{l:Comp}
For $\alpha,\beta\in(0,1)$, $(\T^{\alpha,\beta}_\bsp,\Delta_\bsp)$ is a complete separable metric space and 
$\Ch^{\alpha,\beta}_\bsp$ is closed in it. 
Moreover, a subset $\CA$ of $\T^{\alpha,\beta}_\bsp$ is relatively compact if and only if 
\begin{enumerate}[noitemsep]
\item the projection of $\CA$ onto $\T^\alpha_\Sp$ is relatively compact and
\item for every $r>0$ and $\eps > 0$ there exist constants $K=K(r)>0$ and $\delta=\delta(r,\eps)>0$ such that
\begin{equation}\label{e:Equicont2}
\sup_{\chi\in\CA}\|X_\chi\|^{(r)}_{\infty}\leq K\qquad\text{and}\qquad\sup_{\zeta\in\CA}\delta^{-\beta}\omega^{(r)}(X_\chi,\delta)<\eps\,.
\end{equation}
\end{enumerate}
\end{lemma}

The following lemma will be needed in some of the proofs below. It gives a way to estimate the distance between 
a branching spatial $\R$-tree and one of its subset, provided the Hausdorff distance 
(see~\eqref{e:Hausdorff}) between the metric spaces is known. 

\begin{lemma}\label{l:Approx}
Let $\alpha,\,\beta\in(0,1)$, $\chi=(\ST,\ast,d,M,X)\in\T_\bsp^{\alpha,\beta}$ and assume $\ST$ is compact. 
Let $\delta>0$, $T \subset \ST$ be such that $\ast \in T$ and 
the Hausdorff distance between $T$ and $\ST$ is bounded above 
by $\delta$ and define $\bar\chi=(T,\ast,d,M \restr T, X \restr T)$. Then 
\begin{equ}[e:Approx]
\uDelta_\bsp^\com(\chi,\bar\chi)\lesssim (2\delta)^{-\alpha}\omega(M,2\delta)+ (2\delta)^{-\beta}\omega(X,2\delta)
\end{equ}
\end{lemma}
\begin{proof}
The proof follows the very same steps of~\cite[Lemma 2.12]{CHbwt}. 
\end{proof}



\subsection{Stochastic Processes on trees}\label{sec:MEC}

We want to understand how to realise the branching map $X$ as a {\it H\"older continuous} real-valued
stochastic process indexed by a pointed locally compact complete $\R$-tree $(\ST,d,\ast)$, 
whose covariance structure is suitably related to the metric $d$. 
In this section, we will mostly consider the case of a {\it fixed} (but generic) $\R$-tree, 
and provide conditions for the process $X$ to admit a $\beta$-H\"older continuous modification, for some $\beta\in(0,1)$. 
We will always view the law of the process $X$ on a given $\ST$, $(\ST,\ast,d,M)\in\T^\alpha_\Sp$, as a 
probability measure on the space branching spatial $\R$-trees $\T^{\alpha,\beta}_\bsp$.

A function $\varphi:\R_+\to\R_+$ is said to be a {\it Young function} if it is convex, increasing and such that 
$\lim_{t\to\infty}\varphi(t)=+\infty$ and $\varphi(0)=0$. 
From now on, fix a standard probability space $(\Omega, \CA,\Prob)$. Given $p\geq0$ and a Young function $\varphi$,
the Orlicz space on $\Omega$ associated to $\varphi$ is
the set $L^\varphi$ of random variables $Z \colon \Omega \to \R$ such that
\begin{equation}\label{def:Orlicz}
\|Z\|_\varphi\eqdef\inf\{c>0\,:\,\Exp[\varphi(|Z|/c)]\leq 1\}<\infty\,.
\end{equation}
Notice that if $Z$ is a positive random variable such that $\|Z\|_{L^\varphi}\leq C$ for some Young function $\varphi$ 
and some finite $C>0$, 
then Markov's inequality yields 
\begin{equ}[e:Tail]
\P(Z>u)\leq 1 / \varphi(u/C)\;,\quad \forall \,u > 0\;.
\end{equ}
The following proposition shows that if $\varphi$ is of exponential type $\phi(x) = e^{x^q}-1$ 
and $\cN_{d}(\ST,\eps)$
grows at most polynomially as $\eps \to 0$, then be obtain a modulus of continuity of order 
$d |\log d|^{1/q}$.

\begin{proposition}\label{p:Holder}
Let $q\geq 1$ and $\varphi_q(x)\eqdef e^{x^q}-1$. 
Let $(\ST,\ast,d)$ be a pointed complete proper metric space 
and $\{X(\sz):\sz\in\ST\}$ a stochastic process indexed by $\ST$. Assume there are 
$\nu \in (0,1]$, $\theta >0$ and, for every $r>0$, 
there exists a constant $c >0$ such that 
\begin{equ}[e:nMEC]
\cN_{d}(\ST^{(r)},\eps)\leq c\,\eps^{-\theta}\;,\qquad \text{for all $\eps\in(0,1)$}
\end{equ}
and 
\begin{equ}[e:Holder1]
\|X(\sz)-X(\sz')\|_{\varphi_q}\leq c\, d(\sz,\sz')^\nu\;,\qquad \text{for all $\sz,\sz' \in \ST^{(r)}$,}
\end{equ}
where $\ST^{(r)}$ is defined as in~\eqref{e:rRestr}. 
Then, $X$ admits a continuous version such that for every $r>0$, there exists a random variable $K=K(\omega,r)$ such that
\begin{equ}
|X(\sz)-X(\sz')|\leq K \,d(\sz,\sz')^\nu \,|\log d(\sz,\sz')|^{1/q}\;,\quad \text{for all $\sz,\sz' \in \ST^{(r)}$.}
\end{equ}
Furthermore, one has the bound $\P(K \ge u) \le C_1 \exp(-C_2 u^q)$ for some constants $C_i > 0$ depending only $r$,
$\nu$, $c$ and $\theta$.
\end{proposition}

\begin{proof}
We closely follow the argument and notations of \cite[Sec.~1.2]{Talagrand}. We also note that, 
since $d^\nu$ is again a metric, it suffices to consider the case $\nu = 1$, so we do this now.
Set $\eps_n = 2^{-n}$ and 
$N_n = \cN_d(\ST,\eps_n)$. Noting
that in our case $d(\pi_n(\sz), \pi_{n+1}(\sz)) \le 2\eps_n$, \eqref{e:Tail} yields
\begin{equ}
\P(|X(\pi_n(\sz))-X(\pi_{n+1}(\sz))| \ge u n^{1/q} \eps_n) \lesssim \exp(-c u^q n)\;,
\end{equ}
for some constant $c > 0$. Note furthermore that in our case $N_n N_{n-1}
\lesssim 2^{2\theta n}$ by assumption. Proceeding similarly to \cite[Sec.~1.2]{Talagrand},
we consider the event $\Omega_u$ on which 
$|X(\pi_n(\sz))-X(\pi_{n+1}(\sz))| \le u n^{1/q} \eps_n$ for every $n \ge 0$, so that
\begin{equ}
\P(\Omega_u^c) \lesssim \sum_{n \ge 0} 2^{2\theta n} \exp(-c u^q n) \lesssim \exp(-c u^q)\;.
\end{equ}
Furthermore, on $\Omega_u$, one has
\begin{equs}
|X(\sz)-X(\sz')| &\le  \sum_{n \ge n_0} |X(\pi_n(\sz))-X(\pi_{n+1}(\sz))| + \sum_{n > n_0} |X(\pi_n(\sz'))-X(\pi_{n+1}(\sz'))| \\
&\qquad  + |X(\pi_{n_0}(\sz)) - X(\pi_{n_0+1}(\sz'))|\\
&\le 4 u \sum_{n \ge n_0} n^{1/q} \eps_n \approx u\, d(\sz,\sz') |\log d(\sz,\sz')|^{1/q}\;,
\end{equs}
where $n_0$ is such that $\eps_{n_0+1} \le d(\sz,\sz') \le \eps_{n_0}$. The claim then follows at once.
\end{proof}

Condition~\eqref{e:nMEC} on the size of the $\eps$-nets will always be met by the $\R$-trees we will consider, 
that this is the case for the Brownian Web Tree is stated in Proposition~\ref{p:BW}. 
Therefore, we introduce a subset of the space of  characteristic spatial trees whose elements satisfy it locally uniformly. 
Let $\alpha\in(0,1)$, $\theta>0$ and let $c \colon \R_+ \to \R_+$ be an increasing function. 
We define $\tilde\sE_\alpha(c,\theta)$, $\tilde\sE(\theta)$ and $\tilde\sE_\alpha$ respectively as
\begin{equs}[def:MeasSet]
\tilde\sE_\alpha(c,\theta)&\eqdef \{\zeta=(\ST,\ast,d,M)\in \T^\alpha_\Sp\,:\,\forall r,\eps>0\,,\,\,\cN_d(\ST^{(r)},\eps)\leq c(r)\eps^{-\theta}\}\,,\\
\tilde\sE_\alpha(\theta)&\eqdef \bigcup_{c} \tilde\sE_\alpha(c,\theta)\qquad\text{and}\qquad \tilde\sE_\alpha\eqdef \bigcup_\theta\tilde\sE_\alpha(\theta)\,.
\end{equs}
and $\sE_\alpha(c,\theta)\eqdef \tilde\sE_\alpha(c,\theta)\cap\Ch^\alpha_\Sp$ and $\sE_\alpha(\theta),\,\sE_\alpha$ accordingly. 
Thanks to~\cite[Proposition 7.4.12]{BBI}, it is not difficult to verify that for every given $c$ and $\theta$ as above 
$\tilde\sE_\alpha(c,\theta)$  is 
closed in $\T^\alpha_\Sp$, and consequently $\tilde\sE_\alpha(\theta)$ and $\tilde\sE_\alpha$ are measurable with respect 
to the Borel $\sigma$-algebra induced by the metric $\Delta_\Sp$ (and so are $\sE_\alpha(\theta)$ and $\sE_\alpha$).

In the next proposition we show how~\eqref{e:Holder1} can be used to prove tightness for the laws, 
on the space of branching spatial 
$\R$-trees, of a family of stochastic processes 
indexed by different spatial $\R$-trees that uniformly belong to $\tilde\sE_\alpha(c,\theta)$, for some $\theta$ and $c$. 

\begin{proposition}\label{p:TightnessTree}
Let $q\geq 1$, $\alpha\in(0,1)$ and let $\sK \subset \tilde\sE_\alpha(c,\theta)$ for some $c,\,\theta>0$ be relatively compact. 
For every $\zeta = (\ST,\ast,d,M)\in\sK$, let $X_\zeta$ be a stochastic process indexed by $\ST$ 
and denote by $\CQ_{\zeta}$ the law of $(\zeta,X_\zeta)$.
Assume that there exists $\nu\in(0,1)$ such that for all $\eps>0$ and $r>0$, there are constants 
$c_\infty=c_\infty(\eps)>0$ and $c_\nu=c_\nu(r)>0$ such that 
\begin{equ}\label{e:Infty-tightness}
\inf_{\zeta\in\sK}\CQ_{\zeta}\left(|X_\zeta(\ast)|\leq c_\infty\right)\geq 1-\eps
\end{equ}  
and 
\begin{equation}\label{e:Holder-tightness}
\| X_\zeta(\sz)-X_\zeta(\sw)\|_{\varphi_q}\leq c_\nu d(\sz,\sw)^\nu\,,\qquad\text{for all $\sz,\sw\in\ST^{(r)}$ and $\zeta \in \sK$.}
\end{equation}
Then, the family of probability measures $\{\CQ_{\zeta}\}_{\zeta \in \sK}$ is 
tight in $\T_\bsp^{\alpha,\beta}$ for any $\beta<\nu$. 
\end{proposition}
\begin{proof}
By Lemma~\ref{l:Comp}, we only need to focus on the maps $X_\zeta$ and, more specifically, on their restriction to the 
$r$-neighbourhoods of $\ast$. 
 Since $\{X_\zeta(\ast)\}_{\zeta\in \sK}$ is tight by~\eqref{e:Infty-tightness}, it remains to argue that for every $\eps>0$
\begin{equation}\label{e:TightCond}
\lim_{\delta\to 0}\sup_{\zeta \in \sK} \CQ_{\zeta}\Big( \delta^{-\beta}\sup_{d(\sz,\sw)\leq\delta} |X_\zeta(\sz)-X_\zeta(\sw)|>\eps\Big)=0\,.
\end{equation}
This in turn is immediate from Proposition~\ref{p:Holder} and our assumption.
\end{proof}

Our main example is that of a Brownian motion indexed by a pointed $\R$-tree.
Let $(\ST,\ast,d)$ be a pointed locally compact complete $\R$-tree, and let $\{B(\sz)\,:\,\sz\in\ST\}$ 
be the centred Gaussian process such that $B(\ast)\eqdef 0$ and such that
\begin{equation}\label{e:GP}
\E[(B(\sz)-B(\sz'))^2] = d(\sz,\sz')\;,
\end{equation}
for all $\sz,\,\sz'\in\ST$. We call $B$ the Brownian motion on $\ST$.

\begin{remark}\label{rem:GPExistence}
The existence of a Gaussian process whose covariance matrix is as above is guaranteed by the fact that any 
$\R$-tree $(\ST,d)$ is of strictly negative type, see~\cite[Cor.~7.2]{HLM}.  
\end{remark}

\begin{remark}\label{rem:contBM}
If $\zeta = (\ST,d,\ast,M)\in\tilde\sE_\alpha$,
then it follows from Proposition~\ref{p:Holder} that, for $B$ a Brownian motion on $\ST$,
one has $(\zeta,B) \in \T^{\alpha,\beta}_\bsp$ almost surely, for every $\beta < \beta_\Gau\eqdef1/2$.
We will denote by $\CQ_\zeta^\Gau$ its law on $\T^{\alpha,\beta}_\bsp$. 
\end{remark}

In the study of $0$-Ballistic Deposition, we will also consider Poisson processes indexed by an $\R$-tree. 
The Poisson process 
is clearly not continuous so, in order to fit it in our framework, we introduce a smoothened version of it.
First, recall that for any locally compact complete $\R$-tree $\ST$, the skeleton of $\ST$, $\ST^o$, is defined 
as the subset of $\ST$ obtained by removing all its endpoints, i.e. 
\begin{equation}\label{def:Skeleton}
\ST^o\eqdef \bigcup_{\sz\in\ST}\llb\ast,\sz\llb\,.
\end{equation}
For any $\ST$ as above, there exists a unique $\sigma$-finite measure $\ell=\ell_\ST$, called the {\it length measure}
such that $\ell(\ST\setminus\ST^o)=0$ and 
\begin{equation}\label{def:LengthMeasure}
\ell\big(\llb\sz,\sz'\rrb\big)=d(\sz,\sz')\,,
\end{equation}
for all $\sz,\sz'\in\ST$.

\begin{remark}\label{rem:disjoint}
One important property of the Brownian motion $B$ is that, given any four points $\sz_i$, one has
\begin{equ}
\E\big[(B(\sz_1)-B(\sz_2))(B(\sz_3)-B(\sz_4))\big] = \ell \bigl(\llb\sz_1,\sz_2\rrb \cap \llb\sz_3,\sz_4\rrb\bigr)\;,
\end{equ}
where the equality follows immediately by polarisation and the tree structure of $\ST$. 
In particular, increments of $B$ are independent on any two disjoint
subtrees of $\ST$.
\end{remark}

Let $\alpha\in(0,1)$, $\zeta=(\ST,\ast,d,M)\in\Ch^\alpha_\Sp$ a characteristic tree 
and $\dagger$ the open end for which~\eqref{e:Ray} holds. 
For $\gamma>0$, let $\mu_\gamma$ be the Poisson random measure on $\ST$ with intensity $\gamma\ell$ and, for $a>0$, 
let $\psi_a$ be a smooth non-negative real-valued function on $\R$, 
compactly supported in $[0,a]$ and such that $\int_\R \psi_a(x)\,\dd x=1$. 
We define the smoothened Poisson random measure on $\ST$ as
\begin{equation}\label{def:PRM}
\mu^{a}_\gamma(\sw)\eqdef \int_{\llb\sw,\dagger\rangle}\psi_a(d(\sw,\bar\sz))\mu_\gamma(\dd\bar\sz)\,,
\qquad\sw\in\ST\,.
\end{equation}
In other words, we are smoothening the Poisson random measure $\mu_\gamma$ by fattening its points along the rays 
from the endpoints to the open end $\dagger$, so that
the value at a given point $\sw$ depends only on Poisson points in 
$\llb\sw, \sw(a)\rrb$, where $\sw(a)$ is the unique point on the ray $\llb\sw,\dagger\rangle$ such that 
$d(\sw,\sw(a))=a$.
%

\begin{definition}\label{def:Poisson}
Let $\alpha\in(0,1)$ and $\zeta=(\ST,\ast,d,M)\in \Ch^\alpha_\Sp$ with length measure 
$\ell$. Let $\gamma>0$ and $\mu_\gamma$ be the Poisson random measure on 
$\ST$ with intensity measure $\gamma \ell$. For $a>0$, let $\psi_a$ be a smooth non-negative real-valued function on $\R$, 
compactly supported in $[0,a]$ and such that $\int_\R \psi_a(x)\,\dd x=1$. 
We define the {\it rescaled compensated smoothened} (RCS in short) {\it Poisson process on $\ST$} as
\begin{equation}\label{def:RCSPPtree}
N^{a}_\gamma(\sz)\eqdef {1\over \sqrt \gamma} \int_{\llb\ast,\sz\rrb}\bigl(\mu^{a}_\gamma(\sw) - \gamma\bigr)\ell(\dd \sw)\;,\qquad\text{for $\sz\in\ST$,}
\end{equation}
where $\mu^a_\gamma$ is the smoothened Poisson random measure given in~\eqref{def:PRM}. 
\end{definition}

In case the $\R$-tree has (locally) finitely many endpoints and $\gamma$ and $a$ are fixed, it is easy 
to see that the smoothened Poisson process defined above is Lipschitz. 
That said, we want to obtain more quantitative information about its regularity and how the latter relates to the parameters
$\gamma$ and $a$, in order to be able to identify a regime in which a family of 
RCS Poisson processes on a given tree converges weakly. 

In the following Lemma (and the rest of the paper), for $\zeta=(\ST,\ast,d,M)\in\sE_\alpha$ and 
$X$ a stochastic process indexed by $\ST$, which admits a $\beta$-H\"older 
continuous modification, we will denote by $\CQ_\zeta(\dd X)$ 
the law of $(\zeta,X)$ in the space of $(\alpha,\beta)$-characteristic 
branching spatial $\R$-trees and by $\cM(\Ch^{\alpha,\beta}_\bsp)$ the space of probability measures 
on $\Ch_\bsp^{\alpha,\beta}$ endowed with the topology of weak convergence.

\begin{lemma}\label{l:PPtree}
Let $\alpha\in(0,1)$, $\zeta=(\ST,\ast,d,M)\in\sE_\alpha$, $N^{a}_\gamma$ be as in Definition~\ref{def:Poisson} 
and set $\beta_\Poi\eqdef\frac{1}{2p}$. 
If $a=\gamma^{-p}$ for some $p>1$, then for any $\beta<\beta_\Poi$, $(\zeta,N^{a}_\gamma)\in\Ch^{\alpha,\beta}_\bsp$ almost surely. Furthermore, denoting its law by $\CQ_\zeta^{\Poi_\gamma}$, these are tight over $\gamma \ge 1$.
%
\end{lemma}
\begin{proof}
Thanks to Propositions~\ref{p:Holder} and~\ref{p:TightnessTree}, 
it suffices to show that there exists a constant $C$ depending only on $r$ such that 
\begin{equation}\label{e:IncrementPP}
\| N^{a}_\gamma(\sz)-N^{a}_\gamma(\sw)\|_{\varphi_1}\leq C d(\sz,\sw)^{\frac{1}{2p}}\,,
\end{equation}
for all $\sz,\sw\in \ST^{(r)}$, which in turn is essentially a consequence of 
Lemma~\ref{l:SmoothPoisson} in the appendix. Indeed, if the points 
$\sz,\sw$ belong to the same ray, then the increment $N^{\psi_a}_\gamma(\sz)-N^{\psi_a}_\gamma(\sw)$
coincides in distribution with that of $P^{a}_\gamma( d(\sz,\sw))$ in~\eqref{def:SmoothRCPP}, so 
that~\eqref{e:IncrementPP} follows from~\eqref{b:SmoothRCPP}. 

If $\sz$ and $\sw$ lie on different branches, let $\sz_\dagger$ be the unique point for which 
$\llb\sz,\dagger\rangle\cap\llb\sw,\dagger\rangle=\llb\sz_\dagger,\dagger\rangle$. Then, the triangle
inequality for Orlicz norms yields
\begin{equ}
\| N^{a}_\gamma(\sz)-N^{a}_\gamma(\sw)\|_{\varphi_1}
\le 
\| N^{a}_\gamma(\sz)-N^{a}_\gamma(\sz_\dagger)\|_{\varphi_1} + \| N^{a}_\gamma(\sz_\dagger)-N^{a}_\gamma(\sw)\|_{\varphi_1}\;,
\end{equ}
and the required bound follows from \eqref{e:IncrementPP}.
%
\end{proof}

\subsection{Probability measures on the space of branching spatial trees}\label{sec:MECc}

The results in the previous section identify suitable conditions on a spatial $\R$-tree
$\zeta=(\ST,\ast,d,M)$ and the distribution of the increments of a stochastic process $X$ indexed by $\ST$, under which 
the couple $(\zeta,X)$ is (almost surely) a branching spatial $\R$-tree. 
We now want to let $\zeta$ vary and understand the behaviour of the map $\zeta\mapsto \CQ_\zeta = \Law(\zeta,X)$. 
Since in rest of the paper we will only deal with characteristic $\R$-trees, we will directly work with 
these, even though some of our statements remain true for general spatial $\R$-trees.
We now write $\CQ_\zeta^{\gamma,p} = \Law(\zeta,X)$ for $X = N_\gamma^a$ as in \eqref{def:RCSPPtree} 
with the choice $a = \gamma^{-p}$, and $\CQ_\zeta^{\Gau}$ for $X = B$ as in \eqref{e:GP}.  
We then have the following continuity property.

\begin{proposition}\label{p:MeasG}
Let $\alpha\in(0,1)$, $c,\theta>0$, and $\CQ_\zeta = \CQ_\zeta^\Gau$, respectively $\CQ_\zeta = \CQ_\zeta^{\gamma,p}$
for some $\gamma > 0$ and $p > 1$. Then, the map 
\begin{equ}[e:MapTreetoPro]
\sE_\alpha(c,\theta)\ni\zeta\mapsto\CQ_\zeta\in\cM(\Ch^{\alpha,\beta}_\bsp)
\end{equ}
is continuous, provided that 
$\beta<\beta_\Gau = \f12$, respectively $\beta < \beta_{\Poi}$ as in Lemma~\ref{l:PPtree}. 
\end{proposition}

In view of Lemma~\ref{l:PPtree}, Proposition~\ref{p:TightnessTree}, and the central limit theorem, 
it is clear that, for a {\it fixed} tree $\zeta$, $\CQ_\zeta^{\gamma,p}$ converges weakly to $\CQ^\Gau_\zeta$ 
as $\gamma\ua\infty$, for any fixed $p$. 
In the next statement, we show that such a convergence is locally uniform in $\zeta$. 
%
%

\begin{proposition}\label{p:PPtoGau}
Let $\alpha\in(0,1)$, $p>1$, and $c,\theta>0$. Then, for any $\beta<\beta_\Poi$, 
$\lim_{\gamma\to\infty}\CQ_\zeta^{\gamma,p}=\CQ_\zeta^\Gau$ in $\cM(\Ch^{\alpha,\beta}_\bsp)$,
uniformly over compact subsets of $\fE_\alpha(c,\theta)$. 
\end{proposition}

The remainder of this subsection is devoted to the proof of the previous two statements. 
We will first focus on Proposition~\ref{p:MeasG} since some tools from its proof will be needed in that 
of Proposition~\ref{p:PPtoGau}. Before delving into the details we need to make
some preliminary considerations which apply to both. 

Let $c,\,\theta>0$ be fixed and $\sK$ be a compact subset of $\sE_\alpha(c,\theta)$. 
Since the constant $C$ in~\eqref{e:IncrementPP} is independent of both $\gamma$ and the 
specific features of the tree,
Proposition~\ref{p:TightnessTree} implies that the families
$\{\CQ^{\gamma,p}_\zeta\,:\gamma>0,\,\zeta\in \sK\}$ and $\{\CQ^\Gau_\zeta\,:\,\zeta\in \sK\}$ are tight in 
$\Ch^{\alpha,\beta}_\bsp$ for any $\beta<\beta_{\Poi}$ and $\beta<\beta_\Gau$ respectively, 
and jointly, for $\beta<\beta_{\Poi}\wedge\beta_{\Gau}$. 

Then, the proof of both Propositions~\ref{p:MeasG} and~\ref{p:PPtoGau} boils down to show 
that if $\{\zeta_n\}_n\subset \sK$ is a sequence 
converging to $\zeta$ with respect to $\Delta_\Sp$ then there exists a coupling between 
$(\zeta_n, X_n)$ and $(\zeta,X)$ such that $\Delta_\bsp((\zeta_n,X_n),(\zeta,X))$ converges to $0$ 
locally uniformly over $\zeta\in \sK$ and the H\"older norms of $X_n$ and $X$. 
If we denote by $\CQ_n$ and $\CQ_\zeta$ the laws of $(\zeta_n, X_n)$ and $(\zeta,X)$ then for the first statement, 
we need to pick $\CQ_n$ and $\CQ_\zeta$ to be either 
$\CQ^{\gamma,p}_{\zeta_n}$ and $\CQ^{\gamma,p}_{\zeta}$, for $\gamma>0$ fixed, 
or $\CQ^{\Gau}_{\zeta_n}$ and $\CQ^\Gau_{\zeta}$, while for the second, $\zeta_n=\zeta$ for all $n$, 
$\CQ_n=\CQ^{\gamma_n,p}_\zeta$, with $\gamma_n\to\infty$ and $\CQ_\zeta=\CQ^\Gau_\zeta$. 

The problem in the first case is that, 
since $X$ and $X_n$ are indexed by different spaces, it is not {\it a priori} clear how to build these couplings. 
In the next subsection, which represents the core of the proof, we construct one.
%
%
%
%
%
%
%
%
%
%
%
%
%
%
%
%

\subsubsection{Coupling processes on different trees}\label{sec:Coupling}

Fix $\alpha\in(0,1)$ and $\eta > 0$, and consider characteristic trees 
$\zeta=(\ST,\ast,d,M),\,\zeta'=(\ST',\ast',d',M')\in\Ch^\alpha_\Sp$ such that $\ST$ and $\ST'$
are \textit{compact}. 
As a shorthand, we set $\delta = \uDelta^{\com}_\Sp(\zeta^{(r)},\zeta'^{\,(r)})$
and we fix a correspondence $\CC$ between $\ST^{(r)}$ and $\ST'^{\,(r)}$ such that
\begin{equ}[e:InitialBound]
\uDelta^{\com,\CC}_\Sp(\zeta^{(r)},\zeta'^{\,(r)})\le 2\delta\,.
\end{equ}
We will always assume that our two trees are sufficiently close so that $\delta \le \eta$.
(We typically think of the case $\delta \ll \eta$.)
Let then $B$ be the Gaussian process on $\ST$ such that~\eqref{e:GP} holds and, for $\gamma>0$, let 
$\mu_\gamma$ be the Poisson random measure on $\ST$ with intensity $\gamma\ell$ and $N^a_\gamma$ be 
the RCS Poisson process of Definition~\ref{def:Poisson} and Lemma~\ref{l:PPtree}. 

The aim of this subsection is to first inductively construct subtrees $T$ and $T'$ of $\ST$ and $\ST'$ 
respectively that
are close to each other and whose distance from the original trees is easily quantifiable. Simultaneously, 
we build a bijection $\phi \colon T \to T'$ which preserves the length measure and has small distortion. 
This provides a natural coupling between $\mu_\gamma$ and a Poisson random measure $\mu_\gamma'$ on
$T'$ by $\mu_\gamma'(A) = \phi^*\mu_\gamma(A)\eqdef\mu_\gamma(\phi^{-1}(A))$, and similarly for the white noise
underlying $B$.

To start our inductive construction, we simply set
\begin{equ}
T_0 = \{\ast\}\;,\qquad T_0' = \{\ast'\}\;,\qquad \phi(\ast) = \ast'\;.
\end{equ}
Assume now that, for some $m\in\N$, we are given subtrees $T_{m-1}$ and $T_{m-1}'$ as well as a length-preserving
measure $\phi\colon T_{m-1}\to T_{m-1}'$.
Let then $\sv \in\ST \setminus T_{m-1}$ be a point whose distance from $T_{m-1}$ is maximal and denote by 
$\fb_m$ the projection of $\sv$ onto $T_{m-1}$. We also set
\begin{equ}
\CC_{m-1}\eqdef \CC \cup \phi_{m-1} = \CC \cup \{(\sz,\phi(\sz))\,:\, \sz \in T_{m-1}\}\;,
\end{equ}
where we have identified the bijection $\phi_{m-1} = \phi\restr T_{m-1}$ with the natural
correspondence induced by it.
If $d(\sv,\fb_m)\leq 2(\eta \vee \dis\CC_{m-1}')$, we terminate our construction and set
\minilab{e:construction}
\begin{equ}[e:approxTree]
T\eqdef T_{m-1}\;,\quad T'\eqdef T_{m-1}'\;,\quad  
Z\eqdef (T,\ast,d, M\restr T)\;,\quad Z'\eqdef (T',\ast',d', M'\restr T')\;.
\end{equ}

Otherwise, let $\sv'\in\ST'$ be such that $(\sv,\sv')\in\CC$ and $\fb_m'$ be its 
projection onto $T_{m-1}'$. If $d(\sv,\fb_m)\geq  d'(\sv',\fb_m')$ then we set $\sv_m' = \sv'$ and denote by
$\sv_m \in \llb\fb_m,\sv\rrb$ the unique point such that $d(\sv_m,\fb_m)=  d'(\sv_m',\fb_m')$. Otherwise,
we set  $\sv_m = \sv$ and define $\sv_m'$ correspondingly. 
We then set
\minilab{e:construction}
\begin{equ}[e:approxTree2]
T_m\eqdef T_{m-1}\cup\llb \fb_m,\sv_m\rrb\,,\qquad T_m'\eqdef T_{m-1}'\cup\llb \fb_m',\sv_m'\rrb\;,
\end{equ}
and we extend $\varphi$ to $\llb \fb_m,\sv_m\rrb \setminus \{\fb_m\}$ to be the unique isometry such that 
$\varphi(\sv_m)=\sv_m'$. We also write $\ell_m = d(\fb_m,\sv_m) = d'(\fb_m',\sv_m')$.
The following shows that this construction terminates after finitely 
many steps.

\begin{lemma}\label{l:GrowingT'}
Let $N$ be the minimal number of balls of radius $\eta/8$ required to cover $\ST$. 
Then, the  construction described above terminates after at most $N$ steps and, 
until it does, one has $\ell_m \ge \eta/2$ so that in particular $\sv_m'\notin T_{m-1}'$. 
\end{lemma}
\begin{proof}
We start by showing the second claim. Assuming the construction has not terminated yet, we 
only need to consider the case $d'(\sv',\fb_m') < d(\sv,\fb_m)$ so that $\sv_m' = \sv'$. Take $j<m$ such that 
$\fb_m'\in\llb\fb_j',\sv_j'\rrb$, then
\begin{equs}
\ell_m &= d'(\sv_m',\fb_m')=\frac{1}{2}(d'(\sv_m',\sv_j')+d'(\sv_m',\fb_j')-d'(\sv_j',\fb_j'))\\
&\geq \frac{1}{2}(d(\sv,\sv_j)+d(\sv,\fb_j)-d(\sv_j,\fb_j))-\frac{3}{2}\dis\CC_{m-1}' \label{e:PointOut}\\
&\geq d(\sv,\fb_m)-\frac{3}{2}\dis\CC_{m-1}'\ge {\eta\over 2}\;.
\end{equs}
The passage from the first to the second line is a consequence of the fact that 
$(\sv,\sv_m'),\,(\sv_j,\sv_j'),\,(\fb_j,\fb_j')\in\CC_{m-1}$, and the last bound follows
from the fact that  $d(\sv,\fb_m) \ge \f32\dis\CC_{m-1} + \f\eta2$ by assumption.

It remains to note that since $\ell_m > \eta/2$, the points $\sv_j$ are all at distance at least $\eta/2$ 
from each other, so there can only be at most $N$ of them.
\end{proof}

Thanks to Lemma~\ref{l:GrowingT'}, we can also define
\begin{align}
B'(\sz')-B'(\fb_m')&\eqdef B(\varphi_m^{-1}(\sz'))-B(\fb_m)\,,\qquad\text{for $\sz'\in\llb \fb_m',\sv_m'\rrb$}\label{e:Gauss}\\
\mu_\gamma'\restr \llb \fb_m',\sv_m'\rrb&\eqdef \varphi_m^*\mu_\gamma\,,\label{e:Poisson}
\end{align}
and $N'^{\,a}_\gamma$ accordingly. 
In the following lemma, we denote by $X$ and $X'$ the processes on $T$ and $T'$, which correspond to
either $B$ and $B'$ or to $N^a_\gamma$ and $N'^{\,a}_\gamma$. 

\begin{lemma}\label{l:Coupling}
In the setting above,  let $Z,\,Z'$ be as in~\eqref{e:approxTree} and $X$ and $X'$ be constructed inductively via~\eqref{e:Gauss} 
or~\eqref{e:Poisson}. If~\eqref{e:InitialBound} holds, then $(Z,X)$ and $(Z',X')$ belong to $\Ch^{\alpha,\beta}_\bsp$ 
and
\begin{equ}[b:Coupling]
\uDelta^{\com,\phi}_\bsp((Z,X),(Z',X'))\lesssim (5^{N}\delta)^{-\alpha}\omega(M, 5^{N}\delta)+N5^{N\kappa}\|X\|_{\beta+\kappa} \delta^\kappa\,.
\end{equ}
for any $\kappa\in(0,1)$ sufficiently small, where $N$ is as in Lemma~\ref{l:GrowingT'}.
Moreover, the Hausdorff distance between both $T$ and $\ST$, and $T'$ and $\ST'$, is bounded above by 
$2( 5^{N}\delta) + \eta$. 
\end{lemma}
\begin{proof}
We will first bound the distortion $\dis\phi$ between $T$ and $T'$ by induction on $m$. 
We begin by showing that $\dis(\CC_{m-1}\cup\{(\sv_m,\sv_m'),(\fb_m,\fb_m')\})\leq\f72\dis\CC_{m-1}$. 
Assume without loss of generality 
that $d(\sv,\fb_m)>  d'(\sv',\fb_m')$ and let $(\sw,\sw')\in\CC_{m-1}$.  
By the triangle inequality and the fact that, by construction, $(\sv,\sv_m'),\,(\fb_m,\phi_{m-1}(\fb_m))\in\CC_{m-1}$ 
we have
\begin{equs}
|d(\sv_m,\sw)-d'(\sv'_m,\sw')|&\leq d(\sv,\sv_m)+\dis\CC_{m-1}\,,\\
|d(\fb_m,\sw)-d'(\fb'_m,\sw')|&\leq \dis\CC_{m-1}+d'(\phi_{m-1}(\fb_m),\fb_m')\,,
\end{equs}
where we added and subtracted $d(\sv_m,\sw)$ in the first and $d'(\phi_{m-1}(\fb_m),\sw')$ in the second. 
Now, as a consequence of~\eqref{e:PointOut} 
\begin{equ}
d(\sv,\sv_m)=d(\sv,\fb_m)-d(\sv_m,\fb_m) =d(\sv,\fb_m)- d'(\sv',\fb_m')<\tfrac{3}{2}\dis \CC_{m-1}
\end{equ}
while 
\begin{equs}
d'(\phi_{m-1}&(\fb_m),\fb_m')=d'(\sv_m',\phi_{m-1}(\fb_m))-d'(\sv_m',\fb'_m)\\
&\leq \dis \CC_{m-1}+d(\sv,\fb_m)-d(\sv_m,\fb_m)= \dis \CC_{m-1}+d(\sv,\sv_m)\leq \tfrac{5}{2}\dis \CC_{m-1}
\end{equs}
and both hold since, by construction, $d(\sv_m,\fb_m)=d'(\sv_m',\fb_m')$. 
Therefore the claim $\dis(\CC_{m-1}\cup\{(\sv_m,\sv_m'),(\fb_m,\fb_m')\})\leq\f72\dis\CC_{m-1}$, follows at once. 
Now, for any $\sz\in T_m\setminus T_{m-1}$, let $\tilde\sz\in\ST$ be such that $(\tilde \sz,\phi_m(\sz))\in\CC_{m-1}$. 
We clearly have 
\begin{equ}
|d(\sz,\sw)-d'(\phi_m(\sz),\sw')|\leq d(\sz,\tilde\sz)+\dis \CC_{m-1}\,.
\end{equ}
Denote by $\tilde\fb\in T_m$ the projection of $\tilde\sz$ onto $T_m$. In order to bound $d(\sz,\tilde\sz)$, 
it suffices to exploit the fact that if $\tilde\fb\in\llb\sz,\sv_m\rrb$, then $d(\sz,\tilde\sz)=d(\fb_m,\tilde\sz)-d(\fb_m,\sz)$, 
while if $\tilde\fb\in T_m\setminus\llb\sz,\sv_m\rrb$, then $d(\sz,\tilde\sz)=d(\sv_m,\tilde\sz)-d(\sv_m,\sz)$. 
Let $(\sy,\sy')$ be either $(\sv_m,\sv_m')$ or $(\fb_m,\fb_m')$, so that
\begin{equs}
d(\sz,\tilde\sz)=d(\sy,\tilde\sz)-d(\sy,\sz)&\leq d'(\sy',\phi_m(\sz))+\dis(\CC_{m-1}\cup\{(\sv_m,\sv_m'),(\fb_m,\fb_m')\})-d(\sy,\sz)\\
&=\dis(\CC_{m-1}\cup\{(\sv_m,\sv_m'),(\fb_m,\fb_m')\})\leq \tfrac{7}{2}\dis\CC_{m-1}
\end{equs}
where we used that, by construction, $d'(\sy',\phi_m(\sz))=d(\sy,\sz)$. Hence, 
$\dis\phi_m\leq\dis\CC_m\leq \tfrac92 \dis\CC_{m-1}$ and since $\dis\CC_0=\dis\CC$, 
we conclude that $\dis\phi_m\leq \dis\CC_m\leq (\tfrac92)^m\dis\CC$ 
and therefore $\dis\phi\leq \dis(\CC\cup\phi)\lesssim 5^N\delta$. 
%
%
%
%
%

Concerning the evaluation maps, take $\sz_1, \sz_2 \in T$ and choose 
$\sw_1, \sw_2\in\ST$ such that $(\sw_i,\phi(\sz_i)) \in\CC$. One has
\begin{equation}\label{e:Points}
d(\sz_i,\sw_i)=|d(\sz_i,\sw_i)-d'(\phi(\sz_i),\phi(\sz_i))|\leq \dis (\phi\cup\CC)\lesssim 5^{N}\delta\;,
\end{equation}
so that
\begin{equs}[e:InftyBoundFinal]
\|M(\sz_1)-M'(\phi(\sz_1))\|&\leq \|\delta_{\sz_1,\sw_1} M\|+\|M(\sw_1)-M'(\phi(\sz_1))\|\\
&\leq  \omega(M,5^N\delta) + 2\delta\;,
\end{equs}
where we used the little H\"older continuity of $M$ and~\eqref{e:InitialBound}. 
For the H\"older part of the distance instead, let $n\in\N$ and assume further that 
$d(\sz_1,\sz_2),\,d(\phi(\sz_1),\phi(\sz_2))\in\cA_n$. 
If there exist no $\sw_1,\,\sw_2\in\ST$ such that $(\sw_i,\phi(\sz_i)) \in\CC$ and also $d(\sw_1,\sw_2)\in\cA_n$, then 
for $n\in\N$ such that $2^{-n}>5^N\delta$, we exploit~\eqref{e:InftyBoundFinal} to obtain
\begin{equ}
\|\delta_{\sz_1,\sz_2}M-\delta_{\phi(\sz_1),\phi(\sz_2)}M'\|\lesssim 2^{-n\alpha}((5^N\delta)^{-\alpha}\omega(M, 5^N\delta))
\end{equ}
while for $n$ such that $2^{-n}\leq 5^N\delta$, by definition of $\Delta^\com_\Sp$ (see below~\eqref{e:MetC}), we get
\begin{equ}
\|\delta_{\sz_1,\sz_2}M-\delta_{\phi(\sz_1),\phi(\sz_2)}M'\|\leq \|\delta_{\sz_1,\sz_2}M\|+\|\delta_{\phi(\sz_1),\phi(\sz_2)}M'\|\lesssim 2^{-n\alpha}(2^{n\alpha}\omega(M, 2^{-n})+\delta)
\end{equ}
which in turn is bounded by $2^{-n\alpha}((5^N\delta)^{-\alpha}\omega(M, 5^N\delta))$. 
In case instead  $d(\sw_1,\sw_2)\in\cA_n$, then we simply apply the triangle inequality to write
\begin{equ}
\|\delta_{\sz_1,\sz_2}M-\delta_{\phi(\sz_1),\phi(\sz_2)}M'\|\leq \|\delta_{\sz_1,\sz_2}M-\delta_{\sw_1,\sw_2}M\|+\|\delta_{\sw_1,\sw_2}M-\delta_{\phi(\sz_1),\phi(\sz_2)}M'\|
\end{equ}
Thanks to the estimate on the distortion of $\CC\cup\phi$ and~\eqref{e:InftyBoundFinal}, 
the first summand can be controlled as in the proof 
of~\cite[Lemma 2.12]{CHbwt}, while the second is bounded by $2\delta$ thanks to~\eqref{e:InitialBound}.

We focus now on the branching maps, for which we proceed once again by induction on the iteration step $m$. 
Notice that for $m=0$, there is nothing to prove. 
We now assume that for some $m<N$, there exists $K,\,K'>0$ such that 
\begin{gather}
\sup_{(\sz,\sz')\in\phi_{m-1}}|X(\sz)-X'(\sz')|\leq K \delta^\rho\label{e:IndInfty}\\
\|X'\restr T'_{m-1}\|_\rho\leq K'\|X\|_\rho\label{e:HolderNorm}
\end{gather}
for $\rho<\beta$, but arbitrarily close to it. Let $(\sz,\sz')\in\phi_m\setminus\phi_{m-1}$ 
be such that $\sz\in\llb\fb_m,\sv_m\rrb$ and $\sw$ be the point in $T_{m-1}$ for which 
$(\sw,\fb_{m}')\in\CC_{m-1}$. Then, 
\begin{equs}
|X(\sz)-X'(\sz')|&=|X(\fb_{m})-X'(\fb'_{m})|\leq |\delta_{\fb_m,\sw}X|+|X(\sw)-X'(\fb'_{m})|\\
&\leq \|X\|_\varrho d(\fb_{m},\sw)^\varrho+K\delta^\varrho\leq (5^{N\rho}\|X\|_\varrho+K)\delta^\varrho
\end{equs}
where the passage from the first to the second line is a consequence of the H\"older regularity of $X$ and~\eqref{e:IndInfty} 
while in the last inequality we exploited the fact that both $(\fb_m,\fb_m')$ and $(\sw,\fb_m')\in\phi_m$ and the same 
bound as in~\eqref{e:Points}. 
Concerning the H\"older norm of $X'$, let $\sz'\in T_m'\setminus T_{m-1}'$ and $\sw'\in T_{m-1}'$, then by triangle inequality 
we have 
\begin{equ}
|\delta_{\sz',\sw'}X'|\leq |\delta_{\sz',\fb_m'}X'|+|\delta_{\fb_m',\sw'}X'|\leq \|X\|_\rho(1+K') d'(\sz',\sw')^\rho
\end{equ}
where we used~\eqref{e:HolderNorm}. Hence, the $\rho$-H\"older norm of $X'$ on $T'$ is bounded above by $N\|X\|_\rho$.

For the second summand in~\eqref{e:MetbC}, let $(\sz,\sz'),(\sw,\sw')\in\phi$ be such that $d(\sz,\sw),d'(\sz',\sw')\in\CA_n$. 
Then, we have 
\begin{equ}
|\delta_{\sz,\sw}X-\delta_{\sz',\sw'}X'|\lesssim 2^{-n\rho}(\|X\|_\rho+\|X'\|_\rho)\lesssim 2^{-n\rho} N \|X\|_\rho
\end{equ}
as well as
\begin{equ}
|\delta_{\sz,\sw}X-\delta_{\sz',\sw'}X'|\lesssim N5^{N\rho}\|X\|_\varrho \delta^\rho\,.
\end{equ}
hence, by geometric interpolation,~\eqref{b:Coupling} follows.

For the last part of the statement, notice that the Hausdorff distance between $\ST$ and $T$ 
is bounded above by $2(\eta \vee \dis\CC') \lesssim \eta + 5^N \delta$ by the definition of our halting condition. 
Concerning the Hausdorff distance between $\ST'$ and $T'$, let $\sz'\in\ST'$. 
Take $\sz\in\ST$ such that $(\sz,\sz')\in\CC$, $\sw\in T$ such that $d(\sz,\sw)\lesssim \eta + 5^N \delta$, which 
exists since $d_H(\ST, T)\lesssim \eta + 5^N \delta$, and $\sw'\in T'$ such that $(\sw,\sw')\in\phi$. Then
\begin{equ}
d'(\sz',\sw')\leq|d'(\sz',\sw')-d(\sz,\sw)|+ d(\sz,\sw)\lesssim \dis(\CC\cup\phi)+\eta + 5^N \delta
\end{equ}
from which the claim follows at once.
\end{proof}

We are now ready for the proof of Propositions~\ref{p:MeasG} and~\ref{p:PPtoGau}. 

\begin{proof}[of Proposition~\ref{p:MeasG}]
Let $\{\zeta_n\}_n\,,\zeta\subset\sE_\alpha(c,\theta)$ be such that $\Delta_\Sp(\zeta_n,\zeta)$ converges to $0$. 
Let $\eps>0$ be fixed. Since the family $\{\CQ_n\}_n$ is tight, there exists $\cK_\eps \subset \Ch^{\alpha,\beta}_\bsp$ compact such that 
\begin{equ}[e:TightFamily]
\inf_{n\in\N} \CQ_n(\cK_\eps)\geq 1-\eps\,.
\end{equ}
We want to bound the Wasserstein distance between $\CQ_n$ and $\CQ$. By the previous we have 
\begin{equs}[e:Wass1]
\cW(\CQ_{n},\CQ_{\zeta})&=\inf_{g\in\Gamma(\CQ_{\zeta_n},\CQ_{\zeta})}\E_g\left[\Delta_\bsp((\zeta_n,X_n),(\zeta,X))\right]\\
&\leq \inf_{g\in\Gamma(\CQ_{\zeta_n},\CQ_{\zeta})}\E_g\left[\Delta_\bsp((\zeta_n,X_n),(\zeta,X))\1_{\cK_\eps}\right]+\eps\\
&\leq  \inf_{g\in\Gamma(\CQ_{\zeta_n},\CQ_{\zeta})}\E_g\left[\Delta^\com_\bsp((\zeta_n^{(r)},X_n^{(r)}),(\zeta^{(r)},X^{(r)}))\1_{\cK_\eps}\right]+2\eps\;,
\end{equs}
where in the last passage we used the definition of metric $\Delta_\bsp$ in~\eqref{def:bspMetric} and chose $r$ so that 
$e^{-r}<\eps$. We now apply the construction \eqref{e:construction} with $\zeta$ replaced by 
$\zeta^{(r)}$ and $\zeta'$ replaced by $\zeta_n^{(r)}$. 
The triangle inequality then yields 
\begin{align}
\Delta^\com_\bsp&((\zeta_n^{(r)},X_n^{(r)}),(\zeta^{(r)},X^{(r)}))\leq \Delta^\com_\bsp((Z_n,X_n^{(r)}),(Z,X^{(r)}))\label{e:Summ1}\\
&+\Delta^\com_\bsp((\zeta_n^{(r)},X_n^{(r)}),(Z_n,X_n^{(r)}))+\Delta^\com_\bsp((\zeta^{(r)},X^{(r)}),(Z,X^{(r)}))\;.\label{e:Summ2}
\end{align}
Since the summands in~\eqref{e:Summ2} only depend on one of the two probability measures, their 
coupling is irrelevant. By the last point of Lemma~\ref{l:Coupling}, the Hausdorff distance of both $\ST_n^{(r)}$ and 
$\ST^{(r)}$ from $T'$ and $T$ is at most of order  $\eta + 5^N\delta$ so that we can apply Lemma~\ref{l:Approx}. 
If we now choose first $\eta \approx \eps$ 
small enough, then we can guarantee that each of the two terms in~\eqref{e:Summ2} is less than $\eps$, 
provided that $n$ 
is sufficiently large (and therefore $\delta$ sufficiently small). Note here that even though $N$ 
depends (badly) on $\eta$, it is independent of $n$.

At last, upon choosing at most an even smaller $\delta$ and exploiting the coupling of the previous section, 
we can use~\eqref{b:Coupling} to control~\eqref{e:Summ1} by $\eps$, 
so that the proof is concluded. 
\end{proof}

\begin{proof}[of Proposition~\ref{p:PPtoGau}] 
Throughout the proof we will write $\CQ^\gamma_\zeta$ for $\CQ^{\gamma,p}_\zeta$ 
and $\CQ_\zeta$ for $\CQ^\Gau_\zeta$ and we fix a compact set $\sK \subset \fE_\alpha(c,\theta)$. 

Let $\eps>0$ be fixed. Since the family $\{\CQ^{\gamma}_\zeta,\,\CQ_\zeta\,:\gamma>0,\,\zeta\in \sK\}$ 
is tight in $\Ch^{\alpha,\beta}_\bsp$ for any $\beta<\beta_{\Poi}$, there exists $\cK_\eps\subset\Ch^{\alpha,\beta}_\bsp$ compact such that~\eqref{e:TightFamily} holds with the infimum taken over $\zeta\in \sK$ and $\gamma>0$. 
Then, proceeding as in~\eqref{e:Wass1} and following the strategy used to control~\eqref{e:Summ2} 
in the previous proof, we see that we can choose $r,\,\eta>0$ in such a way that
\begin{equ}[e:WassDistCoup]
\sup_{\zeta\in \sK}\cW(\CQ^\gamma_\zeta,\CQ_\zeta)\leq\sup_{\zeta\in \sK} \inf_{g\in\Gamma(\CQ^\gamma_{\zeta},\CQ_{\zeta})}\E_g\left[\Delta^\com_\bsp((Z,X_\gamma^{(r)}),(Z,X^{(r)}))\1_{\cK_\eps}\right]+ 4\eps\;,
\end{equ}
where $Z = Z_\zeta$ is again constructed from $\zeta$ as in \eqref{e:construction}.
We are left to determine a coupling under which the first term is small. 
Let $W$ be a standard Brownian motion and $P_\gamma$ 
a rescaled compensated Poisson process of intensity $\gamma$ on $\R_+$, coupled in such a way that, with
probability at least $1-\eps$, one has
\begin{equ}[e:dlambda]
\sup_{t\in [0,L]} |W(t)-P_\gamma(t)| \le (1+L) \gamma^{-1/5} \eqdef d_\gamma\;,
\end{equ}
provided that $\gamma$ is sufficiently small. 
Here $L\eqdef \sup_{\zeta\in \sK}\ell_\zeta(T)$ with $\ell_\zeta$ the length measure on $\ST$
is finite by Lemma~\ref{l:GrowingT'} and the existence of
such a coupling (with $1/5$ replaced by any exponent less than $1/4$) is guaranteed by the quantitative
form of Donsker's invariance principle. Similarly, we have
$\sup_{\zeta\in K} \#\{\sz\in T_\zeta\,:\,\deg(\sz)=1\}<\infty$ and we will denote its value by $N$. 
For $\zeta\in \sK$, we order the endpoints and the respective branching points 
of $T$ according to the procedure of Section~\ref{sec:Coupling} and recursively define the subtrees $T_j$, 
$j\leq N$, of $T$ as in~\eqref{e:approxTree}. For every $j\leq N$, denote by $\phi_j$ the unique bijective isometry 
from $\llb\fb_j,\sv_j\rrb$ to $[\ell_\zeta( T_j), \ell_\zeta(T_j)+d(\fb_j,\sv_j)]$ with $\phi_j(\fb_j) = \ell_\zeta( T_j)$. 

Given any function $Y \colon [0,L] \to \R$, we then define $\tilde Y \colon T \to \R$ to be the unique function
such that $\tilde Y(\ast) = 0$ and
\begin{equ}[e:defYtilde]
\tilde Y(\sz) = 
Y(\phi_j(\sz))-Y(\phi_{j}(\fb_{j}))+\tilde Y(\fb_j)\;,\quad \text{for $\sz\in T_j\setminus T_{j-1}$. }
\end{equ}
This then allows us to construct the desired coupling by setting 
$B = \tilde W$ and $N_\gamma = \tilde P_\gamma$, as well as
$N^a_\gamma$ to be the smoothened version of $N_\gamma$. 
It follows from \eqref{e:defYtilde} that, setting $\delta_j = \sup_{\sz \in T_j} |B(\sz)-N_\gamma(\sz)|$,  we have 
$\delta_{j+1} \le \delta_j + 2 d_\gamma$ on the event \eqref{e:dlambda}.
We now remark that, for any integer $k > 0$, we can guarantee that 
\begin{equ}[e:SmoothNonSmooth]
\P\Big(\sup_{\sz\in T}|N_\gamma(\sz)-N^a_\gamma(\sz)|\ge k\gamma^{-\frac{1}{2}}\Big)
\le  L N  (2\gamma)^{p+k(1-p)}\;,
\end{equ}
so that, choosing $k$ sufficiently large so that $k(p-1) > p$ and then choosing $\gamma$ sufficiently 
large, one has 
\begin{equ}[e:CouplingSup]
\P\Big(\sup_{\sz\in T}|B(\sz)-N^a_\gamma(\sz)| > 2N d_\gamma + k\gamma^{-\frac{1}{2}}\Big) \le 2\eps\;.
\end{equ}
The claim now follows at once by combining this with Lemma~\ref{l:PPtree}.
\end{proof}

\subsection{The Brownian Castle measure}\label{sec:BCM}

In this section, we collect the results obtained so far and show how to define 
a measure, which we will refer to as the {\it Brownian Castle measure}, 
on the space of branching spatial trees which encodes the inner structure of the Brownian Castle. 
We begin with the following proposition, which determines the existence and the continuity properties of 
the Gaussian process defined via~\ref{e:GP} on the Brownian Web Tree of Section~\ref{sec:BW}. 

\begin{proposition}\label{p:GPDBW}
Let $\zeta_\bw^\da$ and $\zeta_\bw^{\per,\da}$ be the 
backward and backward periodic Brownian Web trees in Definition~\ref{def:DBW}. There exist 
Gaussian processes $B_\bc$ and $B^\per_\bc$ indexed by $\ST^\da_\bw$, $\ST^{\per,\da}_\bw$ 
that satisfy~\eqref{e:GP}. Moreover, they admit a version whose realisations are
locally little $\beta$-H\"older continuous for any $\beta<1/2$.
\end{proposition}
\begin{proof}
The existence part of the statement is due to the fact that $\ST^\da_\bw$ and $\ST_\bw^{\per,\da}$ 
are almost surely $\R$-trees (see Remark~\ref{rem:GPExistence}), while, since by Proposition~\ref{p:BW} 
$\zeta_\bw^\da$ and $\zeta_\bw^{\per,\da}$ belong to $\sE_\alpha$ almost surely, 
the H\"older regularity is a direct consequence of Proposition~\ref{p:Holder}. 
\end{proof}

The previous proposition represents the last ingredient to define the Brownian Castle measure. This is given by 
the law of the couple $\chi_\bc\eqdef(\zeta^\da_\bw, B_\bc)$ in the space of characteristic branching 
spatial trees. 

\begin{theorem}\label{thm:BSPT}
Let $\zeta_\bw^\da$ and $\zeta_\bw^{\per,\da}$ be the backward and backward periodic Brownian Web trees 
in Definition~\ref{def:DBW}, and 
$B_\bc$ and $B^\per_\bc$ be the Gaussian processes built in Proposition~\ref{p:GPDBW}. 

Then, almost surely the couple $\chi_\bc\eqdef(\zeta^\da_\bw, B_\bc)$ is a 
characteristic $(\alpha,\beta)$-branching
spatial pointed $\R$-tree according to Definition~\ref{def:BPST}, for any $\alpha,\beta<1/2$. We call its law on 
$\Ch^{\alpha,\beta}_\bsp$ the {\bf Brownian Castle Measure}. The latter can be written as
\begin{equation}\label{e:LawBC}
\CP_\bc(\dd\chi)\eqdef \int\CQ^\Gau_\zeta(\dd X)\,\Theta^\da_\bw(\dd\zeta)
\end{equation}
where 
$\CQ^\Gau_\zeta(\dd X)$ denotes the law of the Gaussian process $B_\bc$ on $\Ch^{\alpha,\beta}_\bsp$. 
Analogously, for the same range of the parameters $\alpha,\beta$, 
$\chi^\per_\bc\eqdef(\zeta^{\per,\da}_\bw, B^\per_\bc)$ almost surely belongs to $\Ch^{\alpha,\beta}_{\bsp, \per}$ 
and we define its law on it, $\CP^\per_\bc(\dd\chi)$, as in~\eqref{e:LawBC}. 
\end{theorem}
\begin{proof}
That $\chi_\bc$ almost surely belongs to $\Ch^{\alpha,\beta}_\bsp$ is a direct consequence of Theorem~\ref{thm:BW} and 
Propositions~\ref{p:GPDBW}. The needed measurability properties that allow to define~\eqref{e:LawBC} follow by 
Proposition~\ref{p:MeasG}. 
\end{proof}

\section{The Brownian Castle}\label{s:bc}

The aim of this section is to rigorously define the ``nice'' version of the Brownian Castle (and its periodic 
version, see Remark~\ref{rem:DistPBC} for its definition) sketched in~\eqref{e:FormalBC}, 
and prove Theorem~\ref{thm:main} along with other properties. 

\begin{remark}\label{rem:DistPBC}
The periodic Brownian Castle $h^\per_\bc$ is defined similarly to $h_\bc$.  
We require it to start from a periodic c\`adl\`ag function in $D(\T,\R)$ and its finite-dimensional 
distributions to be characterised as in Definition~\ref{def:bcFD}, but  
$z_1,\dots,z_n\in(0,+\infty)\times\T$ and
the coalescing backward Brownian motions $B_k$'z are periodic. 
\end{remark}

\subsection{Pathwise properties of the Brownian Castle}\label{sec:BC}

Let $\alpha,\beta<\frac{1}{2}$, $\chi=(\zeta,X)$, for $\zeta=(\ST,\ast,d,M)$, be a (periodic) $(\alpha,\beta)$-branching 
characteristic spatial $\R$-tree according to Definition~\ref{def:BPST} 
satisfying~\ref{i:TreeCond} (see Definition~\ref{def:TreeCond}), 
and $\rho$ be its radial map as in~\eqref{e:RadMap}. 
We introduce the following maps
\begin{equation}\label{e:SBC}
\sh_\chi^\st(z)\eqdef X(\sT(z))\qquad z\in\R^2
\end{equation}
and, for $\sh_0\in D(\R,\R)$ (or in $D(\T,\R)$),
\begin{equation}\label{e:BC}
\sh_\chi^{\sh_0}(z)\eqdef \sh_0(M_x(\rho(\sT(z),0)))+ \sh^\st(z)- X(\rho(\sT(z),0))\,
\end{equation}
for $z=(t,x)\in\R_+\times\R$ (resp. $\R_+\times\T$). 
We are now ready for the following theorem and the consequent definition, which identify the version of the 
Brownian Castle we will be using throughout the rest of the paper.

\begin{theorem}\label{thm:BC}
Let $\chi_\bc$ be the $\Ch^{\alpha,\beta}_\bsp$-valued random variable introduced in Theorem~\ref{thm:BSPT} and 
$\cP_\bc$ be its law given by~\eqref{e:LawBC}. 
Then, for  every c\`adl\`ag function $\sh_0\in D(\R,\R)$, the map $\sh_{\chi_\bc}^{\sh_0}$ in~\eqref{e:BC} 
is $\cP_\bc$-almost surely well-defined and is a version of the 
Brownian Castle $h_\bc$, i.e. it starts at $\sh_0$ at time $0$ and its finite-dimensional distributions 
are as in Definition~\ref{def:bcFD}. 
In the periodic setting, the same statement holds, i.e. for every periodic c\`adl\`ag function $\sh_0\in D(\T,\R)$ 
$\sh_{\chi^\per_\bc}^{\sh_0}$ is $\cP^\per_\bc$-almost surely well-defined, starts at $\sh_0$ at time $0$ and 
has the same finite-dimensional distributions as $h_\bc^\per$ in Remark~\ref{rem:DistPBC}. 
%
\end{theorem}
\begin{proof}
The proof is a direct consequence of our construction in Sections~\ref{sec:BW} and~\ref{sec:BCM} 
and follows by Proposition~\ref{p:BW}, Lemma~\ref{l:Right-most} and Theorems~\ref{thm:BW} and~\ref{thm:BSPT}.
\end{proof}

\begin{definition}\label{def:BC}
We define the {\bf stationary (periodic) Brownian Castle}, $\sh^\st_\bc$ (resp. $\sh^{\per,\st}_\bc$), as the field 
$\sh^\st_{\chi_\bc}$ on $\R^2$ (resp. $\sh^\st_{\chi^\per_\bc}$ on $\R\times\T$) given by~\eqref{e:SBC}, 
while, for $\sh_0\in D(\R,\R)$ (resp. $\sh_0\in D(\T,\R)$), 
we define the {\bf (periodic) Brownian Castle starting at $\sh_0$}, $\sh_\bc$ (resp. $\sh^\per_\bc$), as the map 
$\sh^{\sh_0}_{\chi_\bc}$ (resp. $\sh^{\sh_0}_{\chi^\per_\bc}$) in~\eqref{e:BC}. 
\end{definition}

\begin{remark}\label{rem:Stationary}
Since we require the Gaussian process $B_\bc$ to start from $0$ at $\ast$, $\sh^\st_\bc$ is not, strictly speaking,
stationary but its increments are. As a consequence, writing $\tilde\sh^\st_\bc$ for the projection of $\sh^\st_\bc$ 
onto a space of functions in which two elements are identified if they differ by a fixed constant, we see that 
$\tilde\sh^\st_\bc$ is truly stationary in time. 
%
%
%
\end{remark}

The previous theorem guarantees that thanks to $\chi_\bc$ it is possible to provide a construction of the Brownian Castle
which highlights its inner structure. We will now see how to exploit this construction in order to prove certain
continuity properties that the Brownian Castle $\sh_\bc$ and its periodic counterpart $\sh^\per_\bc$ enjoy. 

\begin{proposition}\label{p:BCcadlag}
$\CP_\bc$-almost surely, for every initial condition $\sh_0\in D(\R,\R)$, the map $t\mapsto\sh_\bc(t,\cdot)$ 
takes values in $D(\R,\R)$, and is continuous from above, i.e. for every $t\in\R_+$ 
$\lim_{s\da t} d_\Sk(\sh_\bc(s,\cdot),\sh_\bc(t,\cdot))=0$. Moreover, $\CP_\bc$-almost surely, for every $t\in\R_+$ such that 
there is no $x\in\R$ for which $(t,x)\in S^\da_{(0,3)}$, $t\mapsto \sh_\bc(t,\cdot)$ is continuous at $t$, i.e. 
$\lim_{s\to t} d_\Sk(\sh_\bc(s,\cdot),\sh_\bc(t,\cdot))=0$. 
The same holds in the periodic setting $\cP^\per_\bc$-almost surely. 
\end{proposition}


\begin{proof}
The definition of $\sh_\bc$, together with Proposition~\ref{p:ContT} and the continuity of $B_\bc$
immediately imply that almost surely for every $t\in\R$, 
$\R\ni x\mapsto \sh_\bc(t,x)$ is c\`adl\`ag and therefore belongs to $D(\R,\R)$. 

In order to prove the second part of the statement, fix $t>0$ and let $s> t$. 
By definition of the Skorokhod distance, it suffices to exhibit a $\lambda_s\in\Lambda$ such that 
$\gamma(\lambda_s)<\eps$ for $s$ sufficiently small and $\sup_{x\in[-R,R]}|\sh_\bc(s,\lambda_s(x))-\sh_\bc(t,x)|<\eps$ for $R$ 
big enough. 
Since $\sh_\bc(t,\cdot)\in D(\R,\R)$,~\cite[Lemma 1.12.3]{Bil99} implies that 
there exist $-R=x_1<\dots<x_n=R$ such that for all $i=1,\dots,n$
\begin{equation}\label{e:PointsD}
\sup_{x,y\in[x_i,x_{i+1})}|\sh_\bc(t,x)-\sh_\bc(t,y)|<\eps/2\,.
\end{equation}
We assume, without loss of generality, that for every $x\in[x_i,x_{i+1})$, $\rho^\da(\sT_\bw(t,x),0)$ coincide. 
Indeed, if this is not the case, it suffices to add a finite number of points $x_i$ and~\eqref{e:PointsD} would still hold.  
Now, for each of the $z_i=(x_i,t)$, we consider $\sz_\rl^i\in\ST^\ua_\bw$,  i.e. the left-most point 
(see Remark~\ref{rem:Right-mostUA}) in the preimage of $z_i$ for the forward Brownian Web tree (which, 
by Theorem~\ref{thm:DBW} is deterministically fixed by $\zeta^\da_\bw$). 
Now, let $\kappa>0$ be small enough so that for $s\in(t,t+\kappa)$
\begin{equ}
M^\ua_{\bw,x}(\rho^\ua(\sz^1_\rl,s))<\dots<M^\ua_{\bw,x}(\rho^\ua(\sz^n_\rl,s))\,,
\end{equ}
set 
\begin{equ}
\lambda_s(x_i)\eqdef M^\ua_{\bw,x}(\rho^\ua(\sz^i_\rl,s))
\end{equ}
and define $\lambda_s(x)$ for $x\neq x_i$ by linear interpolation.
Clearly, $\gamma(\lambda_s)$ converges to $0$ as $s\da t$,
so that we can choose $\tilde s$ sufficiently close to $t$ for which $\gamma(\lambda_s)<\eps$ for all $s\in(t,\tilde s)$.
Now, by the non-crossing property (see point~\ref{i:Cross} in Theorem~\ref{thm:DBW}), 
$y\eqdef M_{\bw,x}^\da(\rho^\da(\sT_\bw(\lambda_s(x),s),t))\in[x_i,x_{i+1})$ for $s\in(t,\tilde s)$ and $x\in[x_i,x_{i+1})$
and clearly $d_\bw^\da(\sT_\bw(\lambda_s(x),s), \rho^\da(\sT_\bw(\lambda_s(x)),t))= s-t$. 
Recall that $B_\bc$ is locally $\beta$-H\"older continuous so that 
upon taking $\bar s\eqdef \tilde s\wedge(t+\eps^{1/\beta}/2)$, we obtain
\begin{equs}
|\sh_\bc(s,\lambda_s(x))-\sh_\bc(t,x)|&\leq |\sh_\bc(s,\lambda_s(x))-\sh_\bc(t, y)|+|\sh_\bc(t,y)-\sh_\bc(t,x)|\\
&<|B_\bc(\sT_\bw(s,\lambda_s(x)))-B_\bc(\sT_\bw(t,y))|+\frac{\eps}{2}<\eps
\end{equs}
for all $x\in[-R,R)$ and $s\leq\bar s$, and from this the result follows. 

It remains to prove the last part of the statement. 
Let $t\in\R_+$ be such that $\{t\}\times\R\cap S^\da_{(0,3)}=\emptyset$ and $\eps>0$. 
We now consider a finite subset of 
$\{t-\eps^{1/\beta}\}\times\R$,  $\tilde \Xi_{[-R,R]}^\da$ 
(which is the image via $M^\da_\bw$ of $\Xi_R$ in~\eqref{e:CPS}), given by 
\begin{equ}[e:Xitilde]
\tilde\Xi^\da_{[-R,R]}(t,t-\eps^{1/\beta}) \eqdef \{M^\da_{\bw,x}(\rho^\da(\sz,t-\eps^{1/\beta}))
\,:\, M^\da_{\bw}(\sz) \in\{t\}\times[-R,R]\}\,.
\end{equ}
Order the elements in the previous set in increasing order, i.e. $\tilde\Xi^\da_{[-R,R]}(t,t-\eps^{1/\beta})=\{x_i\,:\,i=1,\dots, N\}$ 
and $x_1\eqdef \min \tilde\Xi^\da_{[-R,R]}(t,t-\eps^{1/\beta})$. Now, for any $x_i\in\tilde\Xi^\da_{[-R,R]}(t,t-\eps^{1/\beta})$, let 
\begin{equs}[e:yis]
y_i&\eqdef\inf\{y\in\R\,:\,\rho^\da (\sT_\bw(t,y),t-\eps^{1/\beta})=x_i\}\;,\qquad i=1,\dots, N\text{ and }\\
y_{N+1}&\eqdef\sup\{y\in\R\,:\,\rho^\da (\sT_\bw(t,y),t-\eps^{1/\beta})=x_N\}
\end{equs}
then, since by Theorem~\ref{thm:DBW}~\ref{i:Cross} forward and backward paths do not cross, we know that 
$\{y_i\}$ coincides with $\tilde \Xi^\ua_{[x_1,x_N]}(t-\eps^{1/\beta},t)$ (defined as $\tilde\Xi^\da$ but with all arrows reversed). 
By duality $S^\da_{(0,3)}=S^\ua_{(2,1)}$, hence $\{(y_i,t)\}\cap S^\ua_{(2,1)}=\emptyset$. Therefore, there exists 
a time $\tilde t\in(t-\eps^{1/\beta},t)$ such that no pair of forward paths started before $t-\eps^{1/\beta}$ 
and passing through $[x_1,x_N]$ at time $t-\eps^{1/\beta}$, coalesces at a time $s\in(\tilde t,t]$. In other words, 
the cardinality of $\tilde \Xi^\ua_{[x_1,x_N]}(t-\eps^{1/\beta},t)$ coincides with that of $\tilde \Xi^\ua_{[x_1,x_N]}(t-\eps^{1/\beta},s)$
for any $s\in (\tilde t,t]$. 

For $i\leq N$, let $\sz_i$ be the unique point in $\ST^\ua_\bw$, 
such that for all $\sz\in\ST^\ua_\bw$ for which $M^\ua_{\bw,x}(\sz)\in\{t-\eps^{1/\beta}\}\times[x_i,x_{i+1}]$,  
$\rho^\ua(\sz, \tilde t)=\sz_i$. We define the map $\lambda_{s}$, $s\in(\tilde t, t)$, as
\begin{equ}
\lambda_s(x_i)\eqdef M^\ua_{\bw,x}(\sz_i)
\end{equ}
and for $x\neq x_i$ we extend it by linear interpolation. Clearly, $\gamma(\lambda_s)$ converges to $0$, so that we can choose
$\tilde s>\tilde t$ sufficiently close to $t$ so that $\gamma(\lambda_s)<\eps$. 
Now, notice that, by construction (and Theorem~\ref{thm:DBW}~\ref{i:Cross}), 
for all $x\in[-R,R]$ and $s\in(\tilde s, t)$, $\sT_\bw(t,x)$ and $\sT_\bw (s,\lambda_s(x))$ must be such that 
$d^\da(\sT_\bw(t,x), \sT_\bw (s,\lambda_x(x)))<\eps^{1/\beta}$ which, 
by the (local) $\beta$-H\"older continuity of $B_\bc$, guarantees that 
\begin{equ}
|\sh_\bc(t,x)-\sh_\bc(s,\lambda_s(x))|=|B_\bc(\sT_\bw(t,x))-B_\bc(\sT_\bw (s,\lambda_x(x))|\lesssim \eps
\end{equ}
which concludes the proof in the non-periodic setting. 
The periodic case follows the same steps but, as spatial interval, one can directly take the whole of $\T$ instead of $[-R,R]$.
\end{proof}

\begin{remark}
By Proposition~\ref{p:ContT}, the fact that $\sh_\bc(t,\bigcdot)$ is càdlàg simply follows from its description 
in terms of an element of $\C^{\alpha,\beta}_\bsp$. The fact that it is right-continuous
as a function of time however uses specific properties of the Brownian Castle itself and wouldn't be true 
for $\sh_\bc$ built from an arbitrary element of $\C^{\alpha,\beta}_\bsp$.
\end{remark}

In the next proposition, we show that it is possible to obtain a finer control over the fluctuations of 
the Brownian Castle. 

\begin{proposition}\label{p:pvar}
$\CP_\bc$-almost surely, for every $t>0$, $\sh_\bc(t,\cdot)$  
has finite $p$-variation for every $p>1$, locally on any bounded interval of $\R$.
\end{proposition}
\begin{proof}
Let $t,\,R>0$ and consider $\sh_\bc(t,\cdot)$ restricted to the interval $[-R,R]$. 
At first, we will approximate  $\sh_\bc(t,\cdot)$ by piecewise constant functions. 

For any $n \ge 0$, let $\tilde\Xi(t,t-2^{-n})$ be the set defined according to~\eqref{e:Xitilde} 
and let $N_n\eqdef\eta_R(t,2^{-n})$ be its cardinality, which we recall  
satisfies the bound $N_n \lesssim 2^{\varsigma n}$ (for some random proportionality constant
independent of $n$) given in~\eqref{e:CPSbound}. 
We order the points of $\tilde\Xi(t,t-2^{-n})$ as in the proof of Proposition~\ref{p:BCcadlag}, 
denote them by $x_1^{(n)}<\dots<x_{N_n}^{(n)}$, and set $x_0^{(n)}\eqdef -R$ and $x_{N_n+1}^{(n)}\eqdef R$. 
We define the piecewise constant function $\sh^n_\bc(t,\cdot)$ by 
\begin{equ}
\sh^n_\bc(t,x)=\sh_\bc(t,x_i^{(n)})\qquad\text{for $x\in[x_i^{(n)},x_{i+1}^{(n)})$.}
\end{equ} 
We then note that, for any $x\in[-R,R]$ we have the identity 
\begin{equ}
\sh_\bc(t,x)-\sh^0_\bc(t,x)=\sum_{n\ge 0} \big(\sh^{n+1}_\bc(t,x)-\sh^{n}_\bc(t,x)\big)
\end{equ} 
so that in particular, for any $p\ge 1$,
\begin{equ}[e:Trpvar]
\|\sh_\bc(t,\cdot)-\sh^0_\bc(t,\cdot)\|_{\pvar}\leq \sum_{n\ge 0} \|\sh^{n+1}_\bc(t,\cdot)-\sh^{n}_\bc(t,\cdot)\|_{\pvar}\;,
\end{equ}
the $p$-variation norm $\|\cdot\|_{\pvar}$ being defined as in~\eqref{e:pvar}. 
Thanks to the $\beta$-H\"older continuity of $B_\bc$, we then have 
\begin{equ}
\|\sh^{n+1}_\bc(t,\cdot)-\sh^{n}_\bc(t,\cdot)\|_{\pvar}\lesssim 2^{-\beta n} N_{n+1}^{1/p}\;,
\end{equ}
since $\sh^{n+1}_\bc(t,\cdot)-\sh^{n}_\bc(t,\cdot)$ is a piecewise constant 
function with sup-norm over $[-R,R]$ bounded by $C 2^{-\beta n}$ (for some random $C$
independent of $n$) and at most $N_{n+1}$ jumps. 
Inserting this bound into~\eqref{e:Trpvar} and exploiting the bound on $N_{n+1}$ 
provided by Lemma~\ref{l:CPScard}, we obtain
\begin{equ}
\|\sh_\bc(t,\cdot)-\sh^0_\bc(t,\cdot)\|_{\pvar}\lesssim \sum_{n\ge 0} 2^{(\varsigma/p - \beta)n}\;,
\end{equ}
which is finite for any $p>\varsigma/\beta$. Since both $\varsigma$ and $\beta$ can be chosen 
arbitrarily close to $1/2$, the statement follows. 
\end{proof}

Combining the (H\"older) continuity of the map $B_\bc$ 
(or of $B^\per_\bc$) with Proposition~\ref{p:ContT}, we conclude 
that the set of discontinuities of $\sh_\bc$ is contained in $S^\da_{0,2} \cup S^\da_{1,2} \cup S^\da_{0,3}$ (see Definition~\ref{def:Type}) or, by duality (see Theorem~\ref{thm:Types}), 
in the image through $M^\ua_\bw$ of the {\it skeleton} of the forward Brownian Web 
$\ST^{o,\ua}_\bw$\footnote{In~\cite{CHbwt}, it was shown that the skeleton is given by $\ST^{\ua}_\infty(\cD)$ (resp. $\ST^{\per,\ua}_\infty(\cD)$), $\cD$ being any countable dense of $\R^2$ 
(resp. $\T\times\R$), but with the endpoints removed} 
given by~\eqref{def:Skeleton}, and the same holds for $\sh^\per_\bc$. 

This means that we can identify specific events in the spatio-temporal evolution of the (periodic) Brownian Castle with 
special points of the (periodic) Brownian Web. 
Let us define the {\it basin of attraction} for the shock at $z=(t,x)\in\R_+\times\R$  as 
\begin{equation}\label{def:boa}
\begin{split}
A_z\eqdef\{z'=(t',x')&\in\R^2\,:\,t'<t\text{ and there exists}\\
&\text{$\sz'\in\ST^\ua_\bw$ s.t. $M^\ua_\bw(\sz')=z'$ and $M^\ua_\bw(\rho^\ua(\sz',t))=z$}\}
\end{split}
\end{equation}
and the {\it age} of the shock as 
\begin{equation}\label{e:age}
a_z = t - \sup\{t' < t\,:\, \{t'\}\times \R \cap A_z = \emptyset\}\,.
\end{equation}
and define {\it mutatis mutandis} $A_z^\per$ and $a_z^\per$ as the basin of attraction of a shock at $z\in\R_+\times\T$ and its age, 
in the periodic setting. 
In the following proposition, we show properties of the age of a point $z$ and characterise its basin of attraction.

\begin{proposition}\label{p:BoA}
\begin{enumerate}[noitemsep]
\item In both the periodic and non-periodic case, the set of points with strictly positive age coincides with the union of 
points of type $(i,j)$ for $j>1$.
\item In the non-periodic case, almost surely, for every $z$, $a_z<\infty$ and there exists a unique 
$z' = (t',x') \in A_z$ that realises the 
supremum in~\eqref{e:age}, i.e.\ such that $a_z = t-t'$. In the periodic case, for every $t\in\R$, the 
previous holds for all $z\in\{t\}\times\T$ except for exactly one value $z_\per=(t,x_\per)$, which is such that $a_{z_\per}=\infty$.  
\item In the non-periodic case, if $z=(t,x)$ is such that $a_z > 0$, then the unique $z' = (t',x') \in A_z$ 
(determined in the previous point) such that $a_z=t-t'$, belongs to $S^\ua_{0,3}$. Moreover, 
the left-most and right-most points at $z$, $\sz_\rl,\,\sz_\rr\in (M^\da_\bw)^{-1}(z)$  are such that 
$\sz_\rl \neq \sz_\rr$ and
\begin{equation}\label{e:intBoA}
\mathring A_z = \bigcup_{s<t}\, (M_\bw^\da(\rho^\da(\sz_\rl,s)),M_\bw^\da(\rho^\da(\sz_\rr,s)))\;.
\end{equation}
where $\mathring A_z$ denotes the interior of $A_z$ and $A_z$ is compact. 
\end{enumerate}
\end{proposition}

\begin{proof}
Point 1. is an immediate consequence of Theorem~\ref{thm:Types}. Indeed, if $z$ is such that $a_z>0$, 
then there exists a point in $(M^\ua_\bw)^{-1}(z)$ whose degree is strictly greater than~$1$, 
which implies that $|(M^\da_\bw)^{-1}(z)|\geq 2$ so that $z$ belongs to the union of $S^\da_{i,j}$ for $j>1$. 
Vice-versa, if $z$ belongs to one of the $S^\da_{i,j}$ for $j>1$ then, by duality, 
it belongs to one of $S^\ua_{i,j}$ for $(i,j)=(1,1),\,(2,1)$ or $(1,2)$
so that there exists at least one $\sz'\in\ST^\ua_\bw$ such that $M^\ua_{\bw,t}(\sz')<t$
and $M^\ua_\bw(\alpha^\ua_\bw(\sz',t))=z$. Hence $a_z\geq t-M^\ua_{\bw,t}(\sz')>0$.

For point 2., we first show that if $a_z$ is finite then the point realising the supremum is unique, 
the proof being the same in the periodic 
and non-periodic setting. Assume there exist 
$z'=(t',x'),\,z''=(t',x'')\in A_z$ realising the supremum in~\eqref{e:age}. Then, 
by the coalescing property, every point in $\{t'\}\times[x',x'']$ (or $\{t'\}\times(\T\setminus[x',x''])$) belongs to $A_z$. 
But, according to Theorem~\ref{thm:Types} almost surely
for every $s\in\R$, $S^\ua_{1,1}\cap \{s\}\times\R$ is dense in $\{s\}\times\R$, hence, 
there is $\tilde z\in S^\ua_{1,1}\cap \{t'\}\times[x',x'']$ and $\sz\in\ST^\ua_\bw$ such that $M^\ua_{\bw,t}(\sz)<t'$ and 
$M^\ua_{\bw,t}(\rho^\ua(\sz,t'))=\tilde z$. But then $a_z\geq t-M^\ua_{\bw,t}(\sz)>t-t'$, which is a contradiction. 

Since, by~\cite[Proposition 3.21]{CHbwt}, $\ST^\ua_\bw$ has a {\it unique} open end with unbounded rays, 
for every $z$, $a_z<\infty$. 
This is not true anymore for $\ST^{\per,\ua}_\bw$ which has exactly two open ends with unbounded rays, but 
since there exists a {\it unique} bi-infinite edge connecting them (see ~\cite[Proposition 3.25]{CHbwt}), it follows that 
for every $t$ there is a unique $x_\per\in\T$ such that $(t,x_\per)$ has infinite age.

Let us now focus on 3. Let $z=(t,x)$ be such that $a_z>0$. From~2., 
there exists a unique point $z'\in\R^2$ and a point $\sz'\in\ST^\ua_\bw$
such that $M^\ua_\bw(\sz')=z'$ and $M^\ua_\bw(\rho^\ua(\sz',t))=z$, hence 
$z\in S^\ua_{1,1}\cup S^\ua_{2,1}\cup S^\ua_{1,2}$. 
Thanks to Theorem~\ref{thm:Types}, the right-most and left-most points in $(M_\bw^\da)^{-1}(z)$, 
$\sz_\rr,\sz_\rl\in\ST^\da_\bw$ must be distinct. By Theorem~\ref{thm:DBW}~\ref{i:Cross}, 
forward and backward trajectories cannot cross, therefore for every $s\in(t',t)$, 
$M^\da_\bw(\rho^\da(\sz_\rl,s))<M^\ua_\bw(\rho^\ua(\sz',s))<M^\da_\bw(\rho^\da(\sz_\rr,s))$. In particular, 
the backward paths starting from $z$ cannot coalesce before $t'$. They cannot coalesce after $t'$ either since, 
if this 
were to be the case, then for the same reasons as above the path in the forward web starting from any point in
$\{s\}\times(M^\da_\bw(\rho^\da(\sz_\rl,s), M^\da_\bw(\rho^\da(\sz_\rr,s)))$, $s<t'$ would be contained 
in $A_z$, contradicting point~2. It follows that the point at which the two backward paths coalesce is exactly $z'$, 
which implies that $z'\in S^\da_{2,1}=S^\ua_{0,3}$. Moreover the previous argument also shows that~\eqref{e:intBoA} 
holds (with $\sz_1=\sz_\rr$ and $\sz_2=\sz_\rl$). 
%
\end{proof}

\begin{remark}
Proposition~\ref{p:BoA} and its proof underline one of the main visible differences between the Brownian Castle and its 
periodic counterpart. Indeed, only $\ST^{\per,\ua}_\bw$ possesses a bi-infinite edge $\beta^\ua$, 
which implies that $\sh^\per_\bc$ exhibits a 
``master shock'' starting back at $-\infty$ and running along $M^{\per,\ua}_\bw(\beta^\ua(\cdot))$.
Indeed, as we have seen above, for every $s\in\R$ 
there exist two backward paths starting in or passing through $M^{\per,\ua}_\bw(\beta^\ua(s))$ that before meeting need to 
transverse the whole torus. On the other hand, all the discontinuities of $\sh_\bc$ have a finite origin that can be tracked with 
the methods shown in Proposition~\ref{p:BoA}. 
\end{remark}

The following proposition collects the most important connections between certain events we witness on the Brownian Castle 
and special points in the Web. 

\begin{proposition}
\begin{enumerate}[noitemsep]
\item Shocks for $\sh_\bc$ and $\sh_\bc^\per$ correspond to the trajectories of the forward and periodic forward 
Brownian Web trees respectively, i.e. they are points of type  $(1,1)$ or $(1,2)$ for $\zeta^\ua_\bw$ (resp. $\zeta^{\per,\ua}_\bw$).  
\item If two shocks merge at $z$, then $z$ is of type $(0,3)$ for the backward (periodic) Brownian Web. 
\end{enumerate}
\end{proposition}

\begin{proof}
The result follows by the fact that, by construction, the paths of backward Brownian Web tree represent the backward 
characteristics of the Brownian Castle, and by duality.
\end{proof}

The previous proposition provides the reason why there is no chance for the Brownian Castle $\sh_\bc(s,\cdot)$ 
to admit a limit as $s\ua t$ for all $t\in\R_+$, in the Skorokhod topology 
(or any of the $M_1$, $J_2$ and $M_2$-topologies on this space, see~\cite[Section 12]{Whitt}), 
{\it independently} of the specifics of the proof of Proposition~\ref{p:BCcadlag} or our construction. 
Indeed the Skorokhod topology allows for discontinuities to evolve {\it continuously} and to merge only if their difference 
{\it continuously converges to $0$}. This is not necessarily the case here. 
Indeed, if $z=(t,x)\in S^\ua_{2,1}$, then there are two paths in the forward Web that coalesce at $z$, 
i.e. two discontinuities merging there. According to Proposition~\ref{p:ContT}, for $s$ sufficiently close to $t$
these discontinuities evolve continuously up to the time at which they merge but there is no reason for their difference 
to vanish. The pointwise limit of $\sh_\bc(s,\cdot)$ as $s\ua t$ would then need to encode three different values 
at the point $z$, but the resulting object is not an element of $D(\R,\R)$. Furthermore, according to Theorem~\ref{thm:Types} 
$S^\ua_{2,1}$ is a countable yet {\it dense} subset of $\R^2$ so that points at which c\`adl\`ag continuity 
fails are very common!

In the following proposition, whose proof is based on the above heuristics, we show that for any choice 
of initial condition, there is no version of the Brownian Castle $h_\bc$ 
(defined by simply specifying its finite dimensional distributions) which is c\`adl\`ag in time and space. 

\begin{proposition}
Given any initial condition $\sh_0\in D(\R,\R)$ and $T >0$, the Brownian Castle starting at $\sh_0$ 
does not admit a
version in $D([0,T],D(\R,\R))$. The same is true for the periodic Brownian Castle. 
\end{proposition}
\begin{proof}
Since a right-continuous function with values in $D(\R,\R)$ is uniquely determined by its values at space-time 
points with rational coordinates (for example), it suffices to show that the exists a (random) time 
for which $\sh_\bc$ admits
no left limit in $D(\R,\R)$. For this, it suffices to find a point $(t,x)$ and three
sequences $(t_k,x^{(i)}_k)_{k \ge 0}$ (here $i \in \{1,2,3\}$) with $t_k \uparrow t$, 
$x^{(i)}_k \to x$, and $\lim_{k \to \infty} \sh_\bc(t_k,x^{(i)}_k) = L_i$ with all three limits 
$L_i$ different from each other.

Now, notice that, almost surely, one can find two elements $\sx_0, \sx_1 \in \ST_\bw^\uparrow$ 
with $M_\bw^\uparrow(\sx_i) = (x_i, 0)$ and $x_0$, $x_1$ in $[-1,1]$ 
such that, for the forward Brownian Web tree, one has $\rho^\uparrow(\sx_0,T) = \rho^\uparrow(\sx_1,T)$.
Writing $t = \inf\{s > 0\,:\, \rho^\uparrow(\sx_0,s) = \rho^\uparrow(\sx_1,s)\}$
and $x = M_\bw^\ua(\rho^\uparrow(\sx_0,t))$, we then
necessarily have $(t,x) \in S_{0,3}^\downarrow$ by duality and, furthermore, the three trajectories emanating
from $(t,x)$ in the backwards Brownian Web cannot coalesce before time $0$ by the 
non-crossing property. Since further, the Gaussian process $B_\bc$ is locally H\"older continuous and, 
with high probability, $t\geq c>0$ for some positive constant $c$, the claim then follows by taking for 
$(t_k, x_k^{(i)})$, sequences accumulating at 
$(t,x)$ and belonging to these three trajectories.
\end{proof}

\subsection{The Brownian Castle as a Markov process}\label{sec:BCprocess}

We are now interested in studying the properties of the (periodic) Brownian Castle as a random interface evolving in time, i.e.\ as a 
stochastic process with values in $D(\R,\R)$ (resp. $D(\T,\R)$). 
To do so, we need to introduce a suitable filtration on the probability space $(\Omega,\cF,\cP_\bc)$ 
(resp.\ $(\Omega_\per,\cF_\per, \cP^\per_\bc)$) on which the Castle is defined\footnote{For example $\Omega$ can be taken to be $\Ch^{\alpha,\beta}_\bsp$ and $\cF$ the Borel $\sigma$-algebra induced by the metric in~\eqref{def:bspMetric}}. 
From now on, we assume that all 
sub-$\sigma$-algebras of $\cF$ (resp. $\cF_\per$) that we consider contain all null events. 

We will make use of the following construction. Given a metric space $\CX$, we write $\CX^{2c} \subset \CX^2$ 
for the set of all pairs $(x,y)$ such that $x$ and $y$ are in the same path component of $\CX$.
Given $B \colon \CX \to \R$ (or into any abelian group), we then write $\delta B \colon \CX^{2c} \to \R$
for the map given by $\delta B(x,y) = B(y) - B(x)$.

Let $\chi=(\ST,\ast,d,M,B)\in \Omega=\Ch^{\alpha,\beta}_\bsp$ (or $\Omega_\per$) and, for $-\infty\leq s\leq t<+\infty$, 
define $\ST_{s,t}\eqdef M^{-1}((s,t]\times\R)$ (or $\ST_{s,t}\eqdef M^{-1}((s,t]\times\T)$). 
Let $\eval_{s,t}$ be the map given by
\begin{equation}\label{e:eval}
\eval_{s,t}(\chi)\eqdef \big(\zeta_{s,t}, \delta \big(B\restr \ST_{s,t}\big)\big)
\end{equation}
where 
$\zeta_{s,t}\eqdef(\ST_{s,t},d,M\restr \ST_{s,t})$. We use the notations
\begin{equ}[e:filt]
\cF_{s,t} \eqdef \sigma(\eval_{s,t})\;,\qquad \cF_t\eqdef \sigma(\eval_{-\infty,t})\;,
\end{equ}
for the $\sigma$-algebras that they generate. The following property is crucial.

\begin{lemma}\label{lem:indep}
If the intervals $(s,t]$ and $(u,v]$ are disjoint, then $\cF_{s,t}$ and $\cF_{u,v}$ are
independent.
\end{lemma}

\begin{proof}
The fact that $\ST_{s,t}$ and $\ST_{u,v}$ are independent under the law of the Brownian Web tree
was shown for example in \cite[Prop.~2]{HoW} (this is for a slightly 
different representation of the Brownian web, but the topological space $\ST$ can be recovered from it in a measurable way).
It remains to note that, conditionally on $\ST_{s,t}$ and $\ST_{u,v}$, the joint law of $\delta \big(B\restr \ST_{s,t}\big)$ and 
$\delta \big(B\restr \ST_{u,v}\big)$ is of product form with the two factors being $\ST_{s,t}$ and $\ST_{u,v}$-measurable 
respectively. This follows immediately from the independence properties of Brownian increments as formulated in
Remark~\ref{rem:disjoint}.
\end{proof}

One almost immediate consequence is that both $\sh_\bc$ and $\sh_\bc^\per$ are time-homogeneous strong Markov processes 
satisfying the Feller property. 

\begin{proposition}
The (periodic) Brownian Castle $\sh_\bc$ (resp. $\sh_\bc^\per$) is a time-homogeneous 
$D(\R,\R)$ (resp. $D(\T,\R)$)-valued Markov process on the 
complete probability space $(\Omega, \cF, \cP_\bc)$ (resp. $(\Omega_\per,\cF_\per, \cP^\per_\bc)$), 
with respect to the filtration $\{\cF_t\}_{t\geq0}$ introduced in~\eqref{e:filt}. 
Moreover, both $\{\sh_\bc(t,\bigcdot)\}_{t\geq0}$ and $\{\sh^\per_\bc(t,\bigcdot)\}_{t\geq0}$ are strong Markov and Feller.
\end{proposition}
\begin{proof}
The proof works {\it mutatis mutandis} for both the periodic and non-periodic case 
so we will focus on the latter. 

We have already shown that $\cP_\bc$-almost surely for every $\sh_0$ and $t\geq 0$, $\sh_\bc(t,\bigcdot)\in D(\R,\R)$ 
(see Proposition~\ref{p:BCcadlag}). Moreover, by construction, $\sh_\bc(t,\bigcdot)$ only depends on $\eval_{-\infty,t}(\chi_\bc)$ so 
that it is clearly $\cF_t$-measurable. 
Notice that, by definition, for every $0\leq s< t$ and $x\in\R$, we can write
\begin{equ}
\sh_\bc(t,x)=\sh_\bc(s,M^\da_\bw(\rho^\da(\sT_\bw(t,x),s)))+ B_\bc(\sT_\bw(t,x))-B_\bc(\rho^\da(\sT_\bw(t,x),s))\;,
\end{equ}
so that $\sh_\bc(t,\bigcdot)$ is $\sh_\bc(s,\bigcdot) \vee \cF_{s,t}$-measurable.
Since $\cF_{s,t}$ and $\cF_s$ are independent by Lemma~\ref{lem:indep}, the Markov property follows, 
while the time homogeneity is an immediate consequence of the stationarity of $(\ST^\da_\bw, M^\da_\bw, B_\bc)$.

Stochastic continuity was already shown in Proposition~\ref{p:BCcadlag}, so, if
we show that the law of $\sh_\bc(t,\bigcdot)$ depends continuously (in the topology of weak convergence) on $\sh_0$, 
then the Feller property holds.
By the definition of the Skorokhod topology, it is sufficient to show that, if $\{\sh_0^n\}_{n\in\N}\subset D(\R,\R)$ 
is a sequence converging to $\sh_0$ in $D(\R,\R)$ then, for every $R > 0$, one has
$\sup_{|x|\le R} |\sh^n_{\bc}(t,x) - \sh_{\bc}(t,x)| \to 0$ in probability, where
we write $\sh^n_{\bc}$ for the Brownian Castle with initial condition $\sh_0^n$.

Note that
\begin{equ}[e:FellerRHS]
\sup_{|x| \le R}|\sh^n_{\bc}(t,x)-\sh_{\bc}(t,x)|\leq \sup_{|x| \le R}|\sh_0^n(y(x))-\sh_0(y(x))|
\end{equ}
where $y(x)\eqdef M^\da_\bw(\rho^\da(\sT_\bw(t,x),0))$. With probability one, the set $\{y(x)\,:\,x\in [-R,R]\}$
is finite and has empty intersection with the set of discontinuities of $\sh_0$. 
Hence, the right-hand side of \eqref{e:FellerRHS} converges to $0$ (almost surely and therefore also in probability) 
by~\cite[Prop.~3.5.2]{EK86}.

Since the Brownian Castle almost surely admits right continuous trajectories and is Feller, the same proof as 
in~\cite[Theorem III.3.1]{RevYor} guarantees that it is strong Markov (even though its state space is not locally compact). 
\end{proof}

The periodic Brownian Castle $\sh_\bc^\per$, is not only Feller, but also strong Feller, namely its 
Markov semigroup maps bounded functions to continuous functions. 
It will be convenient to write $\ST[t] = M^{-1}(\{t\} \times \T)$ for the time-$t$ ``slice''
of a spatial $\R$-tree.

\begin{proposition}
The periodic Brownian Castle satisfies the strong Feller property. 
\end{proposition}
\begin{proof}
Let $\Phi\in\CB_b(D(\T,\R))$ bounded by $1$, $\sh_0\in D(\T,\R)$ and $t>0$. 
We aim to show that, 
for every $\eps>0$ there exists $\delta>0$ such that whenever 
$d_\Sk(\bar\sh_0,\sh_0)<\delta$ 
\begin{equ}[e:wantedSF]
| \Exp_\bc[\Phi(\sh^\per_\bc(t,\cdot))|\sh_0]-\Exp_\bc[\Phi(\sh^\per_\bc(t,\cdot))|\bar\sh_0]|<\eps\,.
\end{equ}
Fix $\eps>0$. Let $\nu\in(0,t)$ sufficiently small and $\bar N$ big enough so that the probability of the events 
\begin{equs}
A_1&\eqdef\left\{\#\{\rho^\da(\sz,\nu)\,:\,\sz\in\ST^{\per,\da}_\bw[t]\} = \#\{\rho^\da(\sz,0)\,:\,\sz\in\ST^{\per,\da}_\bw[t]\}\right\}\\
A_2&\eqdef\left\{\#\{\rho^\da(\sz,0)\,:\,\sz\in\ST^{\per,\da}_\bw[t]\} \le \bar N\right\}
\end{equs}
is each at least $1-{\eps\over 3}$. This is certainly possible since as $\nu$ goes to $0$ and 
$\bar N$ tends to $\infty$ the probability of both $A_1$ and $A_2$ goes to $1$. 
On $A_1\cap A_2$, let $\sz_1,\dots,\sz_N$ be the list of all distinct points of 
$\{\rho^\da(\sz,\nu)\,:\,\sz\in\ST^{\per,\da}_\bw[t]\}$ (clearly, $N\leq\bar N$)
and set $y_i= M^{\per,\da}_{\bw,x}(\sy_i)$ where $\sy_i = \rho^\da(\sz_i,0)$. 
As before, the probability that one of the $y_i$'s is a discontinuity point for $\sh_0$ is $0$. 
Hence, by~\cite[Prop.~3.5.2]{EK86}, for every $\hat \eps > 0$ it is possible to choose $\delta>0$ small enough so that whenever
$d_\Sk(\bar\sh_0,\sh_0)<\delta$ then $\delta_i\eqdef \sh_0(y_i)-\bar\sh_0(y_i)$
satisfies $|\delta_i|<\hat \eps$ for all $i$.

Write simply $B$ instead of $B^{\per,\da}_\bc$ as a shorthand.
Note now that for every $x \in \T$ there exists $i\le N$ such that one can write
\begin{equ}[e:reprBC]
\sh^\per_\bc(t,x) = \sh_0(y_i)  + \delta B(\sy_i,\sz_i) + \delta B(\sz_i, \sT_\bw(t,x))\;,
\end{equ}
and, conditional on $\ST^{\per,\da}_\bw$, the collection of random variables 
$\{\delta B(\rho^\da(\sz,\nu), \sz)\,:\, \sz \in \ST^{\per,\da}_\bw[t]\}$
is independent of the collection $\{\delta B(\sy_i,\sz_i)\}_{i \le N}$.
Conditional on $\ST^{\per,\da}_\bw$ and restricted to $A_1 \cap A_2$, 
the law of the latter is $\CN(0,\nu \Id_N)$ for some $N \le \bar N$. We now choose $\hat \eps$ 
small enough so that $\|\CN(0,\nu \Id_N) - \CN(h,\nu \Id_N)\|_\TV \le \eps/3$, uniformly over all
$N \le \bar N$ and all $h \in \R^{N}$ with $|h_i| \le \hat \eps$.

Writing $\bar\sh^\per_\bc$ for the Brownian Castle with initial condition $\bar \sh_0$, it 
immediately follows from the properties of the total variation distance that
we can couple $\bar\sh^\per_\bc$ and $\sh^\per_\bc$ in such a way that 
$\P(\bar\sh^\per_\bc(t,\bigcdot) = \sh^\per_\bc(t,\bigcdot)) \ge 1-\eps$, uniformly over 
$\bar \sh_0$ with $d_\Sk(\bar\sh_0,\sh_0)<\delta$, and \eqref{e:wantedSF} follows.
\end{proof}

We now want to study the large time behaviour of the Brownian Castle and its periodic counterpart. 
Notice at first that for any sublinearly growing initial condition $\sh_0$, the variance of $\sh_\bc(t,0)$ grows like $t$ 
since, by Definition~\ref{def:bcFD}, $\sh_\bc(t,0)$ conditioned on $\ST_\bw^\da$ is Gaussian with variance $t$ and 
mean given by $\sh_0$, evaluated at the point where the backward Brownian motion starting at $(t,0)$ hits $\{0\}\times\R$. 
On the other hand, it is immediate that the Brownian Castle is equivariant under the action of $\R$ by vertical translations
in the sense that one has $\sh_\bc^{\sh_0 + a} = \sh_\bc^{\sh_0} + a$. As a consequence, writing
$\tilde D(\R,\R) = D(\R,\R) / \R$ for the quotient space, the canonical projection of the Brownian Castle 
onto $\tilde D$ is still a Markov process.	
We henceforth write $\tilde\sh_\bc$ (respectively $\tilde\sh^\per_\bc$) for this Markov process.

Recall the stationary Brownian Castle $\sh_\bc^\st$ given in Definition~\ref{def:BC}.  
As above, we write $\tilde \sh_\bc^s$ for its canonical projection to $\tilde D$ and similarly for its
periodic version, which, according to Remark~\ref{rem:Stationary}, are truly 
stationary. With these notations, we then have the following result.

\begin{proposition}\label{p:LargeTime}
There exists a stopping time $\tau$ with exponential tails such that, for $t \ge\tau$, one has
$\tilde \sh_\bc^\per(t,\bigcdot) = \tilde \sh_\bc^{\per,s}(t,\bigcdot)$ independently of the initial condition
$\sh_0$.
\end{proposition}
\begin{proof}
It suffices to take for $\tau$ the first time such that all the backward paths starting from 
$\{t\}\times\T$ coalesce before hitting time $0$, namely
\begin{equ}
\tau \eqdef\inf\Big\{t \ge 0: d^\da_\bw(\sz,\sz')\leq 2t\,,\quad\forall\,\sz,\sz'\in\ST^{\per,\da}_\bw[t]\Big\}\,.
\end{equ}
Notice that $\tau$ coincides in distribution with $T^\ua(0)$ introduced in~\cite[Sec.~3.1]{CMT}. (This is by duality:
the non-crossing property guarantees that all backwards trajectories starting from $t$ coalesce before time $0$ 
precisely when all forward trajectories starting from $0$ have coalesced.)
It follows immediately from the definitions that, for all $t \ge \tau$, $\tilde \sh_\bc^\per(t,\bigcdot)$ is
independent of $\sh_0$ and therefore equal to $\tilde \sh_\bc^{\per,\st}(t,\bigcdot)$.
Exponential integrability of $\tau$ then follows from~\cite[Prop.~3.11 (ii)]{CMT}. 
\end{proof}

In the non-periodic case, one cannot expect such a strong statement of course, but the following bound still holds.

\begin{proposition}\label{prop:convergenceFull}
The bound
\begin{equ}
\Exp_\bc[d_\Sk(\tilde\sh_\bc(t,\cdot), \tilde\sh^\st_\bc(t,\cdot))] \lesssim {\log t \over \sqrt t}\;,
\end{equ}
holds independently of the initial condition $\sh_0$.
\end{proposition}

\begin{proof}
The definition of $d_\Sk$ implies that if $\tilde\sh_\bc(t,x) = \tilde\sh^\st_\bc(t,x)$ for all $x$ with 
$|x| \le R$, then $d_\Sk(\tilde\sh_\bc(t,\cdot), \tilde\sh^\st_\bc(t,\cdot)) \le e^{-R}$ and that, in any case,
$d_\Sk$ is bounded by $1$. It follows that
\begin{equ}
\Exp d_\Sk(\tilde\sh_\bc(t,\cdot), \tilde\sh^\st_\bc(t,\cdot)) \le \P(A_R^c) + e^{-R}\P(A_R)\;,
\end{equ}
for any $R>0$ and any event $A_R$ implying that $\tilde\sh_\bc(t,\cdot)$ and $\tilde\sh^\st_\bc(t,\cdot)$ agree on $[-R,R]$.
It suffices to take for $A_R$ the event that the two backwards trajectories starting at $(t,\pm R)$ coalesce before time $0$.
Since $\P(A_R^c) \lesssim R/\sqrt t$, choosing $R = \log t$ yields the claim.
\end{proof}

\subsection{Distributional properties}\label{sec:BCdist}

In this section, we will focus on the distributional properties of the stationary Brownian Castle which, 
as showed in the previous section, describes the long-time behaviour of 
$\sh_\bc$ (resp. $\sh^\per_\bc$), at least modulo vertical translations. 
We begin with the following proposition which shows that $\sh_\bc$ is indeed invariant with respect to the $1{:}1{:}2$ scaling, 
i.e.\ its scaling exponents are indeed those characterising its own universality class. 

\begin{proposition}\label{p:Scaling}
Let $\sh^\st_\bc$ be the stationary Brownian Castle defined according to~\eqref{e:SBC}. Then, for any $\lambda>0$,
\begin{equ}[e:eqlaw]
\lambda\sh^\st_\bc(\bigcdot/\lambda^2, \bigcdot/\lambda)\eqlaw\sh_\bc^\st(\bigcdot,\bigcdot)\,.
\end{equ}
\end{proposition}

\begin{proof}
The claim clearly holds for all finite-dimensional distributions from the scaling properties of Brownian motion.
Since these characterise the law of $\sh^\st_\bc$,~\eqref{e:eqlaw} follows at once.
\end{proof}

\begin{remark}
Note that \eqref{e:eqlaw} holds without having to quotient by vertical shifts, while this is necessary to have space-time
translation invariance.
\end{remark}

Although Definition~\ref{def:bcFD} provides a graphic description of the finite dimensional distributions of the Brownian Castle, 
we would like to obtain more explicit formulas characterising them. 
In the next proposition, we begin our analysis by studying the distribution of the increments 
at fixed time and as time goes to $+\infty$. 

\begin{proposition}\label{prop:2pointDist}
Let $\sh_0\in D(\R,\R)$, $t>0$ and $\sh_\bc$ be the Brownian Castle with initial condition $\sh_0$. 
Then, as $|x-y|$ converges to $0$
\begin{equ}[e:smallDist]
\frac{\sh_\bc(t,x)-\sh_\bc(t,y)}{x-y}\overset{\law}{\longrightarrow} \Cauchy(0,2)
\end{equ}
where, for $a\in\R$ and $\gamma>0$, $\Cauchy(a,\gamma)$ denotes a Cauchy random variable 
with location parameter $a$ and scale parameter $\gamma$.
Moreover, for any $x,y\in\R$,
\begin{equ}[e:largeTime]
\sh_\bc(t,x)-\sh_\bc(t,y)\overset{\law}{\longrightarrow} \Cauchy(0,2|x-y|)
\end{equ}
as $t\ua+\infty$. In particular, for any $x,y\in\R$ the stationary Brownian Castle $\sh_\bc^\st$ satisfies 
$\sh_\bc^\st(t,x)-\sh_\bc^\st(t,y)\eqlaw\Cauchy(0,2|x-y|)$ for any $t\geq 0$. 
\end{proposition}
\begin{proof}
The claim for $\sh_\bc^\st$ is clearly true since from its definition we have
\begin{equ}
\sh_\bc^\st(t,x)-\sh_\bc^\st(t,y) \eqlaw \CN(0,2\tau_{y-x})\;,
\end{equ}
where $\tau_r$ is the law of the first hitting time of $r$ for a standard Brownian motion starting at $0$.
Now, $\tau_r \eqlaw \Levy(0,r^2)$ and $\Cauchy(t) \eqlaw \CN(\Levy(0,t^2/2))$, which implies the result.
\eqref{e:largeTime} follows from Proposition~\ref{prop:convergenceFull}, while
\eqref{e:smallDist} can be reduced to \eqref{e:largeTime} by Proposition~\ref{p:Scaling}.
\end{proof}

We now turn our attention to the $n$-point distribution for $n\geq 3$.
Given $x_1,\dots,x_n \in \R$, we aim at deriving an expression for the characteristic function of 
$(\sH_\bc(x_1),\dots,\sH_\bc(x_n))$, where we use the shorthand $\sH_\bc = \sh_\bc^\st(0,\bigcdot)$.
By the definition of $\sH_\bc$ (and Definition~\ref{def:bcFD}), 
once the full ancestral structure of $n$ independent coalescing (backward) 
Brownian motions starting from $x_1,\dots, x_n$ respectively, is known, the conditional joint distribution 
of $(\sH_\bc(x_1),\dots,\sH_\bc(x_n))$ (modulo vertical shifts) is Gaussian and therefore easily accessible. 
In order to get our hands on the aforementioned ancestral structure we will proceed inductively using the strong Markov property 
of a finite family of coalescing Brownian motions. This will be possible if we are able to simultaneously keep 
track of the first time at which any two Brownian motions meet, which are the Brownian motions meeting and 
the position of all of them at that time. 

Let $x_1<\dots<x_n$ and let $(y_i)_{i \le n}$ be independent standard Brownian motions starting at 
$x_i$. Denote by $Z=(\tau,\iota,\Pi^{n-1})$ the $\R_+\times [n-1] \times\R^{n-1}$-valued 
random variable in which $\tau$, $\iota$ and $\Pi^{n-1}$ are defined by 
\begin{equation}\label{def:Z}
\begin{split}
\tau&\eqdef\inf\left\{t>0\,:\,\exists\,\iota \in [n-1] \text{ such that } y_\iota(t)=y_{\iota+1}(t)\right\}\;,\\
\Pi^{n-1}&\eqdef \left(y_1(\tau),\dots,y_{\iota-1}(\tau), y_{\iota+1}(\tau),\dots,y_n(\tau)\right)\;,
\end{split}
\end{equation}
where $\iota$ is implicitly defined as the (almost surely unique) value appearing in the definition of $\tau$.
The random variable $Z$ admits a density with respect to the product of the one-dimensional Lebesgue measure on $\R_+$, 
the counting measure on $[n-1]$ and the $(n-1)$-dimensional Lebesgue measure, 
as the following variant of the Karlin--McGregor formula \cite{McGregor} shows. 

\begin{lemma}\label{l:Density}
Let $Z$ be the $\R_+\times [n-1]\times\R^{n-1}$-valued random variable defined 
in~\eqref{def:Z}. Then, with the usual abuse of notation,
\begin{equation}\label{e:DensityZ}
\Prob_x \left( \tau\in \dd t,\iota=j,\Pi^{n-1}\in \dd y\right)=\det M^j_t (x,y)\,\dd y\,\dd t
\end{equation}
where  the $n\times n$ matrix $M^j$ is defined by 
\begin{equation}\label{e:Kernel}
M^j_t (x,y)_{i,k}\eqdef
\begin{cases}
p_t(y_k-x_i)\,,& \text{for $k< j$,}\\
p_t'(y_j-x_i)\,,& \text{for $k= j$, }\\
p_t(y_{k-1}-x_i)\,,& \text{for $k> j$,}
\end{cases}
\end{equation}
$p$ is the heat kernel and $p'$ its spatial derivative. 
\end{lemma}
\begin{proof}
This is an immediate corollary of Theorem~\ref{thm:Density}, combined with the Karlin--McGregor formula. 
\end{proof}

In the following proposition, we derive a recursive formula for the characteristic function of the $n$-point distribution 
of $\sH_\bc$. 

\begin{proposition}\label{prop:FDist}
Let $\sH_\bc$ be as above.
For $n\in\N$,  $\alpha=(\alpha_1,\dots,\alpha_n), \,x=(x_1,\dots,x_n) \in\R^n$ such that $\sum_j\alpha_j=0$ and 
$x_1<\dots<x_n$, let $F_n(\alpha,x)$ be the characteristic function of $(\sH_\bc(x_1),\dots,\sH_\bc(x_n))$ evaluated at $\alpha$. 
Then, $F_n$ satisfies the recursion
\begin{equation}\label{e:RecF}
F_n(\alpha,x)=\sum_{j=1}^{n-1}\int_{\R_+} \int_{y_1<\dots<y_{n-1}} e^{-\frac{1}{2}|\alpha|^2t} F_{n-1}(\c_j\alpha,y) \det M_t^j(x,y) \,\dd y\,\dd t
\end{equation}
where the $n\times n$ matrix $M^j$ was given in~\eqref{e:Kernel} and $\c_j\alpha\in\R^{n-1}$ is the vector defined by
\begin{equ}
(\c_j\alpha)_l\eqdef
\begin{cases}
\alpha_l\,,& l<j\\
\alpha_j+\alpha_{j+1}\,,& l=j\\
\alpha_{l+1}\,,&l>j\,.
\end{cases}
\end{equ}
\end{proposition}
\begin{proof}
Fix $x_1<\dots<x_n$ and consider the stochastic process $(y_i, B_{ij})_{i,j=1}^n$ where the $y_i$ are 
coalescing Brownian motions with initial conditions $y_i(0) = x_i$ and the $B_{ij}$ are given by
\begin{equ}[e:defBij]
dB_{ij}(t) = \1_{y_i \neq y_j} \bigl(dW_i(t) - dW_j(t)\bigr)\;, \qquad B_{ij}(0) = 0\;,
\end{equ}
the $W_i$ being i.i.d.\ standard Wiener processes independent of the $y_i$'s.
This process is strong Markov since so is the family of $y_i$'s (see \cite{TW}). Furthermore, since the $y_i$'s all coalesce at some time, 
the limit $B_{ij}(\infty) = \lim_{t \to \infty} B_{ij}(t)$ is well-defined and, by construction, 
one has 
\begin{equ}
(\sH_\bc(x_i) - \sH_\bc(x_j))_{ i,j=1}^n
\eqlaw (B_{ij}(\infty))_{ i,j=1}^n
\end{equ}
We now combine the strong Markov property with the fact that $B$, as defined by \eqref{e:defBij},
depends on its initial condition in an affine way with unit slope.
This implies that, writing $\tilde \sH_\bc$ for a copy of $\sH_\bc$ that is independent
of the process $(y,B)$ and $\tau$ for any stopping time, one has the identity in law
\begin{equ}
\big(B_{ij}(\infty)\big)_{i,j} \eqlaw  \big(B_{ij}(\tau) + \tilde \sH_\bc(y_i(\tau)) - \tilde \sH_\bc(y_j(\tau))\big)_{i,j}\;.
\end{equ}
Write now $(\CF_t)$ for the filtration generated by the $y_i$ and the $W_i$
and $\tau$ for the first time at which any two of the $y$'s coalesce. Using furthermore the shorthand
$\sH_\bc(x)$ for $(\sH_\bc(x_1),\ldots,\sH_\bc(x_n))$ we have 
\begin{equs}
F_n&(\alpha,x)=\E \left[\E \left[e^{i\scal{\alpha, \sH_\bc(x)}}\Big|\CF_\tau\right]\right] = \E\left[ e^{i\scal{\alpha, \tilde \sH_\bc(y(\tau))}-\frac{1}{2}|\alpha|^2\tau}\right]\\
&=\sum_{l=1}^{n-1}\int_0^\infty\int_{y_1<\dots<y_{n-1}}e^{-\frac{|\alpha|^2}{2}t} \E \left[ e^{i\scal{\alpha, \sH_\bc(y)}}\right]\Prob_x(\tau\in\dd t,\iota=l,\Pi^{n-1}\in\dd y)\\
&=\sum_{l=1}^{n-1}\int_0^\infty\int_{y_1<\dots<y_{n-1}}e^{-\frac{|\alpha|^2}{2}t} \E \left[  e^{i\scal{\c_l\alpha, \sH_\bc(y)}}\right] \det M_t^j(x,y) \,\dd y\,\dd t
\end{equs}
where $\P_x$ denotes the measure appearing in Lemma~\ref{l:Density}.
The required identity~\eqref{e:RecF} follows at once. 
\end{proof}

Thanks to the results in Sections~\ref{sec:BC} and~\ref{sec:BCprocess}, and Proposition~\ref{prop:2pointDist}, we know that 
$\sH_\bc$ has increments which are stationary and distributed according to a Cauchy 
random variable with parameter given by their lengths, and admits a version whose trajectories are 
c\`adl\`ag and have (locally) finite $p$-variation for any $p>1$ (see Proposition~\ref{p:pvar}). 
If moreover we knew that the increments were independent, 
we could conclude that $\sH_\bc$ is nothing but a Cauchy process. 

The lack of independence is already evident by formula~\eqref{e:RecF} in Proposition~\ref{prop:FDist}, therefore 
the $k$-point distributions of $\sH_\bc$ with $k > 2$ are different from those of the Cauchy process. 
In the next proposition, we actually show more, namely that the law of 
$\sH_\bc$ is a {\it genuinely new} measure on $D(\R,\R)$ since it is 
singular with respect to that of the Cauchy process. 

\begin{proposition}\label{prop:Singularity}
Let $\sH_\bc$ be as above and $\sC$ be the standard Cauchy process on 
$\R_+$. Then, when restricted to $[0,1]$, the laws of $\sH_\bc$ and $\sC$ are mutually singular.
\end{proposition}
\begin{proof}
We want to exhibit an almost sure property that distinguishes the laws of $\sH_\bc$ and $\sC$ on 
$D([0,1],\R)$. 
Let $1\geq x_0>\dots>x_m\geq \f12$, $\varphi:\R^m\to\R$ be a bounded function and $\{\lambda_n\}_n$ 
an increasing sequence in $[1,\infty)$ such that $\lambda_{n+1} \ge 4\lambda_n$ for all $n\in\N$. For $n\geq 1$, define the functional 
$\Phi_n: D([0,1],\R)\to\R$ by
\begin{equ}
\Phi_n(h)\eqdef\frac{1}{n}\sum_{k=1}^n \varphi(I_k(h))\,,\qquad\text{where}\qquad (I_k(h))_i\eqdef \lambda_k\left(h(x_i/\lambda_k)-h(x_{0}/\lambda_k)\right)
\end{equ}
for $h\in D([0,1],\R)$. By scaling invariance and the independence properties of the Cauchy process $\{I_k(\sC)\}_k$ is i.i.d.\ while, 
by Proposition~\ref{p:Scaling}, $\{I_k(\sH_\bc)\}_k$ is a sequence of identically distributed (but not independent!) 
random variables. Since furthermore $\varphi$ is bounded, 
the classical strong law of large numbers holds for 
$\{\varphi(I_k(\sC))\}_k$, which implies that almost surely 
\begin{equation}\label{e:SLLN}
\lim_{n\to\infty} \Phi_n(\sC)=\E[\Phi_1(\sC)]\,.
\end{equation} 
We claim that, provided that we choose a sequence $\{\lambda_n\}_n$ 
that increases sufficiently fast,~\eqref{e:SLLN} holds also for $\sH_\bc$. 
Before proving the claim, notice that, assuming it holds, we are done. Indeed, it suffices to 
take $m\geq 2$, and determine a function $\varphi$ such that $\E[\Phi_1(\sC)]\neq \E_\bc[\Phi_1(\sH_\bc)]$. Such a function 
clearly exists since by Proposition~\ref{prop:FDist} $\sC$ and $\sH_\bc$ have different $n$-point distributions for $n\geq 3$. 

We now turn to the proof of the claim. 
We will construct two sequences $\{\hat J_k\}_k$ and $\{J_k\}_k$ of $\R^m$-valued random variables such that 
\begin{equ}
\hat J_1 = J_1\;,\qquad
\{J_k\}_{k \in \N} \eqlaw \{I_k(\sH_\bc)\}_{k \in \N}\;,
\end{equ}
and the sequence $\hat J_k$ is i.i.d. 
Arguing as for the Cauchy process,~\eqref{e:SLLN} holds for $\{\varphi(\hat J_k)\}_k$ hence the claim follows if we can 
build $\{\hat J_k\}_k$ and $\{J_k\}_k$ in such a way that, almost surely, $\hat J_k=J_k$ for all but finitely many values of $k$. 

Let $\{W_{k,i}\,:\, k \in \N, i \in \{0,\ldots,m\}\}$ be a collection 
of i.i.d.\ standard Wiener processes and $z_{k,i} = (0, x_{k,i})$ be points with $x_{k,i} = x_i/\lambda_k$.
We use them in two 
different ways. First, we apply the construction given at the start of Section~\ref{sec:BW} 
to each of the groups $\{(W_{k,i}, z_{k,i})\}_{i \le m}$ separately, which 
yields a collection $\{\zeta_k = (\ST_k,\ast_k,d_k, M_k)\}_{k \in \N}$ of characteristic spatial $\R$-trees 
with each $\zeta_k$ representing coalescing Brownian motions starting from $\{z_{k,i}\}_{i \le m}$
and $M_k(\ast_k) = z_{0,k}$. We then apply the construction to the whole collection at once,
taken in lexicographical order, so to obtain one ``big'' spatial $\R$-tree $\zeta = (\ST,d,\ast, M)$.
Let $\hat\sz_{k,i} \in \ST_k$ and $\sz_{k,i} \in \ST$ be the unique points such that $M_k(\hat\sz_{k,i}) = z_{k,i}$
and $M(\sz_{k,i}) = z_{k,i}$, respectively. Write furthermore
\begin{equ}
\tau_k = \sup\{t < 0\,:\, \rho(\sz_{k,0},t) = \rho(\sz_{k-1,m},t)\}\;.
\end{equ}
Since both $\{\zeta_k\}_k$ and $\zeta$ are built via the same Brownian motions, we clearly have 
$M_k(\rho(\hat\sz_{k,i},t)) = M(\rho(\sz_{k,i},t))$ for all $i \le m$ and
all $t \in [\tau_k,0]$. Denote by $\bar \ST \subset \ST$ the subspace given by 
\begin{equ}
\bar \ST = \{\rho(\sz_{k,i},t)\,:\, t \in [\tau_k,0],\; i \le m,\; k \in \Z\}\;,
\end{equ}
and similarly for $\bar \ST_k \subset \ST_k$, so that there is a canonical bijection 
$\iota \colon \bigcup_{k \in \N} \bar \ST_k \to \bar \ST$.

We now turn to the branching maps. 
Fix independent Brownian motions $B_k$ on each of the $\ST_k$ and 
write $\tilde B \colon \bigcup_{k \in \N} \bar \ST_k \to \R$ for the map that restricts to $B_k$ on each $\ST_k$. 
We then construct a Brownian motion $B$ on
$\ST$ such that
\begin{claim}
\item Writing $\bar B$ for the restriction of $B$ to $\bar \ST$, one has $\delta \bar B(\sz,\sz') = \delta \tilde B(\iota \sz,\iota \sz')$.
\item Conditionally on the $W$'s, all the increments $B(\sz) - B(\sz')$ are independent of all the $B_k$'s
for any $\sz,\sz' \in \ST \setminus \bar \ST$.
\end{claim}
The independence properties of Brownian motion mentioned in Remark~\ref{rem:contBM} guarantee that such a construction is 
possible and uniquely determines the law of $B$, conditional on the $W_{k,i}$'s and the $B_k$'s.
We claim that setting
\begin{equ}
J_{k,i} = \lambda_k \big(B(\sz_{k,i})- B(\sz_{k,0})\big)\;,\quad
\hat J_{k,i} = \lambda_k \big(B_k(\hat \sz_{k,i})- B_k(\hat \sz_{k,0})\big)\;,
\end{equ}
the sequences $\{\hat J_k\}_k$ and $\{J_k\}_k$ satisfy all the desired properties mentioned above.

Clearly the $\hat J_k$ are independent and they are identically distributed by Brownian scaling. 
The fact that the $J_{k,i}$ are distributed like $I_k(\sH_\bc)$ is also immediate from the construction, so it remains to
show that $\hat J_k = J_k$ for all but finitely many values of $k$. 
Writing
\begin{equ}
\hat \tau_k = \sup\{t < 0\,:\, \rho(\hat\sz_{k,0},t) = \rho(\hat\sz_{k,m},t)\}\;,
\end{equ}
we have that $\hat J_k = J_k$ as soon as $|\hat \tau_k| \le |\tau_k|$.
For $k > 0$, 
\begin{equs}
|\tau_k| &\eqlaw \Bigl({x_m \over \lambda_{k-1}} - {x_0 \over \lambda_{k}}\Bigr)^2\Levy(1)
\ge (4\lambda_{k-1})^{-2}\Levy(1)\;,\\
|\hat \tau_k| &\eqlaw \Bigl({x_0-x_m\over \lambda_k}\Bigr)^2 \Levy(1) \le \lambda_k^{-2}\Levy(1)\;,
\end{equs}
where we use the fact that $\lambda_k \ge 4 \lambda_{k-1}$ in the first inequality. 
If we choose $0 < c_k < C_k < \infty$ such that
\begin{equ}
\P(\Levy(1) < c_k) \le k^{-2}\;,\qquad
\P(\Levy(1) > C_k) \le k^{-2}\;,
\end{equ}
we can conclude that $\P(|\hat \tau_k| \ge |\tau_k|) \le 2k^{-2}$ provided that we choose 
the $\lambda_k$'s in such a way that $c_k \lambda_k^2 \ge  16 C_k \lambda_{k-1}^2$. 
Hence, by Borel-Cantelli we have $\hat J_k = J_k$ for all but finitely many $k$'s and the proof is concluded. 
\end{proof}
%
%
%
%
%
%
%

\section{Convergence of $\boldsymbol{0}$-Ballistic Deposition}\label{sec:0BD}

In this last section, we show that the $0$-Ballistic Deposition model does indeed converge to the 
Brownian Castle. In order to prove Theorem~\ref{thm:convTime}, we will begin by associating 
to $0$-BD a branching spatial $\R$-tree and 
prove that, when suitably rescaled, the 
law of the latter converges to $\cP_\bc$ defined in~\eqref{e:LawBC}. 


\subsection{The Double Discrete Web tree}
\label{sec:graphical}

We begin our analysis by recalling the construction and the results obtained in~\cite[Section 4]{CHbwt}, 
concerning the spatial tree representation of a family of coalescing backward random walks and its dual. 

Let $\delta\in(0,1]$ and $(\Omega,\cA,\P_\delta)$ be a standard probability space supporting 
four Poisson random measures, $\mu_{\gamma}^L$, $\mu_{\gamma}^R$, 
$\hat\mu_{\gamma}^L$ and $\hat\mu_{\gamma}^R$.  
The first two, $\mu_{\gamma}^L$ and $\mu_{\gamma}^R$, live on $\D^\da_\delta\eqdef\R\times\delta\Z$, 
are independent and 
have both intensity $\gamma\lambda$, 
where, for every $k\in\delta\Z$, $\lambda(\dd t, \{k\})$ is a copy of the Lebesgue measure on $\R$ and
throughout the section 
\begin{equ}[e:gamma]
\gamma=\gamma(\delta)\eqdef\frac{1}{2\delta^2}\,.
\end{equ} 
The others live on $\D^\ua_\delta\eqdef\R\times\delta(\Z+1/2)$, and are obtained from the formers by setting, 
for every measurable $A\subset\D^\ua_\delta$
\begin{equ}[e:DualPoisson]
\hat\mu^L_\gamma(A)\eqdef\mu^R_\gamma(A-\delta/2)\qquad\text{and}\qquad 
\hat\mu^R_\gamma(A)\eqdef\mu^L_\gamma(A+\delta/2)\,.
\end{equ}
Here, $A\pm\delta/2\eqdef\{z\pm(0,\delta/2)\,:\,z\in A\}$. 

Representing the previous Poisson processes via arrows as in Figure~\ref{f:PPEvalMap}, it 
is not hard to define two families of coalescing random walks $\{\pi^{\da,\delta}\}_{z\in\D^\da_\delta}$ 
and $\{\pi^{\ua,\delta}\}_{z\in\D^\ua_\delta}$, the first running backward in time and the second, forward, 
in such a way they never cross (see~\cite[Section 4.1]{CHbwt} and in particular Figure 1 therein). 
Thanks to these, we can state the following definition, which is taken from~\cite[Definition 4.1]{CHbwt}. 
 
\begin{definition}\label{def:DWT}
Let $\delta\in(0,1]$, $\gamma$ as in~\eqref{e:gamma}, 
$\mu_{\gamma}^L$ and $\mu_{\gamma}^R$ be two independent Poisson random measures on
$\D^\da_\delta$ of intensity $\gamma\lambda$, $\hat\mu^L$ and $\hat\mu^R$ be given as in~\eqref{e:DualPoisson} 
and $\{\pidd_z\}_{z\in\D^\da_\delta}$ and $\{\piud_{\hat z}\}_{\hat z\in \D^\ua_\delta} $ be the 
families of coalescing random walks introduced above. We define the {\it Double Discrete Web Tree} as 
the couple $\zeta^{\uda}_\delta\eqdef(\zeta^\da_\delta, \zeta^\ua_\delta)$, in which
\begin{itemize}[noitemsep,label=-]
\item $\zeta^{\da}_\delta\eqdef(\ST^{\da}_\delta, \ast^{\da}_\delta, d^{\da}_\delta, M^{\da}_\delta)$ is given by setting $\ST^{\da}_\delta = \D^\da_\delta$, $\ast^{\da}_\delta = (0,0)$, $M^{\da}_\delta$ the canonical inclusion, and 
\begin{equ}[e:defdd]
d^{\da}_\delta(z,\bar z) = t + t' - 2 \sup \{ s \le t \wedge t' \,:\, \pidd_z(s) = \pidd_{\bar z}(s)\}\;.
\end{equ}
\item $\zeta^{\ua}_\delta\eqdef(\ST^{\ua}_\delta, \ast^{\ua}_\delta, d^{\ua}_\delta, M^{\ua}_\delta)$ is built 
similarly, but with $\ast^\ua_{\delta} = (0,\delta/2)$ 
and the supremum in \eqref{e:defdd} replaced by $\inf \{ s \ge t \vee t' \,:\, \piud_z(s) = \piud_{\bar z}(s)\}$.
\end{itemize}
\end{definition}

As was pointed out in~\cite{CHbwt}, the Discrete Web Tree and its dual are not {\it spatial} trees since 
the random walks $\pi^{\da,\delta}$, $\pi^{\ua,\delta}$ are not continuous. 
To overcome the issue, in~\cite[Section 4.1]{CHbwt}, it was introduced a suitable modification 
$\tilde\zeta^{\uda}_\delta$ of $\zeta^{\uda}_\delta$, called the {\it Interpolated Double Discrete Tree} 
(see~\cite[Definition 4.2]{CHbwt}). 
Simply speaking, the latter was obtained by keeping the same pointed $\R$-trees 
$(\ST^{\dotp}_\delta, \ast^{\dotp}_\delta, d^{\dotp}_\delta)$, $\dotp\in\{\ua,\da\}$, 
as those of the Double Discrete Tree
and defining new evaluation maps $\tilde M^\da_\delta$ and $\tilde M^\ua_\delta$ by interpolating 
the discontinuities of $M^\da_\delta$ and $M^\ua_\delta$ in such a way that 
the properties~\ref{i:Back},~\ref{i:MonSpace} (and~\ref{i:Spread}) of Definition~\ref{def:CharTree} were kept 
and continuity was restored.

The following result is a consequence of Propositions 4.3 and 4.4, and Theorem 4.5 in~\cite{CHbwt}. 

\begin{theorem}\label{thm:DWTConv}
For any $\delta\in(0,1)$ and $\alpha\in(0,1/2)$, the Interpolated Double Discrete Web Tree $\tilde\zeta^{\uda}_\delta$ 
introduced above almost surely belongs to $\Ch^\alpha_\Sp\times\hat\Ch^\alpha_\Sp$ and satisfies~\ref{i:TreeCond} 
of Definition~\ref{def:TreeCond}. 

Let $\Theta^{\uda}_\delta$ be the law of $\tilde\zeta^{\uda}_\delta$ on $\Ch^\alpha_\Sp\times \hat\Ch^\alpha_\Sp$, 
with marginals $\Theta^{\da}_\delta$ and $\Theta^{\ua}_\delta$. 
Then, $\Theta^\da_\delta$ is tight in $\fE^\alpha(\theta)$ (see~\eqref{def:MeasSet}) 
for any $\theta>\tfrac32$ and, as $\delta\da 0$, 
$\Theta^{\uda}_\delta$ converges to the law of the Double Web Tree 
$\Theta^{\uda}_\bw$ of Definition~\ref{def:DWT} weakly on $\Ch^\alpha_\Sp\times \hat\Ch^\alpha_\Sp$. 

At last, almost surely, for $\dotp\in\{\ua,\da\}$
\begin{equ}[e:Dist_p_Mp]
\sup_{\sz\in\ST^{\dotp}_\delta}\|\tilde M^{\dotp}_\delta(\sz)-M^{\dotp}_\delta(\sz)\|\leq \delta
\end{equ}
where $M^{\dotp}_\delta$ are the evaluation maps of the double Discrete Web Tree in Definition~\ref{def:DWT}.
\end{theorem}

\subsection{The  $\boldsymbol{0}$-BD measure and convergence to BC measure}\label{sec:0BDconv}

We are now ready to introduce the missing ingredient in the construction of the $0$-BD tree, 
namely the Poisson random measure responsible of increasing the height function by $1$.

Let $\delta\in(0,1]$, $\gamma$ as in~\eqref{e:gamma}, 
$\mu_{\gamma}^L$ and $\mu_{\gamma}^R$ be as in the previous section and 
$\mu_{\gamma}^\bullet$ be a Poisson random measure on $\D_\delta^\da$ of intensity 
$2\gamma\lambda$ and  independent of both $\mu_{\gamma}^L$ and $\mu_{\gamma}^R$. 

For a typical realisation of $\mu_{\gamma}^L$ and $\mu_{\gamma}^R$, 
consider the Discrete Web Tree $\zeta^\da=(\ST^\da_\delta,\ast_\delta,d^\da_\delta, M^\da_\delta)$ 
in Definition~\ref{def:DWT}. Let $\mu_{\gamma}$ be the measure on $\ST^\da_\delta$ induced by 
$\mu_{\gamma}^\bullet$ via 
\begin{equ}[e:numeasure]
\mu_{\gamma}(A)\eqdef \mu_{\gamma}^\bullet( M^\da_\delta(A))
\end{equ}
for any $A$ Borel subset of $\ST^\da_\delta$. 
Notice that  $M^\da_\delta$ is bijective on $\D_\delta^\da$ 
so that $\mu_{\gamma}$ is well-defined and, since $\mu_{\gamma}^\bullet$ is independent over 
disjoint sets, $\mu_{\gamma}$ is distributed according to 
a Poisson random measure on $\ST^\da_\delta$ of intensity $2\gamma\ell$, 
$\ell$ being the length measure on $\ST^\da_\delta$ (see~\eqref{def:LengthMeasure}. 

\begin{definition}\label{def:0-BDtree}
Let $\delta\in(0,1)$, $\gamma$ as in~\eqref{e:gamma}, 
$\mu_{\gamma}^L$, $\mu_{\gamma}^R$ and $\mu_{\gamma}^\bullet$ be 
three independent Poisson random measures on $\D_\delta$ of respective intensities 
$\gamma\lambda$, $\gamma\lambda$, and $2\gamma\lambda$. We define the {\it $0$-BD Tree} as the couple 
$\chi^\delta_{\bd}\eqdef(\zeta^\da_\delta, N_{\gamma})$ in which $\zeta^\da_\delta$ is as 
in Definition~\ref{def:DWT} while $N_{\gamma}$ is the rescaled compensated Poisson process given by
\begin{equ}[e:RCPP]
N_{\gamma}(\sz)\eqdef \delta \big(\mu_\gamma(\llb\ast,\sz\rrb) -2\gamma d_\delta^\da(\sz,\ast)\big)\;.
\end{equ}
and $\mu_{\gamma}$ is the measure given in~\eqref{e:numeasure}. 
\end{definition}

As before, the $0$-BD tree also fails to be a branching spatial tree since neither 
the evaluation nor the branching map are continuous. The remedy here was already 
presented in Section~\ref{sec:MEC} where we introduced the smoothened RC Poisson process. 

\begin{remark}\label{rem:Borel}
We can view the triple $(\mu_{\gamma}^L,\mu_{\gamma}^R,\mu_{\gamma}^\bullet)$ as an element of 
the space of locally finite integer-valued measures endowed with the topology of vague convergence.
All functions of $\chi^\delta_{\bd}$ mentioned later on are Borel measurable with respect to 
this.
\end{remark}

\begin{definition}\label{def:Smooth0-BDtree}
In the setting of Definition~\ref{def:0-BDtree} and for $p>2$ and $a=(2\gamma)^{-p}$, 
we define the {\it smoothened $0$-BD Tree} as the couple 
$\tilde\chi^\delta_{\bd}\eqdef(\tilde\zeta^\da_\delta, N^a_{\gamma})$ in which $\tilde\zeta^\da_\delta$ is the 
Interpolated Discrete Web Tree of Section~\ref{sec:graphical} while $N^a_{\gamma}$ is the RCS
Poisson process of Definition~\ref{def:Poisson} associated to the Poisson random measure
$\mu_{\gamma}$ given in~\eqref{e:numeasure}. 
\end{definition}

\begin{proposition}\label{p:0-BDTreeisChar}
For any $\delta\in(0,1]$ and $\alpha,\,\beta\in(0,1)$ the smoothened $0$-BD Tree 
$\tilde\chi^\delta_{\bd}=(\tilde\zeta^\da_\delta, N^a_{\gamma})$ in Definition~\ref{def:Smooth0-BDtree} 
is almost surely a characteristic $(\alpha,\beta)$-branching spatial pointed $\R$-tree. 
Its law $\cP^\delta_{\bd}$ on $\Ch^{\alpha,\beta}_\bsp$, which we call the {\bf $\boldsymbol{0}$-BD measure}, 
can be written as 
\begin{equ}[e:0-BDMeasure]
\cP^\delta_{\bd}(\dd\chi)\eqdef \int \CQ^{\Poi_{\gamma}}_\zeta(\dd \chi)\Theta^\da_\delta(\dd\zeta)
\end{equ} 
where $\Theta^\da_\delta$ denotes the law of $\tilde\zeta^\da_\delta$ in Theorem~\ref{thm:DWTConv} on 
$\Ch^\alpha_\Sp$ and $\CQ^{\Poi_{\gamma}}_\zeta$ that of the RCS Poisson process $N^a_{\gamma}$ on 
$\Ch^{\alpha,\beta}_\bsp$. 

Moreover, almost surely~\eqref{e:Dist_p_Mp} holds and for every $r>0$ there exists a 
constant $C=C(r)>0$ such that for all $\delta$ small enough and $k>p/(p-2)$
\begin{equ}[e:DistBM]
\Prob_\delta\Big(\sup_{\sz\in(\ST^\da_\delta)^{(r)}}|N^a_{\gamma}(\sz)-N_{\gamma}(\sz)|> k \delta\Big)\leq C\delta\;,
\end{equ}
where $N_{\gamma}$ is the branching map of $\tilde\chi^\delta_{\bd}$ in Definition~\ref{def:0-BDtree}.
\end{proposition}
\begin{proof}
The first part of the statement is a direct consequence of the fact that almost surely
$\tilde\zeta^\da_\delta\in\fE_\alpha$ since $\ST^\da_\delta$ is almost surely locally finite, 
so that Lemma~\ref{l:PPtree} applies, 
while the measurability conditions required to make sense of~\eqref{e:0-BDMeasure} 
are implied by Proposition~\ref{p:MeasG}. 
Moreover,~\eqref{e:Dist_p_Mp} follows by Theorem~\ref{thm:DWTConv} so that we are left to show~\eqref{e:DistBM}. 

Let us recall the definition of the event $E_R^\delta$ given in~\cite[Proposition 4.4]{CHbwt} 
(see also Proposition 3.2 Eq. (3.4) in the same reference). 
For $r\geq 1$ and $R>r$, let $Q_R^\pm$ be two squares of side $1$ centred at $(r+1,\pm(2R+1))$, 
$z^\pm=(t^\pm,x^\pm)$ be two points in the interior of $Q_R^\pm$ and 
$\{z^{\pm}_\delta\}_\delta\subset Q^\pm_R\cap (\D_\delta)$ be sequences converging to $z^\pm$. 
We set 
\begin{equ}
E_R^\delta\eqdef\{\sup_{0\geq s\geq -r}|\pi^{\da,\delta}_0(s)|\leq R\,,\,\sup_{t_\delta^\pm\geq s\geq -r}|\pi^{\da,\delta}_{z_\delta^\pm}(s)-x_\delta^\pm|\leq R\}
\end{equ}
Notice that, as was shown in~\cite[Eq. (4.9) and (3.5)]{CHbwt}, 
\begin{equ}[e:RP]
\liminf_{\delta\da 0} \Prob_\delta(E_R^\delta)\geq1-\frac{\sqrt{r}}{R} e^{-\frac{R^2}{2r}}\,.
\end{equ} 
Moreover, on $E_R^\delta$, 
$\tilde M^\da_\delta((\ST^\da_\delta)^{(r)})\subset\Lambda_{r,R}^\delta\eqdef([-r,r]\times[-3R-1,3R+1])\cap \D^\da_\delta$. 

Now, for any positive integer $k$, $\sup_\sz |N^a_{\gamma}(\sz)-N_{\gamma}(\sz)|>k (2\gamma)^{-1/2}$
only if there exists $\sz\in (\ST^\da_\delta)^{(r)}$ and a neighbourhood of $\sz$ of size $a=(2\gamma)^{-p}$ 
which contains more than $k$ $\mu_\gamma$-points. By the definition of $\mu_\gamma$ in~\eqref{e:numeasure}, 
this implies that there must exist $i=0,\dots, \lceil 2r a^{-1}\rceil$ such that the rectangle 
$([t_{i+1},t_i]\times\R)\cap \Lambda_{r,R}^\delta$, $t_i\eqdef r-2 ia$, contains at least 
$k$ $\mu^\bullet_\gamma$-points. 
These considerations together with~\eqref{e:RP}, lead to the bound 
\begin{equs}
\P\Big(\sup_{\sz\in(\ST^\da_\delta)^{(r)}}|N^a_{\gamma}(\sz)-N_{\gamma}(\sz)|>k\gamma^{-\tfrac12}\Big)\lesssim 
\frac{\sqrt{r}}{R} e^{-\frac{R^2}{2r}} + 2r \delta^{-2p}(4\delta^{2p-3} R)^k\,.
\end{equs}
Therefore, taking $R=\delta^{-1}$, choosing $k$ sufficiently large so that $k>p/(p-2)$, and 
then $\delta$ sufficiently small~\eqref{e:DistBM} follows at once. 
\end{proof}

We are now ready to show that the law of smoothened $0$-BD tree converges to the Brownian Castle measure 
$\cP_\bc$ of Theorem~\ref{thm:BSPT}. 

\begin{theorem}\label{thm:0-BDConv}
Let $\alpha\in(0,1/2)$ and, for $\delta\in(0,1]$ and $p>1$, $a=(2\gamma)^{-p}$ and 
$\cP^\delta_{\bd}$ be the law of the smoothened $0$-BD tree 
given in~\eqref{e:0-BDMeasure}. Then, as $\delta\da 0$, $\cP^\delta_{\bd}$ converges to $\cP_\bc$ weakly on 
$\Ch^{\alpha,\beta}_\bsp$, for any $\beta<\beta_\Poi=\frac{1}{2p}$. 
\end{theorem}
\begin{proof}
Since $\Ch^{\alpha,\beta}_\bsp$ is a metric space,
it suffices to show convergence when testing against any Lipschitz continuous bounded function $F$.
By~\eqref{e:LawBC} and~\eqref{e:0-BDMeasure}, 
we see that $|\int F(\chi)(\cP^\delta_{\bd}-\cP_\bc)(d\chi)| \le I_1 + I_2$ with 
\begin{equs}
I_1&\eqdef\Big|\int\int F(\chi)\Big(\CQ^{\Poi_{\gamma}}_\zeta(\dd \chi)-\CQ^{\Gau}_\zeta(\dd \chi)\Big)\Theta^\da_\delta(\dd\zeta)\Big|\\
I_2&\eqdef\Big|\int \Big(\int F(\chi)\CQ^{\Gau}_\zeta(\dd \chi)\Big)\Big(\Theta^\da_\delta(\dd\zeta)-\Theta^\da_\bw(\dd\zeta)\Big)\Big|\,.
\end{equs}
Since $\Theta^\da_\delta$ converges by Theorem~\ref{thm:DWTConv}, for every $\eps>0$  we can find
a compact subset $K_\eps \subset \Ch^\alpha_\Sp$ with 
$\sup_{\delta>0}\Theta^\da_\delta(K_\eps)\geq 1-\eps$.
Hence, 
\begin{equs}
I_1&\leq \Big|\int_{K_\eps}\int F(\chi)\Big(\CQ^{\Poi_{\gamma}}_\zeta(\dd \chi)-\CQ^{\Gau}_\zeta(\dd \chi)\Big)\Theta^\da_\delta(\dd\zeta)\Big|+2\|F\|_\infty\eps\\
&\leq \sup_{\zeta\in K_\eps}\Big|\int F(\chi)\Big(\CQ^{\Poi_{\gamma}}_\zeta(\dd \chi)-\CQ^{\Gau}_\zeta(\dd \chi)\Big)\Big| +2\|F\|_\infty\eps\,.
\end{equs}
As $\delta\to 0$ the first term converges to $0$ by Proposition~\ref{p:PPtoGau} and, 
since the left hand side is independent of $\eps$, we conclude that $I_1\to 0$. 
Finally, Proposition~\ref{p:MeasG} and the Lipschitz continuity of $F$ imply that the map 
$\zeta\mapsto \int F(\chi)\CQ^{\Gau}_\zeta(\dd \chi)$
is continuous so that $I_2 \to 0$ by Theorem~\ref{thm:DWTConv}.
\end{proof}

\subsection{The $\boldsymbol{0}$-BD model converges to the Brownian Castle}\label{sec:Conv}

In order to establish the convergence of the $0$-Ballistic Deposition model to the Brownian Castle, 
let $\delta>0$ and $\chi^\delta_{\bd}$ the $0$-BD Tree given in Definition~\ref{def:0-BDtree}. 
Let $\sh_0^\delta\in D(\R,\R)$ and, as in~\eqref{e:BC}, set
\begin{equ}[e:Version0BD]
\sh_{\bd}^\delta(z)\eqdef \sh_0^\delta(M^\da_{\delta,x}(\rho^\da_\delta(\sT_\delta(z),0)))+ N_\gamma(\sT_\delta(z))- N_\gamma(\rho^\da_\delta(\sT_\delta(z),0))
\end{equ}
for all $z\in\R_+\times\R$, where $\sT_\delta$ is the tree map associated to $\zeta^\da_\delta$ of Definition~\ref{def:TreeM}. 
Even though $\chi_{\bd}^\delta$ is not 
a characteristic branching spatial tree,~\eqref{e:Version0BD} still makes sense and provides a version 
(say, in $D(\R_+,D(\R,\R))$) of the rescaled and centred $0$-BD in the sense that its $k$-point
distributions agree with those of $h_{\bd}^\delta$ in~\eqref{e:Scaled}. 
Before proving Theorem~\ref{thm:convTime}, let us state the following lemma which will be needed in the proof.

\begin{lemma}\label{l:Points0}
Let $\zeta^\da_\bw$ be the backward Brownian Web tree
of Definition~\ref{def:BW} and $A \subset \R$ be a fixed subset of measure $0$. Then, with probability $1$ 
\begin{equ}
\{M^\da_{\bw,x}(\rho^\da(\sz,0))\colon M^\da_{\bw,t}(\sz)>0\}\cap A=\emptyset\,.
\end{equ}
\end{lemma}
\begin{proof}
It suffices to note that, by Theorem~\ref{thm:BW}, $\zeta^\da_\bw\eqlaw \tilde\zeta^\da(\Q^2)$ 
and that,
for a Brownian motion $B$, one has $\P(B_t \in A) = 0$ for any fixed $t > 0$.
\end{proof}

\begin{proof}[of Theorem~\ref{thm:convTime}]
By Theorem~\ref{thm:0-BDConv} and Skorokhod's representation theorem, there exists a probability 
space supporting the random variables $\chi_{\bd}^{\delta_n}$, $\tilde \chi_{\bd}^{\delta_n}$, $\tilde\zeta^\ua_{\delta_n}$, $\chi_\bc$ and $\zeta^\ua_\bw$ in such a way that the following 
properties hold.
\begin{enumerate}[noitemsep]
\item The random variables $\tilde \chi_{\bd}^{\delta_n}$, $\chi_{\bd}^{\delta_n}$ and $\tilde\zeta^\ua_{\delta_n}$ 
are related by the constructions in Section~\ref{sec:graphical} and 
Definitions~\ref{def:0-BDtree} and~\ref{def:Smooth0-BDtree}.
\item Similarly, the random variables $\chi_\bc$ and $\zeta^\ua_\bw$ are related by
the construction of Definition~\ref{def:DBW}.
\item One has $\tilde \chi_{\bd}^{\delta_n} \to \chi_\bc$ and $\tilde\zeta^\ua_{\delta_n} \to \zeta^\ua_\bw$ 
almost surely in $\Ch^{\alpha,\beta}_\bsp$ and $\Ch^\alpha_\Sp$ respectively, for all 
$\alpha\in(0,1/2)$ and $\beta<\beta_{\Poi}$.
\end{enumerate}

We consider this choice of random variables fixed from now on and, in order to shorten notations, we 
will henceforth
replace $\delta_n$ by $\delta$ with the understanding that we only ever consider values of $\delta$
belonging to the fixed sequence.

We now define the countable set $D \subset \R$ appearing in the statement of the theorem as 
the set of times $t\in\R_+$ for which there is $x\in\R$ with $(t,x)\in S^\da_{(0,3)}$ 
(see Definition~\ref{def:Type}). Our goal is then to exhibit a set of full measure 
such that, for every $T\notin D$, every $R>0$ and every $\eps>0$, there exists $\delta>0$ and 
$\lambda\in\Lambda([-R,R])$ 
for which $\gamma(\lambda)\vee d^{[-R,R]}_\lambda(\sh_\bc(T,\cdot), \sh_{\bd}^\delta(T,\cdot))<\eps$. 
The proof will be divided into four steps, but before delving into the details we will need some 
preliminary considerations. 
\medskip

We henceforth consider a sample of the random variables mentioned above as given and we fix 
some arbitrary $T \notin D$ and $R, \eps > 0$. Since 
the sets $\{\chi_\bc,\,\chi_{\bd}^\delta\}_\delta$ and $\{\tilde\zeta^\ua_\bw,\tilde\zeta^\ua_\delta\}_{\delta}$ 
are compact, point 3.\ of Proposition~\ref{p:Compactness} and~\eqref{e:Dist_p_Mp} imply 
that if we choose
$r>2 \sup_{\delta}b_{\tilde\zeta^\da_\delta}(2 (R\vee T))\vee b_{\tilde\zeta^\ua_\delta}(2(R\vee T))$
then, for all $\delta$ and $\dotp\in\{\da,\ua\}$,
\begin{equ}
(M^{\dotp}_\delta)^{-1}(\{T\}\times[-R,R])\subset \ST^{\dotp,\,(r)}_\delta\,.
\end{equ}
where $M^{\dotp}_\delta$ are the evaluation maps in Definition~\ref{def:DWT} (for the non-interpolated trees). 
Invoking once more Proposition~\ref{p:Compactness}, we also know that the constant $C_r>0$ given by 
\begin{equ}
C_r\eqdef \sup\{\|M^{\dotp}_\bw\|_\alpha^{(r)},\, \|\tilde M^{\dotp}_\delta\|_\alpha^{(r)},\,\|B_\bc\|_\beta^{(r)},\,\|N^a_\gamma\|_\beta^{(r)}\colon \delta\in(0,1]\}\vee 1
\end{equ}
is finite. 
\medskip

\noindent {\bf Step 1.} As a first step in our analysis, we want to determine a set of distinct points $y_1<\dots<y_{N+1}$ 
for which the modulus of continuity of $\sh_\bc$ on $\{T\}\times[y_i,y_{i+1})$ can be easily controlled. 

Let $0<\eta_1<\eps$ be sufficiently small and $\tilde\Xi_{[-R,R]}^\da(T,T-\eta_1)$ be defined 
according to~\eqref{e:Xitilde}.  
We order its elements in increasing order, i.e. $\tilde\Xi_{[-R,R]}^\da(T,T-\eta_1)=\{x_1,\dots,x_N\}$ with 
$x_1\eqdef\min \tilde\Xi_{[-R,R]}^\da(T,T-\eta_1)$ and let $\{y_i\,:\,i=1,\dots,N+1\}$ be as in~\eqref{e:yis}. 
Since $T\notin S^\da_{(0,3)}=S^\ua_{(2,1)}$, arguing as in the proof of Proposition~\ref{p:BCcadlag}, 
there exists $t_\com\in(T-\eta_1,T)$ such that no pair of forward paths 
starting before $T-\eta_1$ and passing through $\{T-\eta_1\}\times[x_1,x_N]$ coalesces at a time $s\in(t_\com,T]$. 
For each $i=1,\dots,N$, let $\sx_i^+,\,\sx_i^-$ be the points in $(M^\ua_\bw)^{-1}(T-\eta_1,x_i)$ from which 
the right-most and left-most forward paths from $(T-\eta_1,x_i)$ depart and such that $M^\ua_\bw(\rho^\ua(\sx_i^-,T))=y_i$ and 
$M^\ua_\bw(\rho^\ua(\sx_i^+,T))=y_{i+1}$. Notice that these coincide with the right-most and 
left-most point from $(T,x_i)$ defined in Remark~\ref{rem:Right-mostUA} unless $(T,x_i)\in S^\ua_{(0,3)}$. 
%
\medskip 

\noindent {\bf Step 2.} As a second step, we would like to determine a sufficiently small $\delta$ and points $y_1^\delta<\dots<y_{N+1}^\delta$ 
which play the same role as the $y_i$'s, but for $\sh_{\bd}^\delta$, and are close to them.

Let $\eta<\tfrac12 \eta_1$ and $M\geq 1$ the number of endpoints of $\ST_\bw^{\ua,(r),\eta}$, which is 
finite by points 2.\ and 3.\ of Lemma~\ref{l:Trim}. Let  $\eta_2>0$ be such that 
\begin{equ}[e:eta2]
12C_r\eta_2^\alpha<\min\{|y_i-y_{i+1}|\colon i=1,\dots N\}\wedge\frac{|T-t_\com|}{10 M}\,.
\end{equ}
Thanks to the fact that $\Delta_\bsp(\chi_\bc, \tilde\chi_{\bd}^\delta)\vee \Delta_\Sp(\tilde\zeta^\ua_\delta, \zeta^\ua_\bc)\to0$ 
and Lemma~\ref{l:Trim}, there exists $\delta=\delta(\eta_2)>0$ and correspondences 
$\CC^{\da}$, between $\ST_\bw^{\da,(r)}$ and $\ST^{\da,(r)}_\delta$, and $\CC^{\ua}$, 
between $\ST_\bw^{\ua,(r),\eta}$ and $\ST_\delta^{\ua,(r),\eta}$ (see \eqref{def:Trim+} below), such that 
\begin{equ}[e:DistBound]
\Delta^{\com,\CC^{\da}}_\bsp(\chi_\bc^{(r)}, \tilde\chi_{\bd}^{\delta,(r)})\vee \Delta_\Sp^{\com,\CC^{\ua}}(\zeta_\bw^{\ua,(r),\eta}, \tilde\zeta_\delta^{\ua,(r),\eta})<\eta_2\,.
\end{equ}
Let us define the subtrees $T^\ua_\bw$ and $T^\ua_\delta$ 
according to~\eqref{e:PathTree} and the corresponding spatial trees $Z^\ua_\bw$ and $\tilde Z^\ua_\delta$ 
as in~\eqref{e:PathSpTrees}. 

For $1\le i \le N$ and $\dotp\in\{+,-\}$, define $\sw_i^{\dotp}\eqdef\rho^\ua(\sx_i^{\dotp}, T-\eta_1+\eta+\eta_2)$. 
Applying Lemma~\ref{l:PathTrees}, it follows from~\eqref{e:Incl} that 
$\llb \sw_i^{\dotp},\rho^\ua(\sw_i^{\dotp},T)\rrb\subset T^\ua_\bw$ and, by the definition of $T^\ua_\bw$, $T^\ua_\delta$ 
and of the path correspondence $\CC^{\ua}_\rp$ of~\eqref{e:PathCorr}, 
there exists $\sw_i^{\dotp,\delta}\in T^\ua_\delta$ such that 
$(\sw_i^{\dotp}, \sw_i^{\dotp,\delta})\in\CC^{\ua}_\rp$. 
In the following lemma, we determine the $y_i^\delta$'s and complete the second step of the proof. 

\begin{lemma}\label{l:Step2}
For $\eta_2$ as in~\eqref{e:eta2}, 
the set $Y_\delta\eqdef \{M^\ua_\delta(\rho^\ua_\delta(\sw_i^{\dotp,\delta},T))\,:\,1\le i \le N\,,\dotp\in\{+,-\}\}$ 
contains exactly 
$N+1$ points $y_1^\delta<\dots<y_{N+1}^\delta$ which satisfy 
$|y_{i+1}^\delta-y_i^\delta|\geq \tfrac13\min_i\{|y_i-y_{1+1}|\}$. Moreover, 
there exists no point $\sz_\delta\in\ST^\ua_\delta$ such that
$M^\ua_{\delta,t}(\sz_\delta)<T-\eta_1-5M \eta_2$ and, for some $i$, 
$y_i^\delta<M^\da_{\delta,x}(\rho^\ua_\delta(\sz_\delta,T))<y_{i+1}^\delta$. 
\end{lemma}
\begin{proof}
In order to verify that $Y_\delta$ has at most $N+1$ points, it suffices to show that, 
for all $i \in \{1,\ldots,N-1\}$, the rays starting at $\sw_i^{+,\delta}$ and $\sw_{i+1}^{-,\delta}$ 
coalesce before time $T$. By Lemma~\ref{l:PathTrees}, the distance between $\sw_i^{+,\delta}$ and $\sw_{i+1}^{-,\delta}$ is bounded by 
\begin{equs}
d^\ua_\delta(\sw_i^{+,\delta}, \sw_{i+1}^{-,\delta}) \leq d^\ua_\bw(\sw_i^{+}, \sw_{i+1}^{-})+4M\eta_2\leq 2(t_\com-(T-\eta_1+\eta+\eta_2)+2M\eta_2)
\end{equs}
so that if $\bar s$ is the first time at which $\rho_\delta^\ua(\sw_i^{+,\delta},\bar s)=\rho_\delta^\ua(\sw_{i+1}^{-,\delta},\bar s)$ 
then, by~\eqref{e:eta2},
\begin{equ}
\bar s= T-\eta_1+\eta+\eta_2+\tfrac12 d^\ua_\delta(\sw_i^{+,\delta}, \sw_{i+1}^{-,\delta})\leq t_\com+(4M+1)\eta_2<T\;.
\end{equ}
Hence, the cardinality of $Y_\delta$ is not bigger than $N+1$ and we can order its elements as 
$y_1^\delta\leq\dots\leq y_{N+1}^\delta$. 
To show that the inequalities are strict, notice that, again by Lemma~\ref{l:PathTrees} and~\eqref{e:Dist_p_Mp}, we have
\begin{equs}[e:boundyydelta]
|y_i&-y_i^\delta|= |M^\ua_\bw(\rho^\ua(\sw_i^{-},T))-M^\ua_{\delta,x}(\rho^\ua_\delta(\sw_i^{-,\delta},T))|\\
\leq& |M^\ua_\bw(\rho^\ua(\sw_i^{-},T))-\tilde M^\ua_{\delta,x}(\rho^\ua_\delta(\sw_i^{-,\delta},T))|\\
&+|\tilde M^\ua_{\delta,x}(\rho^\ua_\delta(\sw_i^{-,\delta},T))|-M^\ua_{\delta,x}(\rho^\ua_\delta(\sw_i^{-,\delta},T))| \lesssim C_r\eta_2^\alpha+\delta\leq \frac{1}{6}\min_i\{|y_i-y_{i+1}|\}\,.
\end{equs}
The lower bound on $|y_i^\delta-y_{i+1}^\delta|$ follows at once. 

For the second part of the statement, we argue by contradiction and assume $\sz_\delta\in\ST^\ua_\delta$ 
is such that $M^\ua_{\delta,t}(\sz_\delta)<T-\eta_1-5M\eta_2$ and 
$y_i^\delta<M^\da_{\delta,x}(\rho^\ua_\delta(\sz_\delta,T))<y_{i+1}^\delta$. 
Note that $I_{\sz_\delta}\eqdef\llb\rho^\ua_\delta(\sz_\delta,T-\eta_1+\eta),\rho^\ua_\delta(\sz_\delta,T)\rrb\subset T_\delta^\ua$ 
since all the points in the segment are at distance at least $\eta+5M\eta_2$ 
from $\sz_\delta$ and, by~\eqref{e:Incl}, $R_{\eta+5M\eta_2}(\ST^{\ua,(r)}_\delta)\subset T_\delta^\ua$. 
Hence, there exists $\sw\in T_\bw^\ua$ such that 
for all $s\in[T-\eta_1+\eta,T]$, $(\rho^\ua(\sw,s),\rho^\ua_\delta(\sz_\delta,s))\in\CC^{\ua}_\rp$. 
Now, $\sw\in T_\bw^\ua$ and the latter is contained in $\ST^{\ua,(r),\eta}_\bw$ by~\eqref{e:Incl}, 
therefore there must be a point $\bar\sw\in\ST^{\ua,(r)}_\bw$ such that $M_{\bw,t}^\ua(\bar\sw)\leq T-\eta_1$ and 
$\rho^\ua(\bar\sw,T-\eta_1+\eta)=\sw$. 
Since, by construction, all the rays in $\ST^\ua_\bw$ starting before $T-\eta_1$ must coalesce before time $t_\com$ 
and the tree is characteristic, 
$\sw$ must be such that  either $M^\ua_{\bw,x}(\rho^\ua(\sw,T-\eta_1+\eta+\eta_2))\geq M^\ua_{\bw,x}(\sw_i^+)$ 
or $M^\ua_{\bw,x}(\rho^\ua(\sw,T-\eta_1+\eta+\eta_2))\leq M^\ua_{\bw,x}(\sw_i^-)$.
Assume the first (the other case is analogous), 
then, by the coalescing property, for all $s\geq t_\com$, $\rho^\ua(\sw,s)=\rho^\ua(\sw_i^+,s)$, which means that 
$(\rho^\ua(\sw_i^+,s), \rho^\ua_\delta(\sz_\delta,s))\in\CC^{\ua}_\rp$. 
Therefore, 
\begin{equs}
d^\ua_\delta(\rho^\ua_\delta(\sz^i_\delta,T-t_\com), &\rho^\ua_\delta(\sw_i^{+,\delta},T-t_\com))=|d^\ua_\delta(\rho^\ua_\delta(\sz_\delta,T-t_\com), \rho^\ua_\delta(\sw_i^{+,\delta},T-t_\com))\\
&-d^\ua(\rho^\ua(\sw_i^+,T-t_\com),\rho^\ua(\sw_i^+,T-t_\com))|<4M\eta_2\leq T-t_\com\,.
\end{equs}
However, the segment $I_{\sz_\delta}$ cannot intersect either 
$\llb\sw_i^{-,\delta},\rho^\ua_\delta(\sw_i^{-,\delta},T)\rrb$ or 
$\llb\sw_{i}^{+,\delta},\rho^\ua_\delta(\sw_{i}^{+,\delta},T)\rrb$, since otherwise 
$M^\da_{\delta,x}(\rho^\ua_\delta(\sz_\delta,T))=y_i^\delta$ or $y_{i+1}^\delta$. 
This implies that $d^\ua_\delta(\rho^\ua_\delta(\sz^i_\delta,T-t_\com), \rho^\ua_\delta(\sw_i^{+,\delta},T-t_\com))>T-t_\com$, 
which is a contradiction thus completing the proof. 
\end{proof}

Before proceeding, let us introduce, for all $i=1,\dots,N$, the following trapezoidal regions 
$\Delta_i$ and $\Delta^\delta_i$ in $\R^2$
\begin{equs}[e:Trap]
\Delta_i&\eqdef\bigcup_{s\in[T-\eta_1+\eta+\eta_2,T]}[M^\ua_{\bw,x}(\rho^\ua(\sw_i^{-},s), M^\ua_{\bw,x}(\rho^\ua(\sw_i^{+},s))]\\
\Delta^\delta_i&\eqdef\bigcup_{s\in[T-\eta_1+\eta+\eta_2,T]}[M^\ua_{\delta,x}(\rho_\delta^\ua(\sw_i^{-,\delta},s), M^\ua_{\delta,x}(\rho_\delta^\ua(\sw_i^{+,\delta},s))]\,.
\end{equs}
By Lemma~\ref{l:Step2} and the non-crossing property of forward and backward trajectories 
(see Theorem~\ref{thm:DBW}\ref{i:Cross} and the construction of the double Discrete Web Tree in Definition~\ref{def:DWT}), 
every couple of points $\sz_1,\,\sz_2\in\sT_\bw(\Delta_i)$ and $\sz^\delta_1,\,\sz^\delta_2\in\sT_\delta(\Delta_i^\delta)$ 
satisfies
$d^\da_\bw(\sz_1,\,\sz_2)\leq 2\eta_1$ and $d^\da_\delta(\sz^\delta_1,\,\sz^\delta_2)\leq 2(\eta_1+5M\eta_2)$. 
Indeed, if there existed points $\sz^\delta_1,\,\sz^\delta_2\in\sT_\delta(\Delta_i^\delta)$, 
for which $d^\da_\delta(\sz^\delta_1,\,\sz^\delta_2)> 2(\eta_1+5M\eta_2)$, 
then the paths $M^\da_\delta(\rho^\da_\delta(\sz^\delta_i,\cdot))$ would coalesce before $T-\eta_1-5M\eta_2$. 
This in turn would imply the existence of a forward path starting before  
$T-\eta_1-5M\eta_2$ at a position $x$ lying in between 
the two trajectories, which, 
because of the non-crossing property, at time $T$ would be located strictly between $y_i^\delta$ and $y_{i+1}^\delta$, 
thus contradicting the above lemma. 
\medskip

\noindent {\bf Step 3.} In this third step, we want to show that for every $i$ we can find a couple 
$(\sz^i,\sz^i_\delta)\in\CC^{\da}$ such that $\sz^i_\delta\in \sT_\delta(\Delta^\delta_i)$ and 
$\rho^\da(\sz^i,\bar s)\in \sT(\Delta_i)$ for some
a $\bar s$ sufficiently close to $T$. 

Let $i\in \{1,\dots,N\}$ and $\sz_\delta^i\in \ST^\da_\delta$ be such that $M^\da_\delta(\sz^i_\delta)=(T,x)$ and 
$x\in(y^\delta_i+6C_r\eta_2^\alpha, y_{i+1}^\delta-6C_r\eta_2^\alpha)$, 
which exists thanks to Lemma~\ref{l:Step2} if we choose $\eta_2$ as in~\eqref{e:eta2}. 
Clearly, $\sz^i_\delta\in\sT_\delta(\Delta_i^\delta)$. 

Let $\sz^i\in\ST^\da_\bw$ be such that $(\sz^i,\sz^i_\delta)\in\CC^{\da}$. 
If $M^\da_{\bw,t}(\sz^i)>T$, 
since $d^\da_\bw(\sz^i,\rho^\da(\sz^i,T))=|M^\da_{\bw,t}(\sz^i)-M^\da_{\delta,t}(\sz_\delta^i)|<\eta_2$ 
(the last inequality being a consequence of~\eqref{e:DistBound})
we have $|M^\da_{\bw,x}(\sz^i)-M^\da_{\bw,x}(\rho^\da(\sz^i,T))|\leq C_r\eta_2^\alpha$. Hence
\begin{equs}
|M^\da_{\bw,x}(\rho^\da(\sz^i,T)-y_i|&\geq |y_i^\delta-M^\da_{\delta,x}(\sz_\delta^i)|-|M^\da_{\delta,x}(\sz_\delta^i)-M^\da_{\bw,x}(\sz^i)|\\
&\quad -|M^\da_{\bw,x}(\sz^i)-M^\da_{\bw,x}(\rho^\da(\sz^i,T))|\\
&\geq |y_i^\delta-M^\da_{\delta,x}(\sz_\delta)|-\eta_2-\|M\|_\alpha^{(r)}\eta_2^\alpha>0
\end{equs}
where the last passage holds thanks to our choice of $\eta_2$ in~\eqref{e:eta2}, 
and the same result can be shown upon replacing 
$y_{i+1}$ to $y_i$. 
If instead $M^\da_{\bw,t}(\sz^i)\leq T$, by the H\"older continuity of $M^\ua_\bw$, 
\begin{equs}
\sup_{s\in[T-\eta_2,T]}|M^\ua_\bw(\rho^\ua(\sw_i^{-},s))-y_i|\vee |M^\ua_\bw(\rho^\ua(\sw_i^{+},s))-y_{i+1}|&\leq C_r \eta_2^\alpha\,.
\end{equs}
so that we can argue as above and show 
$|M^\da_{\bw,x}(\sz)-M^\ua_\bw(\rho^\ua(\sw_i^{-},t))|\wedge |M^\da_{\bw,x}(\sz)-M^\ua_\bw(\rho^\ua(\sw_i^{+},t))|>0$. 

As a consequence of the coalescing property and the previous bounds, 
for all points $\sw\in \sT_\bw(\Delta_i)$ and $\sw_\delta\in \sT_\delta(\Delta^\delta_i)$ we have 
\begin{equ}[e:DistanceBounds]
d^\da_\bw(\sw,\sz^i)\leq \eta_2 +2\eta_1\,,\qquad d^\da_\delta(\sw_\delta,\sz_\delta^i)\leq 2(\eta_1+5M\eta_2)\,.
\end{equ}

\noindent{\bf Step 4.} We can now exploit what we obtained so far, go back to the height functions $\sh_\bc$ and $\sh_{\bd}^\delta$, 
and complete the proof. 
First, let $\lambda:\R\to\R$ be the continuous function such that 
$\lambda(y_i)=y_i^\delta$ for all $i$, 
interpolating linearly between these points, and $\lambda'(x) = 1$ for $x \not\in [y_1,y_{N+1}]$.
In particular, one has $\lambda(x)\in[y_i^\delta,y_{i+1}^\delta)$ if $x\in[y_i,y_{i+1})$. 
Note that as a consequence of \eqref{e:boundyydelta}, we can choose $\eta_1$ and $\delta$ 
sufficiently small so that $\gamma(\lambda) \le \eps$. 
  
Let $x\in [y_i,y_{i+1})$. Then, 
\begin{equs}
|\sh_\bc(T,x)&-\sh_{\bd}^\delta(T,\lambda(x))|\leq|\sh_0(M^\da_{\bw,x}(\rho^\da(\sz^i,0)))-\sh_0^{\da,\delta}(M^\da_{\delta,x}(\rho^\da_\delta(\sz_\delta^i,0)))| \label{e:FinalB}\\
&+ |B_\bc(\sT(T,x))-N_\gamma(\sT_\delta(T,\lambda(x)))|
+|B_\bc(\rho^\da(\sz^i,0))-N_\gamma(\rho^\da_\delta(\sz^i_\delta,0))|\,,
\end{equs}
where we chose $\eta_1$ and $\eta_2$ sufficiently small, so that $T-\eta_1-5M\eta_2>0$ and consequently, 
by~\eqref{e:DistanceBounds}, $\rho^\da(\sT(T,x),0)=\rho^\da(\sz^i,0)$ and 
$\rho^\da_\delta(\sT_\delta(T,\lambda(x)),0)=\rho^\da_\delta(\sz^i_\delta,0)$. 
For the second term in~\eqref{e:FinalB} we exploit the H\"older continuity of $B_\bc$ 
and $N^a_\gamma$,~\eqref{e:DistBM},~\eqref{e:DistanceBounds} and~\eqref{e:DistBound} which give
\begin{equs}
|B_\bc(\sT(T,x))&-N_\gamma(\sT_\delta(T,\lambda(x)))| \leq |B_\bc(\sT(T,x))-B_\bc(\sz^i)| +|B_\bc(\sz^i)-N^a_\gamma(\sz_\delta^i)|\\
&+|N^a_\gamma(\sz_\delta^i)-N^a_\gamma(\sT_\delta(T,\lambda(x)))|
+|N^a_\gamma(\sT_\delta(T,\lambda(x)))-N_\gamma(\sT_\delta(T,\lambda(x)))|\\
&\lesssim (\eta_2+2\eta_1)^\beta+\eta_2+(\eta_1+5M\eta_2)^\beta+\delta\,.\label{e:BranchBound}
\end{equs} 
For the last term in~\eqref{e:FinalB}, arguing as in the proof of~\cite[Lemma 2.28]{CHbwt} (replacing 
$M_1$ and $M_2$ by $B$ and $N^a_\gamma$ in the statement) and using~\eqref{e:DistBM}, we have 
\begin{equs}
|B_\bc(\rho^\da(\sz^i,0))-N_\gamma(\rho^\da_\delta(\sz^i_\delta,0))|\leq& |B_\bc(\rho^\da(\sz^i,0))-N^a_\gamma(\rho^\da_\delta(\sz^i_\delta,0))| \label{e:Branch0}\\
&+|N^a_\gamma(\rho^\da_\delta(\sz^i_\delta,0))-N_\gamma(\rho^\da_\delta(\sz^i_\delta,0))|\lesssim C_r\eta_2^\beta+\delta\,.
\end{equs}
It remains to treat the initial condition. To do so, we make use of~\eqref{e:DistBound},~\cite[Lemma 2.28]{CHbwt} 
and~\eqref{e:Dist_p_Mp}, which give
\begin{equs}[e:initial]
|M^\da_{\bw,x}(\rho^\da(\sz^i,0))-&M^\da_{\delta,x}(\rho^\da_\delta(\sz_\delta^i,0))|\leq |M^\da_{\bw,x}(\rho^\da(\sz^i,0))-\tilde M^\da_{\delta,x}(\rho^\da_\delta(\sz_\delta^i,0))|\\
&+|\tilde M^\da_{\delta,x}(\rho^\da_\delta(\sz_\delta^i,0))- M^\da_{\delta,x}(\rho^\da_\delta(\sz_\delta^i,0))|\lesssim C_r\eta_2^\alpha+\delta\,.
\end{equs}
Now, by Lemma~\ref{l:Points0}, with probability $1$ we have 
\begin{equ}
\{M^\da_{\bw,x}(\rho^\da(\sz,0))\colon M^\da_{\bw,t}(\sz)>0\}\cap\Disc(\sh_0)=\emptyset\,.
\end{equ}
In particular, for all $i$, $M^\da_{\bw,x}(\rho^\da(\sz^i,0))$ is a continuity point of 
$\sh_0$ and by assumption $d_\Sk(\sh_0^\delta, \sh_0)\to 0$. 

By choosing $\eta_1$, $\eta_2$ and $\delta$ sufficiently small, 
we can guarantee on the one hand that each of~\eqref{e:BranchBound} and~\eqref{e:Branch0} is smaller than $\eps/3$, 
while on the other that the distance between 
$M^\da_{\bw,x}(\rho^\da(\sz^i,0))$ and $M^\da_{\delta,x}(\rho^\da_\delta(\sz_\delta^i,0))$ is arbitrarily small. 
This, together with the fact that by assumption $d_\Sk(\sh_0^\delta, \sh_0)\to 0$ and that $M^\da_{\bw,x}(\rho^\da(\sz^i,0))$ 
is a continuity point for $\sh_0$, implies that also the third term in~\eqref{e:FinalB} can be made smaller than $\eps/3$. 
We conclude that $\gamma(\lambda)\vee d^{[-R,R]}_\lambda(\sh_\bc(T,\cdot), \sh_{\bd}^\delta(T,\cdot))<\eps$ as required to complete the proof. 
\end{proof}

\begin{appendix}

\section{The smoothened Poisson process}\label{app:SmoothPoisson}

In this appendix, we derive bounds on the Orlicz norm of the increment of a smoothened version of the Poisson process. 
Let $a>0$ and $\psi_a$ be a smooth non negative function supported in $[-a,0]$ or $[0,a]$ such that 
$\int \psi_a(x)\dd x=1$. For $\lambda>0$ let $\mu_\lambda$ be a Poisson random measure on $\R_+$ with 
intensity measure $\lambda \ell$, where $\ell$ is the Lebesgue measure on $\R_+$, and define the 
{\it rescaled compensated smoothened Poisson process} $P^{a}_\lambda$ and 
the rescaled compensated Poisson process $P_\lambda$ respectively by 
\begin{equation}\label{def:SmoothRCPP} 
P^{a}_\lambda = \frac{1}{\sqrt{\lambda}}\Big(\int_0^t\psi_a\ast\mu_\lambda(s)\dd s-\lambda t\Big) \;,\qquad P_\lambda(t) \eqdef \frac{1}{\sqrt{\lambda}}\big(\mu_\lambda([0,t])-\lambda t\big)\;.
\end{equation} 
Then, the following lemma holds.

\begin{lemma}\label{l:SmoothPoisson}
In the setting above, let $P^{a}_\lambda$ be the rescaled compensated smoothened Poisson process on $[0,T]$ 
defined in~\eqref{def:SmoothRCPP}. 
%
Let $p>1$ and assume $a\lambda^p=1$. Then, there exists a positive constant $C$ depending only on $T$ 
such that for every $0\leq s<t\leq T$, we have 
\begin{equation}\label{b:SmoothRCPP}
\| P^{a}_\lambda(t)-P^{a}_\lambda(s)\|_{\varphi_1}\leq C (t-s)^{\frac{1}{2p}}
\end{equation}
where the norm appearing on the left hand side is the Orlicz norm defined in~\eqref{def:Orlicz} with 
$\varphi_1(x)\eqdef e^x-1$. 
\end{lemma}

\begin{proof}
We prove the result for $\psi_a$ supported in $[-a,0]$. 
Also, writing $P(t) = P_1(t)$, we have $\E e^{P(t)/c} = \exp(t(e^{1/c}-1-1/c))$, so that 
\begin{equ}[e:boundPt]
\|P(t)\|_{\phi_1} \lesssim 1 + \sqrt t\;.
\end{equ}
Fix $0\leq s<t\leq T$ and consider first the case $t-s\geq a$. Notice that we have 
\begin{equs}
P^{a}_\lambda(t)-P^{a}_\lambda(s)&=\frac{1}{\sqrt{\lambda}}\Big(\int_\R \left(\psi_a(t-u)-\psi_a(s-u)\right)\mu_\lambda([0,u])\dd u -\lambda (t-s)\Big)\\
&\leq P_\lambda(t+a)-P_\lambda(s) +\sqrt{\lambda}a \eqlaw {1\over \sqrt \lambda}P\big(\lambda(a+t-s)\big) + \sqrt \lambda a\;.
\end{equs}
It follows from \eqref{e:boundPt} and the triangle inequality that 
\begin{equ}
\| P^{a}_\lambda(t)-P^{a}_\lambda(s)\|_{\varphi_1}
\lesssim \sqrt{t-s+a} + {1\over \sqrt \lambda} + \sqrt \lambda a
\lesssim \sqrt{t-s} + {1\over \sqrt \lambda}\lesssim (t-s)^{1\over 2p}\;.
\end{equ}

For $t-s<a$, we bound the increment of $P^a_\lambda$ by 
\begin{equs}
P^{a}_\lambda(t)- P^{a}_\lambda(s)&=\frac{1}{\sqrt{\lambda}}\Big(\int_s^t\int_u^{u+a} \psi_a(u-r)  \mu_\lambda(\dd r)\dd u-\lambda(t-s)\Big)\\
&\leq \frac{1}{\sqrt{\lambda}}\Big(\frac{1}{a}\int_s^t \mu_\lambda([u,u+a])\dd u-\lambda(t-s)\Big)\eqlaw\frac{t-s}{\sqrt \lambda a} P(\lambda a)\;.
\end{equs}
Since $\lambda a \le 1$, it follows that in this case 
$\| P^{a}_\lambda(t)-P^{a}_\lambda(s)\|_{\varphi_1} \lesssim \frac{t-s}{\sqrt \lambda a} \lesssim (t-s)^{1\over 2p}$ as claimed.
\end{proof}

\section{On the cardinality of the coalescing point set of the Brownian Web}

The aim of this appendix is to derive a uniform bound on the cardinality of the coalescing point set of 
the Brownian Web Tree $\zeta^\da_\bw$ of Theorem~\ref{thm:BW}. 
Let $R,\,\eps>0$ and $t\in\R$, set 
\begin{align}
\Xi_R(t,t-\eps)&\eqdef\{\rho^\da(\sz,t-\eps)\colon M^\da_{\bw}(\sz)\in [t,\infty)\times[-R,R]\}\label{e:CPS}\\
\eta_R(t,\eps)&\eqdef \#\Xi_R(t,t-\eps)\label{e:CPScard}
\end{align}
where, for a set $A$, $\#A$ denotes its cardinality. Then, the following lemma holds. 

\begin{lemma}\label{l:CPScard}
Almost surely, for any $\varsigma>1/2$ and all $r,\,R>0$ there exists a constant $C=C(r,R)$ such that 
\begin{equ}[e:CPSbound]
\eta_R(t,\eps)\leq C\eps^{-\varsigma}\,,
\end{equ}
for all $t\in[-r,r]$ and $\eps\in(0,1]$
\end{lemma}
\begin{proof}
Notice first that, by a simple duality argument, $\eta_R(t,\eps)$ is equidistributed with 
$\hat\eta(t,t+\eps;-R,R)$ of~\cite[Definition 2.1]{FINR}. According to~\cite[Lemma C.2]{GSW}, 
for any $R,\,t,\,\eps$, $\eta_R(t,\eps)$ is a negatively correlated point process with intensity measure 
$\tfrac{2R}{\sqrt{t}}\lambda$, where $\lambda$ is the Lebesgue measure on $\R$. 
In particular,~\cite[Lemma C.5]{GSW} implies that, for any $p>1$
\begin{equ}[e:MB]
\E[\eta_R(t,\eps)^p]\lesssim_p \left(\frac{2R}{\sqrt{\eps}}\right)^{p}\,,
\end{equ}
where the hidden constant depends only on $p$. Moreover, the random variables $\eta_R(t,\eps)$ 
are monotone in the following sense, for any $R,\,t,\,\eps$ we have 
\begin{equ}[e:Mono]
\eta_R(t,\eps)\leq \eta_{R'}(t',\eps')
\end{equ}
for all $R'\geq R$, $t'\in(t-\eps,t]$ and $\eps'\in(0,\eps-(t-t')]$. 

Let $R,\,r>0$, for $k=0,\dots, 2\lceil r\rceil 2^m$ set $t_{k,m}\eqdef -r+k 2^{-m}$ 
and consider the event 
\begin{equ}
E_m\eqdef\{\forall k=0,\dots,2\lceil r\rceil 2^m,\,n\in\N,\quad \eta_R(t_{k,m},2^{-n})\leq 2^{m+1}R 2^{n\varsigma}\}\,.
\end{equ}
By Markov's inequality and~\eqref{e:MB}, we have 
\begin{equs}
\P(E_m^c)&\leq \sum_{k=0}^{2\lceil r\rceil 2^m}\sum_n \P(\eta_R(t_{k,m},2^{-n})\leq 2^{m+1}R 2^{n\varsigma})\lesssim_p \lceil r\rceil 2^{m(1-p)} \sum_n 2^{-np(\varsigma-1/2)}
\end{equs}
which, for $\varsigma>1/2$, is finite and $O_r(2^{m(1-p)})$. Therefore, upon choosing $p>1$ and 
applying Borel-Cantelli and~\eqref{e:Mono}, the statement follows at once. 
\end{proof}

\section{Exit law of Brownian motion from the Weyl chamber}\label{app:Density}

For $n \ge 2$, we define the Weyl chamber $W_n$ as 
\begin{equ}
W_n = \{x \in \R^n \,:\, x_1<\dots<x_n\}\;.
\end{equ}
Let $(B_t^x)_{t\ge 0}$ be a standard $n$-dimensional Brownian motion and,
given a sufficiently regular domain $W \subset \R^n$, let $\tau_W=\inf\{t>0\,:\, B^x_t\in \partial W\}$
and
\begin{equ}
P_t^W(x,y)\,\dd y \eqdef \Prob(\tau_W > t, B_t^x \in \dd y)\;,\qquad x,y \in W\;.
\end{equ}
We then have the following result.

\begin{theorem}\label{thm:Density}
Let $(B_t^x)_{t\ge 0}$ with $x \in W_n$ be as above and let $\tau=\tau_{W_n}$.
Then, 
\begin{equation}
\Prob\left(\tau\in \dd t, B_{\tau}^x\in \dd y\right) = \partial_{n_y} P_t^{W_n}(x,y)\,\sigma_{W_n}(\dd y) \,\dd t \eqdef \nu_x(\dd t,\dd y)\;,
\end{equation}
where $\partial_{n_y}$ is the derivative in the inward normal direction at 
$y\in\partial W_n$ and $\sigma_{W_n}$ is the surface measure on $\partial W_n$. 
\end{theorem}

\begin{proof}
For smooth cones, the claim was shown for example in \cite[Thm~1.3]{BDB}, so it remains to perform an approximation argument. 
Choose a sequence of smooth cones $W_n^{(\eps)}$ such that, for every $\eps > 0$, one has
$W_n^{(\eps)} \subset W_n$ and furthermore $W_n^{(\eps)} \cap C_\eps^c = W_n \cap C_\eps^c$,
where $C_\eps$ denotes those ``corner'' configurations where at least two distinct 
pairs of points are at distance less than $\eps$ from each other.
Writing $\tau_\eps = \tau_{W_n^{(\eps)}}$, it follows immediately from these two 
properties that 
$P_t^{W_n^{(\eps)}}(x,y) \le P_t^{W_n}(x,y)$ for all $t \ge 0$ and $x,y \in W_n^{(\eps)}$, so that in particular,
~\cite[Thm~1.3]{BDB} implies
\begin{equ}
\nu_x^{(\eps)}(\dd t,\dd y) \eqdef \Prob\left(\tau_\eps \in \dd t, B_{\tau_\eps}^x\in \dd y\right)= \partial_{n_y} P_t^{W_n^{(\eps)}}(x,y)\,\sigma_{W_n}(\dd y) \,\dd t \le \nu_x(\dd t,\dd y)\;,
\end{equ}
for all $y \in \d W_n \cap C_\eps^c$. Here, we also used the fact that $P_t^{W_n^{(\eps)}}$ and $P_t^{W_n}$ both
vanish on $\d W_n \cap C_\eps^c$. We also note that $\nu_x$ is a probability measure, as can be seen by combining the
divergence theorem (on the space-time domain $\R_+\times W_n$) with the fact that $P_t^{W_n}$ solves the heat equation
on $W_n$ with Dirichlet boundary conditions and initial condition $\delta_x$. 

On the other hand, writing $\mu_x^{(\eps)}$ for the (positive) measure such that 
\begin{equ}
\Prob\left(\tau\in \dd t, B_{\tau}^x\in \dd y\right) = \nu_x^{(\eps)}(\dd t,\dd y \cap C_\eps^c) + \mu_x^{(\eps)}(\dd t, \dd y)\;,
\end{equ}
one has the bound
\begin{equ}
c_\eps \eqdef \mu_x^{(\eps)}(\R_+ \times \d W_n) \le \P(\hat \tau_\eps \le \tau)\;,\qquad \hat \tau_\eps = \inf\{t \,:\, B_t^x \in C_\eps\}\;. 
\end{equ}
Since $\tau < \infty$ almost surely and Brownian motion does not hit subspaces of codimension~$2$, we have
$\lim_{\eps \to 0} c_\eps = 0$.
For any two measurable sets $I \subset \R_+$ and $A \subset \d W_n$ such that $A \cap C_\delta = \emptyset$ for some 
$\delta > 0$, we then have
\begin{equ}
\Prob\left(\tau\in I, B_{\tau}^x\in A\right) \le \nu_x^{(\eps)}(I,A \cap C_\eps^c) + c_\eps 
\le \nu_x(I,A) + c_\eps\;,
\end{equ}
for all $\eps \le \delta$, so that
\begin{equ}[e:almostGood]
\Prob\left(\tau\in \dd t, B_{\tau}^x\in \dd y\right) \le \nu_x(\dd t, \dd y) + \hat \mu(\dd t, \dd y)\;,
\end{equ}
where $\hat \mu$ is supported on $\R_+ \times \bigcap_{\eps > 0}C_\eps$. As before, one must have 
$\hat \mu = 0$ since Brownian motion does not hit subspaces of codimension~$2$, so that the desired identity 
follows from the fact that both $\nu_x$ and the left-hand side of \eqref{e:almostGood} are probability measures.
\end{proof}

\section{Trimming and path correspondence}\label{app:Trim}

In this appendix we introduce some further tools in the context of spatial trees, which plays a major role 
in the proof of Theorem~\ref{thm:convTime}. 
Let $(\ST,\ast,d)$ be a pointed locally compact complete $\R$-tree and fix $\eta>0$.
We define the {\it $\eta$-trimming} of $\ST$ as
\begin{equation}\label{def:trim}
 R_\eta(\ST)\eqdef\{\sz\in\ST\,:\,\exists\text{ $\sw\in\ST$ such that $\sz\in\llb\ast,\sw\rrb$ and $d(\sz,\sw)\geq \eta$} \}\cup\{\ast\}\,.
\end{equation} 
(The explicit inclusion of $\ast$ is only there to guarantee that $ R_\eta(\ST)$ is non-empty if $\ST$ is
of diameter less than $\eta$.)
$ R_\eta(\ST)$ is clearly closed in $\ST$ and furthermore, it is 
a locally finite $\R$-tree. 
With a slight abuse of notation, we denote again by $R_\eta$ the trimming of a spatial $\R$-tree, i.e. the map
$R_\eta\colon\T^\alpha_\Sp\to\T^\alpha_\Sp$ defined on $\zeta=(\ST,\ast,d,M)\in\T^\alpha_\Sp$ as
$R_\eta(\zeta)= ( R_\eta(\ST),\ast,d,M)$. 
In the following lemma we summarise further properties of the trimming map. 

\begin{lemma}\label{l:Trim}
Let $\alpha\in(0,1)$. For all $\eta>0$, $R_\eta$ is continuous on $\T^\alpha_\Sp$. 
Moreover, if $\{(\ST_a,\ast_a,d_a)\}_{a\in A}$, $A$ being an index set, is a family of compact pointed $\R$-trees then 
\begin{enumerate}[noitemsep]
\item the Hausdorff distance between $ R_\eta(\ST_a)$ and $\ST_a$ is bounded above by $\eta$,  
\item the number of endpoints of $ R_\eta(\ST_a)$ is bounded above by $(c\eta)^{-1}\ell_a( R_{c\eta}(\ST_a))<\infty$, 
for any $c\in(0,1)$,
\item the family is relatively compact if and only if $\sup_{a\in A} \ell_a( R_{\eta}(\ST_a))<\infty$ for all $\eta>0$.
\end{enumerate}
\end{lemma}
\begin{proof}
The continuity of the trimming map is an easy consequence of the definition of the metric 
$\Delta_\Sp$ in~\eqref{e:Metric},~\cite[Lemma 2.6(ii)]{EPW} 
and~\cite[Lemma 2.17]{CHbwt}. Points 1., 2. and 3. were respectively shown in~\cite[Lemma 2.6(iv)]{EPW}, 
the proof of~\cite[Lemma 2.7]{EW06} and~\cite[Lemma 2.6]{EW06}. 
\end{proof}

Let $\alpha\in(0,1)$ and $\zeta_1,\,\zeta_2\in\hat\Ch^\alpha_\Sp$ 
(see Definition~\ref{def:CharTree}) be such that
\begin{equ}[e:ast]
M_{1,t}(\ast_1)=0=M_{2,t}(\ast_2)\,.
\end{equ}
For $j=1,2$, denote by $\rho_j$ the radial map of $\zeta_j$. For $r,\eta>0$, set 
\begin{equ}[def:Trim+]
\ST_j^{(r),\,\eta}\eqdef R_\eta \big(\ST_j^{(r)}\cup\llb\ast_j,\rho_j(\ast_j,r+\eta)\rrb\big)
\end{equ}
and $\zeta_j^{(r),\,\eta}\eqdef (\ST_j^{(r),\,\eta},\ast_j,d_j,M_j)$. 
Assume there exists a correspondence $\CC$ between 
$\zeta^{(r),\,\eta}_1$ and $\zeta^{(r),\,\eta}_2$ for which 
\begin{equ}[e:InitialDist]
\Delta^{\com, \CC}_\Sp(\zeta^{(r),\,\eta}_1,\zeta^{(r),\,\eta}_2 )<\eps\;,
\end{equ} 
for some $\eps>0$. 
Let $N$ be the number of endpoints of 
$\zeta^{(r),\,\eta}_1$, which is finite by 
Lemma~\ref{l:Trim} points 2.\ and 3.
We now number the endpoints of $\ST^{(r),\,\eta}_1$ and denote them by $\{\tilde\sv^1_i\,:\,i=0,\dots, N_\eta\}$, 
where $\tilde\sv^1_0\eqdef\sv^1_0= \ast_1$. 
Let $\sv_0^2\eqdef\ast_2$ and for every $i=1,\dots, N_\eta$, 
let $\tilde\sv_i^2\in \ST_2^{(r),\,\eta}$ be such that $(\tilde\sv^1_i,\tilde\sv_i^2)\in\CC$. 
If $M_{1,t}(\tilde\sv^1_i)\geq M_{2,t}(\tilde\sv_i^2)$, 
set $\sv_i^1\eqdef\tilde\sv^1_i$ and $\sv_i^2\eqdef \rho_2(\tilde\sv_i^2, M_{1,t}(\sv^1_i))$, 
otherwise set $\sv^2_i\eqdef\tilde\sv^2_i$ and 
$\sv^1_i\eqdef \rho_1(\tilde\sv_i^1, M_{2,t}(\sv^2_i))$. 

Setting $\ast_j^{(r)} = \rho_j(\ast_j,r)$, we define the subtree $T_j \subset \ST^{(r),\,\eta}_j$ by
\begin{equ}[e:PathTree]
T_j\eqdef \bigcup_{i \le N_\eta} \llb \sv_i^j, \ast_j^{(r)}\rrb \,.
\end{equ}
We also write $Z_j$ for the corresponding spatial tree 
\begin{equ}[e:PathSpTrees]
Z_j\eqdef (T_j,\ast_j,d_j, M_j)\,.
\end{equ}
Finally, we define the {\it path correspondence} between $T_1$ and $T_2$ by
\begin{equ}[e:PathCorr]
\CC_\rp\eqdef \bigcup_{i=0}^{N_\eta}\{(\rho_1(\sv_i^1,t),\rho_2(\sv^2_i,t)) \colon M_{1,t}(\sv_i^1)\leq t\leq r\}\,.
\end{equ}

\begin{lemma}\label{l:PathTrees}
Let $\alpha\in(0,1)$ and $\zeta_1,\,\zeta_2\in\hat\Ch^\alpha_\Sp$, $r,\,\eta, \,\eps>0$ 
be such that~\eqref{e:ast} and~\eqref{e:InitialDist} hold.
Then, 
\begin{equ}[b:PathTrees]
\dis\CC_\rp+\sup_{(\sz_1,\sz_2)\in\CC_\rp}\|M_1(\sz_1)-M_2(\sz_2)\|\lesssim 4N_\eta\eps+\|M_1\|_\alpha^{(r)}\eps^\alpha\,.
\end{equ}
Moreover, the Hausdorff distance between $T_1$ and $\ST^{(r),\,\eta}_1$ is bounded above by $\eps$, 
while that between $T_2$ and $\ST^{(r),\,\eta}_2$ is bounded by $5N_\eta\eps$ and 
the following inclusions hold 
\begin{equ}[e:Incl]
R_{\eta+\eps}(\ST^{(r)}_1)\subset T_1\subset \ST^{(r),\,\eta}_1\,,\qquad R_{\eta+5N_\eta\eps}(\ST^{(r)}_2)\subset T_2\subset \ST_2^{(r),\,\eta}\,.
\end{equ}
\end{lemma} 
\begin{proof}
The statement follows by applying iteratively~\cite[Lemma 2.28]{CHbwt}. 
Indeed, for $m\leq N_\eta$, let $\tilde\CC_{m}\eqdef\CC\cup\CC_{m-1}$, where $\CC_{m-1}$ is defined as 
the right hand side of~\eqref{e:PathCorr} but the union runs from $0$ to $m-1$ 
(so that in particular $\CC_{N_\eta}=\CC_\rp$). 
Then, for all $m$, $\tilde\CC_{m}=C_{\tilde\CC_{m-1}}$, 
the right hand side being defined according to~\cite[Eq. (2.29)]{CHbwt}. Hence, we immediately see that
$\dis\CC_\rp\leq \dis \tilde \CC_{N_\eta}\lesssim 4N_\eta\eps$ and the bound 
on the evaluation map can be similarly shown. 

For the last part of the statement, notice that the Hausdorff distance between $T_1$ and $\ST^{(r),\eta}_1$ 
is bounded by $\eps$ by construction, 
while, arguing as in the proof of Lemma~\ref{l:Coupling}, it is immediate to show that 
the Hausdorff distance between $T_2$ and $\ST^{(r),\eta}_2$ is controlled by $4N_\eta\eps+\eps$. 
These bounds together with the definition of the trimming map, guarantee that, for any $a>0$, 
all the endpoints of $R_{\eta+\eps}(\ST^{(r)}_1)$ and 
$R_{\eta+5N_\eta\eps+a}(\ST_2^{(r)})$ must belong to $T_1$ and $T_2$ respectively, which in turn implies  
\begin{equ}
R_{\eta+\eps+a}(\ST^{(r)}_1)\subset T_1\,,\qquad R_{\eta+5N_\eta\eps+a}(\ST_2^{(r)})\subset T_2\,.
\end{equ} 
By letting $a\to 0$ using Lemma~\ref{l:Trim} point 1., the conclusion follows. 
\end{proof}

\end{appendix}

\bibliographystyle{Martin}

\bibliography{refs}

\end{document}